\documentclass[11pt, reqno]{amsart}

\usepackage{amsmath,amssymb}
\usepackage{iftex}
\usepackage{tabularray}
\usepackage{subcaption}
\usepackage{amsthm,amsmath,amssymb,amsfonts}
\usepackage{natbib}
\usepackage{txfonts}
\usepackage{booktabs}
\usepackage{graphicx}
\usepackage{graphics}
\usepackage{bm}
\usepackage[colorlinks=true, allcolors=blue]{hyperref}
\usepackage{url}
\usepackage{enumitem}

\hypersetup{colorlinks=true,linkcolor=blue,citecolor=blue,urlcolor=blue}

\newtheorem{theorem}{Theorem}[section]
\newtheorem{proposition}[theorem]{Proposition}
\newtheorem{lemma}[theorem]{Lemma}
\newtheorem{corollary}[theorem]{Corollary}
\theoremstyle{definition}

\theoremstyle{remark}
\newtheorem{remark}[theorem]{Remark}

\newcommand{\R}{\mathbb{R}}
\newcommand{\E}{\mathbb{E}}
\newcommand{\Pbb}{\mathbb{P}}
\newcommand{\Var}{\mathrm{Var}}
\newcommand{\1}{\mathbf{1}}
\newcommand{\ph}{\varphi}
\newcommand{\Ph}{\Phi}
\newcommand{\barPh}{\overline{\Phi}}

\newcommand{\dto}{\xrightarrow[]{d}}

\title[Fixed-level calibration of the Cauchy combination test]{Fixed-level calibration of the Cauchy combination test}
\author{Hirofumi Ota}
\address{Komaba Institute for Science, Graduate School of Arts and Sciences, The University of Tokyo}
\email{hirofumi-ota@g.ecc.u-tokyo.ac.jp}
\date{First version: Mar. 26, 2026. This version: \today}

\begin{document}

\begin{abstract}
The Cauchy combination test (CCT) is widely used because it yields a closed-form combined $p$-value and is known to be asymptotically valid as the nominal level $\alpha\downarrow0$ under broad dependence structures. We study a different asymptotic question: whether the usual Cauchy cutoff remains accurate at an ordinary fixed level when the number $K$ of combined $p$-values grows under dependence. Under a canonical one-factor equicorrelated Gaussian copula model, we show that the raw CCT is generally not asymptotically exact at fixed $\alpha$. With fixed positive correlation, the statistic converges to a random latent-factor limit, so there is no universal fixed-level reference law. When the common correlation $\rho_K$ weakens with $K$, fixed-level behaviour is governed by the boundary-layer scale $s_K=\sqrt{\rho_K}(\log K)^{3/2}$, and the raw CCT is asymptotically exact if and only if $\rho_K(\log K)^3\to0$. Because the size distortion arises entirely from the reference law and not from the statistic, it can be corrected without modifying the test statistic itself. We propose the boundary-layer calibrated CCT (BL-CCT), which replaces the standard Cauchy reference by a one-parameter Gaussian-smoothed Cauchy family. Unlike recent variants that modify the test statistic, BL-CCT leaves the statistic unchanged and corrects only the reference law. BL-CCT is asymptotically exact under the weaker condition $\rho_K\log K\to0$ and provides a useful finite-$K$ approximation on bounded boundary layers. We also conduct several power analyses: although BL-CCT only raises the cutoff, it incurs no first-order power loss relative to the raw CCT on the exactness scale, under local dense, sparse, and dense Gaussian alternatives. Numerical experiments support the calibration theory.
\end{abstract}

\maketitle

\section{Introduction}

\subsection{Overview}
Combining many $p$-values into a single global test is a classical problem in mathematical statistics. Among analytic combination rules, the Cauchy combination test (CCT) of Liu and Xie \cite{LiuXie2020} is attractive because it yields a closed-form combined $p$-value and remains numerically stable at very small significance levels. Its standard asymptotic justification, however, is a vanishing-level result: under broad conditions on the null dependence structure, the CCT statistic is asymptotically standard Cauchy as the nominal level $\alpha\downarrow0$. In practice the test is routinely applied at ordinary levels such as $0.05$ or $0.01$ with a large number $K$ of component tests; see, for example, \cite{LiSTAAR2020,ZhouEtAl2022SAIGEGenePlus}. This creates a gap between theory and practice: does the usual Cauchy cutoff remain asymptotically exact when $\alpha$ is fixed and $K\to\infty$ under dependence? Recent numerical work indicates that the answer can be negative \cite{LongLiZhangLi2023,OuyangEtAl2024,GuiJiangWang2025,AlsulamiLiverani2025}.

We study this question under the one-factor equicorrelated Gaussian copula model, in which all dependence among the $p$-values is generated by a single shared Gaussian factor. It is analytically tractable: conditioning on the factor renders the component tests independent. Although the CCT is known for its asymptotic validity under broad dependence, that validity is a vanishing-level guarantee ($\alpha\downarrow0$); restricting to this simplest nontrivial form of dependence is deliberate, since fixed-level exactness already fails here, and its tractability is what makes the mechanism, the sharp transition scale, and the calibration explicit. The key finding is that, although the conditional Cauchy limit familiar from vanishing-level theory still holds, a deterministic centring term induced by the latent factor grows with $K$ and $\rho_K$ and displaces the null distribution of the statistic. This centring is the source of the fixed-level size distortion. It is of lower order than the statistic itself and therefore vanishes under the rescaling implicit in $\alpha\downarrow0$, which is why vanishing-level analysis does not detect it.

\subsection{Main contributions}
Our first result concerns the raw CCT developed by \cite{LiuXie2020}. When the common correlation $\rho$ is fixed and positive, the CCT statistic converges to a random latent-factor limit, so there is no universal fixed-level reference law. When instead $\rho_K\downarrow0$ as $K\to\infty$, the transition between exact and inexact behaviour is governed by the boundary-layer scale
\[
  s_K=\sqrt{\rho_K}(\log K)^{3/2}.
\]
When $s_K$ is bounded, the conditional $1$-stable fluctuation reduces to a standard Cauchy law, but a deterministic centring term of order $s_K$ persists in the null distribution. The raw CCT is therefore asymptotically exact at fixed $\alpha$ \textit{if and only if} $\rho_K(\log K)^3\to0$.

A key consequence is that the size distortion originates from a mismatch between the standard Cauchy reference law and the actual null distribution of $T_K$, not from the heavy-tailed fluctuation of the statistic. The statistic $T_K$ itself does not need to be modified. We exploit this by introducing the boundary-layer calibrated CCT (BL-CCT), which replaces only the standard Cauchy cutoff by a one-parameter Gaussian-smoothed Cauchy family indexed by $s_K$. Under this correction, exactness holds under the weaker condition $\rho_K\log K\to0$, and on bounded boundary layers the reference family already provides a useful finite-$K$ approximation. On the broader scale $c_K=\rho_K\log K$, the remaining size distortion is explicit and strictly conservative at conventional levels such as $0.05$ and $0.1$.

A final contribution asks what this calibration costs in power. Because BL-CCT only raises the cutoff, it rejects less readily than the raw CCT and could in principle detect fewer signals. We show that it essentially does not. Comparing the two tests under three standard classes of alternatives, local dense, sparse, and dense Gaussian, we find that, whenever the raw CCT is exact, BL-CCT attains the same first-order power; the size correction is therefore free to first order.

\subsection{Related work}

Existing asymptotic theory for the CCT and its extensions~\cite{LiuXie2020,LongLiZhangLi2023,FangEtAl2023,GuiJiangWang2025} operates in the regime $\alpha\downarrow0$ with $K$ fixed or growing. The present paper instead fixes $\alpha$ at a conventional level and lets $K\to\infty$, which reveals a latent-factor drift mechanism that vanishing-level analysis does not detect. Concurrently and independently, \cite{GuiMaoWangWang2025MRV} and \cite{chakraborty2025universal} used multivariate regular variation to show that, at vanishing level, heavy-tailed combination tests beat Bonferroni only under tail-dependence and match it under tail-independence. Our fixed-level analysis is complementary: under the tail-independent Gaussian copula the issue is not power against Bonferroni but the first-order-negligible cost of the size calibration, and, consistent with their verdict, the CCT's sparse detection boundary coincides with the Bonferroni/max boundary.

Several recent studies address the finite-$K$ size distortion of the CCT by modifying the test statistic: a positive adjustment to the combination weights~\cite{OuyangEtAl2024}, a right-tail-weighted combination encompassing Cauchy-type statistics~\cite{LiuMengPillai2025}, a truncated variant~\cite{ChenXuGao2025TCCT}, and a stepwise procedure~\cite{Bouamara2025StepC}. Calibration-based viewpoints are also well established in the broader $p$-value-combination literature, including likelihood-ratio guidance for choosing a combiner~\cite{HeardRubinDelanchy2018} and admissible merging functions under arbitrary dependence~\cite{VovkWangWang2022}. BL-CCT belongs to this calibration family: it keeps the statistic unchanged and corrects only the reference law. The contribution is not the calibration paradigm itself, but the explicit one-parameter analytic reference family and a sharp fixed-level transition scale under the one-factor benchmark: $\rho_K(\log K)^3\to0$ for the raw CCT and $\rho_K\log K\to0$ for BL-CCT (necessary and sufficient at conventional levels).

Beyond the genomic applications mentioned above~\cite{LiSTAAR2020,ZhouEtAl2022SAIGEGenePlus}, the CCT has been adopted for microbiome studies~\cite{YuEtAl2025} and high-dimensional regression~\cite{ZhaoSongMa2026}. For general background on $p$-value combination, see~\cite{xie2011confidence,Wilson2019HMP,VovkWangWang2022,ChenLiuTanWang2023}. The connection between Cauchy averages and $1$-stable laws, identified by Pillai and Meng~\cite{PillaiMeng2016}, underpins the stable-limit analysis developed throughout this paper.

\subsection{Notation}
We write $\ph$ and $\Ph$ for the standard normal density and distribution function, $\barPh=1-\Phi$ for the upper tail, and denote convergence in probability and in distribution by $\xrightarrow{p}$ and $\xrightarrow{d}$. We write $\mathsf{C}(a,b)$ for the Cauchy distribution with location $a$ and scale $b>0$, and let $\mathsf{C}\sim\mathsf{C}(0,1)$ denote a standard Cauchy variable. Unless stated otherwise, fixed-level exactness statements concern $\alpha\in(0,1/2)$; the calibrated reference family itself is defined for all $\alpha\in(0,1)$.

\subsection{Organization}

The remainder of the paper is organized as follows. Section~\ref{sec:setup} introduces the Gaussian copula model and formulates the fixed-level problem. Section~\ref{sec:raw cct} analyses the raw CCT and establishes the sharp exactness threshold. Section~\ref{sec:calibration} develops BL-CCT and its broader-scale behaviour. Section~\ref{sec:power} analyses the power of the calibrated test under local dense, sparse, and dense Gaussian alternatives. Section~\ref{sec:numerics} presents numerical illustrations, and Section~\ref{sec:discussion} concludes with discussion. Proofs and auxiliary lemmas are collected in the appendices.

\section{Setup and the standard Gaussian benchmark}\label{sec:setup}

This section defines the one-factor equicorrelated Gaussian copula benchmark used throughout the paper and formulates the fixed-level calibration problem.

For $\rho\in[0,1)$, let
\[
  Z=(Z_1,\dots,Z_K)^\top \sim N(0,\Sigma_\rho),
\]
where $\Sigma_\rho$ is the equicorrelated covariance matrix given by
\[
  (\Sigma_\rho)_{ii}=1,
  \qquad
  (\Sigma_\rho)_{ij}=\rho
  \quad (i\neq j).
\]
Since $\Sigma_\rho=(1-\rho)I_K+\rho \mathbf 1\mathbf 1^\top$, this model admits the one-factor representation
\begin{equation}\label{eq:main-onefactor}
  Z_i=\sqrt{\rho}V+\sqrt{1-\rho}\varepsilon_i,
  \qquad i=1,\dots,K,
\end{equation}
where $V\sim N(0,1)$, the $\varepsilon_i$ are i.i.d.\ $N(0,1)$, and $V$ is independent of $(\varepsilon_1,\dots,\varepsilon_K)$. Thus the null dependence is generated by a single latent Gaussian factor. Conditional on $V$, the coordinates are independent and satisfy
\begin{equation}\label{eq:main-condlaw}
  Z_i \mid V \sim N(\sqrt{\rho}V,1-\rho),
  \qquad i=1,\dots,K.
\end{equation}

We work with one-sided Gaussian $p$-values
\[
  p_i=\barPh(Z_i)=1-\Phi(Z_i),
\]
and with the standard Cauchy quantile transform
\begin{equation}\label{eq:main-score}
  f(z)=\tan\{\pi(\Phi(z)-1/2)\}.
\end{equation}
The equal-weight Cauchy combination statistic is given by the average
\begin{equation}\label{eq:main-stat}
  T_K=\frac1K\sum_{i=1}^K f(Z_i).
\end{equation}
At level $\alpha\in(0,1/2)$, the usual Cauchy combination test rejects when
\begin{equation}\label{eq:main-reject}
  T_K>t_\alpha,
  \qquad
  t_\alpha:=\cot(\pi\alpha)
\end{equation}
holds.

Conditioning on $V$ renders the coordinates independent, so the rejection probability can be analysed by first conditioning on $V$ and then integrating over the latent factor. The independent case $\rho=0$ serves as the reference point, and the transition of interest arises when $\rho>0$, in triangular regimes with $\rho=\rho_K\downarrow0$ as $K\to\infty$. We use one-sided Gaussian $p$-values because they retain the effect of the common Gaussian factor most directly.

\begin{remark}[One-sided versus two-sided $p$-values]
The main analysis is for one-sided Gaussian $p$-values. If instead $p_i=2\barPh(|Z_i|)$, the transformation is even in $Z_i$, so the conditional mean of the statistic is an even function of the latent factor and the linear drift analysed below cancels exactly. The leading centring should then be quadratic in the factor and appear on a different scale. A sharp two-sided analogue therefore requires a separate expansion and is not claimed here.
\end{remark}

\section{Raw CCT at fixed levels}\label{sec:raw cct}

This section analyses the fixed-level behaviour of the raw CCT. We first consider the case of fixed positive correlation and then the triangular regime $\rho=\rho_K\downarrow0$.

\subsection{Fixed positive correlation: a latent-factor limit}

When $\rho>0$ is fixed, the loading $\sqrt{\rho}$ on the latent factor $V$ remains constant as $K\to\infty$. Conditional on $V$, the Cauchy-transformed scores are i.i.d., so by the law of large numbers $T_K$ converges to a limit that depends on $V$.

For $\rho>0$, define the conditional mean
\begin{equation}\label{eq:main-mu}
  \mu_\rho(v):=\E\left[f(\sqrt{\rho} v+\sqrt{1-\rho} \varepsilon)\right],
  \qquad \varepsilon\sim N(0,1).
\end{equation}

\begin{theorem}[Random latent-factor limit at fixed positive correlation]\label{thm:randomlimit}
Fix $\rho\in(0,1)$. Then we have
\begin{equation}\label{eq:main-fixed-limit}
  T_K \xrightarrow[K\to\infty]{a.s.} \mu_\rho(V).
\end{equation}
Moreover, we have
\begin{equation}\label{eq:main-fixed-size}
  \Pbb(T_K>t_\alpha)\to \Pbb\left(\mu_\rho(V)>t_\alpha\right).
\end{equation}
\end{theorem}

Because $\rho>0$ is fixed, the law of large numbers yields a conditional limit $\mu_\rho(v)$ that varies with the realization of $V$. The unconditional limit $\mu_\rho(V)$ is therefore a nondegenerate random variable, and the fixed-level rejection probability is determined by the distribution of $\mu_\rho(V)$ rather than by a universal Cauchy reference.

To connect this regime with independence, we next examine $\mu_\rho(v)$ as $\rho\downarrow0$. The two iterated limits $\lim_{\rho\downarrow0}\lim_{K}$ and $\lim_{K}\lim_{\rho\downarrow0}$ differ, so the triangular regime $\rho_K\downarrow0$ with $K\to\infty$ must be analysed directly.

\begin{proposition}\label{prop:noncommuting-main}
For each fixed $v\in\R$, we have
\begin{equation}\label{eq:main-mu-smallrho}
  \rho \mu_\rho(v)\to \sqrt{2/\pi} v e^{v^2/2}
  \qquad (\rho\downarrow0).
\end{equation}
In particular, $\Pbb\left(\mu_\rho(V)>t_\alpha\right)\to \frac12$, so
\begin{equation}\label{eq:main-double-limit}
  \lim_{\rho\downarrow0}\lim_{K\to\infty}\Pbb(T_K>t_\alpha)=\frac12,
  \qquad
  \lim_{K\to\infty}\lim_{\rho\downarrow0}\Pbb(T_K>t_\alpha)=\alpha.
\end{equation}
\end{proposition}

The non-commutativity \eqref{eq:main-double-limit} confirms that the fixed-$\rho$ regime and the independent case $\rho=0$ are separated by a nontrivial phase transition: taking $K\to\infty$ first preserves the latent-factor effect, whereas taking $\rho\downarrow0$ first recovers the standard Cauchy reference.

\subsection{Weakening dependence and the fixed-level boundary layer}

We now consider the triangular regime $\rho_K\downarrow0$, which interpolates between the random-limit regime above and the independent Cauchy limit at $\rho=0$. The fixed-level behaviour turns out to depend on two logarithmic scales: the broader scale $c_K=\rho_K\log K$ and the finer boundary-layer scale $s_K=\sqrt{\rho_K}(\log K)^{3/2}$.

Both scales arise from the same source: a nonzero latent factor $V$ shifts each $Z_i$ by $\sqrt{\rho_K}V$, but this shift affects the CCT statistic only through the extreme upper tail of the $Z_i$, where values near $\sqrt{2\log K}$ produce Cauchy scores of order~$K$. The parameter $c_K$ measures whether the shift is large enough to change the density of such extreme values, while $s_K$ captures the resulting net effect on the mean of the statistic. Theorem~\ref{thm:stablelimit} below makes this decomposition precise.

Now let $\rho=\rho_K\downarrow0$ with $K$, and write
\begin{equation}\label{eq:main-scales}
  c_K:=\rho_K\log K,
  \qquad
  s_K:=\sqrt{\rho_K}(\log K)^{3/2}.
\end{equation}
The two scales satisfy $c_K=s_K^2/(\log K)^2$, so bounded boundary layers always lie inside the small-$c_K$ regime. For each fixed $v\in\R$, let
\begin{equation}\label{eq:main-triangular-array}
  Z_{i,K}^{(v)}:=\sqrt{\rho_K} v+\sqrt{1-\rho_K} \varepsilon_i,
  \qquad
  X_{i,K}^{(v)}:=f \left(Z_{i,K}^{(v)}\right),
  \qquad
  T_K^{(v)}:=\frac1K\sum_{i=1}^K X_{i,K}^{(v)}.
\end{equation}
The truncated-mean centring term is defined by
\begin{equation}\label{eq:main-bk-def}
  b_K(v)
  :=
  \E \left[
    X_{1,K}^{(v)} \1 \left\{|X_{1,K}^{(v)}|\le K\right\}
  \right]
  =
  K \E \left[
    \frac{X_{1,K}^{(v)}}{K} \1 \left\{\left|\frac{X_{1,K}^{(v)}}{K}\right|\le 1\right\}
  \right].
\end{equation}
If $c_K\to c\in[0,\infty)$, define
\begin{equation}\label{eq:main-lambda}
  \lambda^{+}(v):=\exp \left(-c+\sqrt{2c} v\right),
  \qquad
  \lambda^{-}(v):=\exp \left(-c-\sqrt{2c} v\right),
\end{equation}
and let
\begin{equation}\label{eq:main-levy}
  \Lambda_v(dx)
  :=
  \frac{\lambda^{+}(v)}{\pi x^2} \1\{x>0\} dx
  +
  \frac{\lambda^{-}(v)}{\pi x^2} \1\{x<0\} dx.
\end{equation}

The next theorem decomposes the conditional null distribution of $T_K$ into a $1$-stable fluctuation and a deterministic centring term, and gives the small-$c_K$ expansion of the latter.

\begin{theorem}[Conditional decomposition into stable fluctuation and centring]\label{thm:stablelimit}
Assume $\rho_K\downarrow0$ and $c_K\to c\in[0,\infty)$. Fix $v\in\R$. Then, conditionally on $V=v$, we have
\begin{equation}\label{eq:main-conditional-stable}
  T_K-b_K(v)\dto S_{c,v},
\end{equation}
where $S_{c,v}$ is the $1$-stable law determined by the Lévy measure $\Lambda_v$ in \eqref{eq:main-levy}. Moreover,
\begin{equation}\label{eq:main-bk-limit}
  \rho_K b_K(v)\to
  B_c(v):=
  \frac{2}{\pi}\int_0^{\sqrt{2c}} t e^{-t^2/2}\sinh(vt) dt.
\end{equation}
If, in addition, $c_K\to0$, then
\begin{equation}\label{eq:main-bk-smallc}
  b_K(v)=\kappa v s_K+o(s_K),
  \qquad
  \kappa:=\frac{4\sqrt2}{3\pi}.
\end{equation}
\end{theorem}

The stable law $S_{c,v}$ in \eqref{eq:main-conditional-stable} is characterized by
\begin{equation}\label{eq:main-stable-cf}
  \E \left[e^{itS_{c,v}}\right]
  =
  \exp \left\{
    \int_{\R\setminus\{0\}}
    \left(e^{itx}-1-itx\1\{|x|\le1\}\right) \Lambda_v(dx)
  \right\},
  \qquad t\in\R,
\end{equation}
with $\Lambda_v$ built from $\lambda^{\pm}(v)$ in \eqref{eq:main-lambda}.

Because the stable fluctuation $S_{c,v}$ does not depend on $K$, the $K$-dependence of the size distortion is carried entirely by the deterministic centring $b_K(v)$. Taking the expectation over $V$ yields the unconditional rejection probability. The first corollary specialises to $c_K\to0$, where $S_{c,v}$ reduces to the standard Cauchy law and $b_K(v)\approx\kappa vs_K$.

\begin{corollary}\label{cor:phasediagram}
Let $\alpha\in(0,1/2)$ be fixed and assume $\rho_K\downarrow0$.
\begin{enumerate}
\item If $s_K\to s\in[0,\infty)$, then we have
\begin{equation}\label{eq:main-Psi}
  \Pbb(T_K>t_\alpha)\to
  \Psi_\alpha(s):=
  \E \left[
    \frac12-\frac1\pi\arctan \left(
      t_\alpha-\kappa sV
    \right)
  \right].
\end{equation}
\item If $s_K\to\infty$, then we have
\begin{equation}\label{eq:main-raw-half}
  \Pbb(T_K>t_\alpha)\to \frac12.
\end{equation}
\end{enumerate}
\end{corollary}

Thus $s_K$ is the sole parameter governing the fixed-level size of the raw CCT in this regime: $s_K\to0$ yields exactness, bounded $s_K$ yields the explicit distortion $\Psi_\alpha(s)$, and $s_K\to\infty$ drives the rejection probability to $1/2$.

\begin{corollary}\label{cor:iffexact}
Assume $\rho_K\downarrow0$. Then, for every fixed $\alpha\in(0,1/2)$, we have
\begin{equation}\label{eq:main-raw-exactness}
  \Pbb(T_K>t_\alpha)\to \alpha
  \quad\Longleftrightarrow\quad
  \rho_K(\log K)^3\to0.
\end{equation}
\end{corollary}

The exponent $3$ on $\log K$ comes from the small-$c$ behaviour of $B_c(v)$.
Expanding $\sinh(vt)\approx vt$ in \eqref{eq:main-bk-limit} gives
\[
  B_c(v)
  =\frac{2v}{\pi}\int_0^{\sqrt{2c}} t^2 e^{-t^2/2}dt + O(c^{5/2})
  =\kappa v c^{3/2}+O(c^{5/2}),
\]
so $b_K(v)\approx \kappa vs_K$ with $s_K=\sqrt{\rho_K}(\log K)^{3/2}$, and the exactness condition $s_K\to0$ reduces to $\rho_K(\log K)^3\to0$.
Because $b_K(v)$ is a centering term rather than a component of the stable limit $S_{c,v}$, this distortion does not appear in the vanishing-level regime $\alpha_K\to0$ studied by \cite{LiuXie2020}.

\section{Boundary-layer calibration}\label{sec:calibration}

This section introduces a calibrated version of the CCT that corrects the fixed-level size distortion identified in Section~\ref{sec:raw cct}. The statistic $T_K$ is kept unchanged; only the reference law is replaced so that the cutoff accounts for the leading boundary-layer centring.

\subsection{Calibrated reference family}

Specifically, we replace the standard Cauchy tail probability by the boundary-layer family
\begin{equation}\label{eq:main-pbl}
  p_{\mathrm{BL}}(t;s)
  :=
  \E \left[
    \frac12-\frac1\pi\arctan \left(t-\kappa sV\right)
  \right],
  \qquad s\ge0.
\end{equation}
For $s=0$ this reduces to the standard Cauchy tail probability used by the raw CCT. The parameter $s$ is the boundary-layer scale defined in \eqref{eq:main-scales}, so the reference family is parametrised by the same quantity that governs the size distortion of the raw test.

\begin{proposition}\label{prop:bl-structure}
Let $\mathsf{C}$ and $V\sim N(0,1)$ be independent. Then we have
\begin{equation}\label{eq:main-bl-law}
  T_s\stackrel{d}{=}\mathsf{C}+\kappa sV,
  \qquad
  \phi_{T_s}(u)=\exp \left(-|u|-\frac{\kappa^2s^2u^2}{2}\right).
\end{equation}
Consequently, for each $\alpha\in(0,1)$ the calibrated cutoff
\begin{equation}\label{eq:main-qalpha}
  q_\alpha(s):=\inf\{t:\ p_{\mathrm{BL}}(t;s)\le \alpha\}
\end{equation}
is uniquely defined. If $\alpha\in(0,1/2)$, then the map $s\mapsto q_\alpha(s)$ is strictly increasing on $[0,\infty)$ and satisfies
\begin{equation}\label{eq:main-qalpha-small}
  q_\alpha(0)=t_\alpha,
  \qquad
  q_\alpha(s)=t_\alpha+\frac{\kappa^2 t_\alpha}{1+t_\alpha^2}s^2+O_\alpha(s^4)
  \quad (s\downarrow0),
\end{equation}
while
\begin{equation}\label{eq:main-qalpha-large}
  q_\alpha(s)=\kappa z_{1-\alpha}s+O_\alpha(1)
  \quad (s\to\infty),
\end{equation}
where $z_{1-\alpha}:=\Ph^{-1}(1-\alpha)$.
\end{proposition}

The calibrated law $T_s\stackrel{d}{=}\mathsf{C}+\kappa sV$ is a location mixture of the Cauchy distribution with Gaussian mixing. The cutoff $q_\alpha(s)$ departs from the standard Cauchy cutoff $t_\alpha$ by $O(s^2)$ for small $s$ and grows as $\kappa z_{1-\alpha}s$ for large $s$.

\subsection{Uniform validity on bounded boundary layers}

The next result establishes a uniform distributional approximation of $T_K$ by the calibrated family \eqref{eq:main-pbl}, valid whenever $s_K$ remains bounded.

\begin{theorem}[Boundary-layer approximation]\label{thm:bluniform}
Assume $\rho_K\downarrow0$ and $\sup_K s_K<\infty$. Let $F_s(t)=1-p_{\mathrm{BL}}(t;s)$. Then we have
\begin{equation}\label{eq:main-bl-cdf}
  \sup_{t\in\R}
  \left|
    \Pbb(T_K\le t)-F_{s_K}(t)
  \right|
  \to0.
\end{equation}
Consequently, we have
\begin{equation}\label{eq:main-bl-unif}
  \sup_{u\in[0,1]}
  \left|
    \Pbb \left(p_{\mathrm{BL}}(T_K;s_K)\le u\right)-u
  \right|
  \to0.
\end{equation}
If $\widehat s_K\ge0$ and $\widehat s_K-s_K\to0$ in probability, then the same uniformity conclusion holds with $\widehat s_K$ in place of $s_K$.
\end{theorem}

Since BL-CCT modifies only the cutoff and not the statistic, the approximation \eqref{eq:main-bl-cdf} immediately yields calibrated rejection probabilities. The next corollary records the resulting exactness criterion.

\begin{corollary}\label{cor:bl-c0}
Assume $\rho_K\downarrow0$ and $c_K=\rho_K\log K\to0$. Then we have
\[
  \sup_{u\in[0,1]}
  \left|
    \Pbb \left(p_{\mathrm{BL}}(T_K;s_K)\le u\right)-u
  \right|
  \to0.
\]
Equivalently, for every fixed $\alpha\in(0,1)$, we have
\begin{equation}\label{eq:main-bl-exactness}
  \Pbb \left(T_K>q_\alpha(s_K)\right)\to \alpha.
\end{equation}
\end{corollary}

Comparing with Corollary~\ref{cor:iffexact}, the exactness threshold improves from $\rho_K=o((\log K)^{-3})$ for the raw CCT to $\rho_K=o((\log K)^{-1})$ for BL-CCT, because the boundary-layer centring has been absorbed into the reference family. The next subsection examines the residual size distortion when $c_K=\rho_K\log K$ remains bounded away from zero.

\subsection{Behaviour on the broader scale}

When $c_K=\rho_K\log K$ does not vanish, the boundary-layer correction no longer eliminates all size distortion. The remaining discrepancy is governed by $c_K$.

\begin{proposition}\label{prop:bl-c-phase}
Fix $\alpha\in(0,1/2)$ and assume $\rho_K\downarrow0$.
If $c_K=\rho_K\log K\to c\in(0,\infty)$, then we have
\begin{equation}\label{eq:main-Xi}
  \Pbb \left(T_K>q_\alpha(s_K)\right)
  \to
  \Xi_\alpha(c)
  :=
  \Pbb \left(M_c>\kappa z_{1-\alpha}\sqrt c\right),
\end{equation}
where $M_c=B_c(V)/c$.
If $c_K\to\infty$, then we have
\[
  \Pbb \left(T_K>q_\alpha(s_K)\right)\to0.
\]
\end{proposition}

The limiting size $\Xi_\alpha(c)$ is thus determined by the distribution of $M_c=B_c(V)/c$, a univariate transformation of the latent factor. The next result extends this to a finite-$K$ uniform approximation on compact positive $c$-windows.

\begin{proposition}\label{prop:bl-c-current}
Fix $0<c_-<c_+<\infty$ and assume
\[
  c_K\in[c_-,c_+]
\]
for all sufficiently large $K$. Let $H_c$ denote the distribution function of $M_c$. Then we have
\begin{equation}\label{eq:main-current-cdf}
  \sup_{x\in\R}
  \left|
    \Pbb \left(\frac{T_K}{\log K}\le x\right)-H_{c_K}(x)
  \right|
  \to 0.
\end{equation}
Consequently, for every fixed $\alpha\in(0,1/2)$, we have
\begin{equation}\label{eq:main-current-size}
  \Pbb \left(T_K>q_\alpha(s_K)\right)
  =
  \Xi_\alpha(c_K)+o(1).
\end{equation}
\end{proposition}

Together with Proposition~\ref{prop:bl-c-phase}, this shows that $c_K$ fully determines the calibrated size beyond the boundary layer. At conventional levels such as $\alpha=0.05$ and $0.1$, the size function $\Xi_\alpha(c)$ is strictly below $\alpha$ for all $c>0$, so the residual distortion is conservative.

\begin{corollary}\label{cor:bl-iff-common}
Assume $\rho_K\downarrow0$.  If $\alpha\in[\barPh(\sqrt{3}),1/2)$, then $\Xi_\alpha(c)<\alpha$ for every $c>0$, and therefore
\begin{equation}\label{eq:main-bl-common-exactness}
  \Pbb \left(T_K>q_\alpha(s_K)\right)\to\alpha
  \quad\Longleftrightarrow\quad
  \rho_K\log K\to0.
\end{equation}
In particular, this covers standard levels such as $\alpha=0.05$ and $\alpha=0.1$.
\end{corollary}

Corollary~\ref{cor:bl-iff-common} establishes that $\rho_K\log K\to0$ is both necessary and sufficient for exactness of BL-CCT at conventional levels, paralleling Corollary~\ref{cor:iffexact} for the raw CCT. When $c_K$ is bounded away from zero, the calibrated size converges to $\Xi_\alpha(c)<\alpha$, which can be computed numerically. Section~\ref{sec:numerics} illustrates the finite-sample behaviour.

\section{Power analysis}\label{sec:power}

This section studies alternatives of the form
\[
  Z_i=\mu_i+\sqrt{\rho_K}V+\sqrt{1-\rho_K}\varepsilon_i,
  \qquad i=1,\ldots,K,
\]
while keeping the CCT statistic
\[
  T_K=\frac{1}{K}\sum_{i=1}^K f(Z_i)
\]
unchanged; the null model of Section~\ref{sec:setup} is the case $\mu_i\equiv0$.  Throughout this section,
\[
  \bar F_\gamma(x):=\frac12-\frac1\pi\arctan(x/\gamma),\qquad \gamma>0,
\]
is the survival function of a centred Cauchy law with scale $\gamma$.  For an alternative law, define
\[
  \Pi^{\rm raw}_{K}:=\Pbb_{\rm alt}(T_K>t_\alpha),
  \qquad
  \Pi^{\rm BL}_{K}:=\Pbb_{\rm alt}(T_K>q_\alpha(s_K)).
\]
Since $q_\alpha(s_K)\ge t_\alpha$, the BL rejection region is contained in that of the raw CCT, so the calibration can only lower power; the question is whether the loss is asymptotically negligible. We examine three canonical classes of alternatives: \emph{local dense} shifts that perturb every coordinate at the Pitman scale; \emph{sparse} signals, a vanishing fraction of large effects and the regime for which the CCT was designed~\cite{LiuXie2020}, where one small $p$-value already yields a $K$-scale Cauchy score that dominates the polylogarithmic cutoff gap; and \emph{dense Gaussian random effects}, a variance-type departure with no fixed shift. In each case the signal enters the limit of $T_K$ through the same channel as the latent factor, by tilting the marginal mean, so the analysis reuses the stable-limit machinery of Sections~\ref{sec:raw cct}--\ref{sec:calibration}.

\subsection{Local dense alternatives}

A marginal mean shift $\mu$ tilts the truncated mean of the Cauchy score by approximately $\kappa\mu(\log K)^{3/2}$, the same channel through which the latent factor produces the null centring $b_K(v)\approx\kappa\sqrt{\rho_K}\,v\,(\log K)^{3/2}$ in Theorem~\ref{thm:stablelimit}; shifts of order $(\log K)^{-3/2}$ therefore produce a nondegenerate limiting drift. The theorem below is stated in a heterogeneous form: the first condition rules out individually large shifts, which would activate the sparse mechanism of the next subsection, and the other two normalise the aggregate drift to a limit $h$. The homogeneous alternative $\mu_i\equiv h/(\kappa (\log K)^{3/2})$ is the special case.

\begin{theorem}[Local dense power]\label{thm:power-local}
Assume
\[
  \rho_K\downarrow0,\qquad
  s_K=\sqrt{\rho_K}(\log K)^{3/2}\to s\in[0,\infty).
\]
Let $\mu_{1,K},\ldots,\mu_{K,K}$ be deterministic shifts satisfying
\[
  \max_{1\le i\le K}|\mu_{i,K}|\sqrt{\log K}\to0,
\]
and
\[
  \frac{(\log K)^{3/2}}{K}\sum_{i=1}^K|\mu_{i,K}|=O(1),
  \qquad
  \frac{\kappa (\log K)^{3/2}}{K}\sum_{i=1}^K\mu_{i,K}\to h .
\]
Then, conditionally on $V=v$,
\[
  T_K\dto \mathsf{C}(h+\kappa sv,1).
\]
Consequently,
\[
  \Pi^{\rm raw}_{K}\to
  \Pi^{\rm raw}_\alpha(h,s)
  :=
  \E\left[\bar F_1(t_\alpha-h-\kappa sV)\right],
\]
and
\[
  \Pi^{\rm BL}_{K}\to
  \Pi^{\rm BL}_\alpha(h,s)
  :=
  \E\left[\bar F_1(q_\alpha(s)-h-\kappa sV)\right]
  =
  p_{\rm BL}(q_\alpha(s)-h;s).
\]
Moreover,
\[
  0\le
  \Pi^{\rm raw}_\alpha(h,s)-\Pi^{\rm BL}_\alpha(h,s)
  =
  \frac{1}{\pi}\E\left[
    \arctan(q_\alpha(s)-h-\kappa sV)
    -
    \arctan(t_\alpha-h-\kappa sV)
  \right],
\]
and, as $s\downarrow0$,
\[
  \Pi^{\rm raw}_\alpha(h,s)-\Pi^{\rm BL}_\alpha(h,s)=O_\alpha(s^2).
\]
In particular, on the exactness scale $\rho_K(\log K)^3\to0$, raw CCT and BL-CCT have the
same local asymptotic power,
\[
  \bar F_1(t_\alpha-h)
  =
  \frac{1}{2}-\frac{1}{\pi}\arctan(t_\alpha-h).
\]
\end{theorem}

At $h=0$ the two limits reduce to the null limits $\Psi_\alpha(s)$ of Corollary~\ref{cor:phasediagram} and $\alpha$, so for $s>0$ part of the raw CCT test's apparent power advantage is its size inflation; the difference display above quantifies this effect, and it vanishes at rate $s^2$.

\begin{remark}[Homogeneous local shift]\label{rem:power-hom-local}
For the homogeneous alternative $\mu_i\equiv h/(\kappa (\log K)^{3/2})$, the assumptions of Theorem~\ref{thm:power-local} hold automatically, so $(\log K)^{-3/2}$ is the Pitman scale for a dense deterministic mean shift of the CCT score.
\end{remark}

\subsection{Sparse signals}

Let $S_K\subset\{1,\ldots,K\}$ be a signal set, independent of $V$ and of the
$\varepsilon_i$'s, such that
\[
  |S_K|=K^{1-\beta+o_p(1)},\qquad \beta\in(1/2,1).
\]
For $i\in S_K$, set
\[
  \mu_i=\sqrt{2r\log K},\qquad r\in(0,1),
\]
and for $i\notin S_K$, set $\mu_i=0$.  This includes the Bernoulli sparse mixture in which each coordinate is independently a signal with probability $K^{-\beta}$.

The first-order boundary for CCT is the Bonferroni/max boundary
\[
  r_{\rm max}(\beta):=\bigl(1-\sqrt{1-\beta}\bigr)^2.
\]
An aggregate heuristic, which asks when the expected sum of signal scores reaches the order $K$ of the cutoff fluctuations, would instead point to the larger value $r=\beta$. But the CCT can already reject once a single coordinate produces a $K$-scale Cauchy score: the largest of the $|S_K|\approx K^{1-\beta}$ signal coordinates is of size $\{\sqrt{2r\log K}+\sqrt{2(1-\beta)\log K}\}\{1+o_p(1)\}$, and this exceeds the $K$-score threshold $\sqrt{2\log K}$ precisely when $r>r_{\rm max}(\beta)$.

\begin{theorem}[Sparse power and the CCT detection boundary]\label{thm:power-sparse}
Assume
\[
  \rho_K\downarrow0,\qquad \rho_K\log K=O(1),
\]
and let $d_K>0$ be any deterministic rejection threshold satisfying
\[
  \max\{\log d_K,0\}=o(\log K).
\]
Thus $d_K$ may grow, but more slowly than any positive power of $K$. This class contains $d_K=t_\alpha$ and, under $\rho_K\log K=O(1)$, also
$d_K=q_\alpha(s_K)$.  Under the sparse alternative above,
\[
  \Pbb_{\rm alt}(T_K>d_K)\to1
  \qquad\text{if } r>r_{\rm max}(\beta).
\]
If $r<r_{\rm max}(\beta)$, then
\[
  T_K=T_K^{(0)}+o_p(1),
\]
where $T_K^{(0)}$ is the null CCT statistic with the same one-factor dependence and the
same noises.  Consequently, for the same deterministic threshold sequence $d_K$,
\[
  \Pbb_{\rm alt}(T_K>d_K)-\Pbb_0(T_K^{(0)}>d_K)\to0 .
\]

In particular, if $r>r_{\rm max}(\beta)$, then
\[
  \Pi^{\rm raw}_{K}\to1,\qquad \Pi^{\rm BL}_{K}\to1,
\]
so calibration has no first-order power cost in the detectable sparse regime.  If
$r<r_{\rm max}(\beta)$, then the signal is invisible to CCT at first order: on the
exactness scale $\rho_K(\log K)^3\to0$, both raw CCT and BL-CCT have limiting power
$\alpha$; on a bounded boundary layer $s_K\to s<\infty$, raw CCT has the null limit
$\Psi_\alpha(s)$ and BL-CCT has the calibrated null limit $\alpha$.
\end{theorem}

\subsection{Dense Gaussian random effects}

Now let
\[
  \mu_i\overset{\rm iid}{\sim}N(0,\tau_K^2),\qquad
  \tau_K^2\log K\to w\in[0,\infty)
\]
with the $\mu_i$ independent of $V$ and of the $\varepsilon_i$. The probabilities in this subsection are annealed over the Gaussian random effects: conditionally on $V=v$, we integrate out the $\mu_i$, so that the coordinates are independent with marginal law
\[
  Z_i\mid V=v\sim
  N\left(\sqrt{\rho_K}v,1-\rho_K+\tau_K^2\right).
\]

\begin{theorem}[Dense Gaussian random-effects power]\label{thm:power-random-effects}
Assume $\tau_K^2\log K\to w\in[0,\infty)$.  The rejection probabilities below are annealed
over the Gaussian random effects.

\emph{(i) Boundary-layer scale.}  If $\rho_K\log K\to0$ and $s_K\to s\in[0,\infty)$,
then, conditionally on $V=v$,
\[
  T_K\dto \mathsf{C}(\kappa_wsv,e^w),
  \qquad
  \kappa_w:=\frac{2}{\pi}\int_0^{\sqrt2}y^2e^{wy^2/2}dy ,
\]
and consequently
\[
  \Pi^{\rm raw}_{K}\to
  \Pi^{\rm raw,G}_\alpha(w,s)
  :=
  \E\left[\bar F_{e^w}(t_\alpha-\kappa_wsV)\right],
  \qquad
  \Pi^{\rm BL}_{K}\to
  \Pi^{\rm BL,G}_\alpha(w,s)
  :=
  \E\left[\bar F_{e^w}(q_\alpha(s)-\kappa_wsV)\right],
\]
with
\[
  0\le
  \Pi^{\rm raw,G}_\alpha(w,s)-\Pi^{\rm BL,G}_\alpha(w,s)
  =O_\alpha(s^2)
  \qquad(s\downarrow0).
\]
In particular, on the exactness scale $\rho_K(\log K)^3\to0$, both powers converge to
$\bar F_{e^w}(t_\alpha)$.

\emph{(ii) Broader common-correlation scale.}  If $\rho_K\log K\to c\in(0,\infty)$, then,
with
\[
  D_{c,w}(v):=
  \frac{2}{\pi}\int_0^{\sqrt2}
       y\exp\left\{\frac{(w-c)y^2}{2}\right\}
       \sinh(\sqrt cvy)dy ,
\]
we have
\[
  \Pi^{\rm raw}_{K}\to \frac{1}{2},
  \qquad
  \Pi^{\rm BL}_{K}\to
  \Pbb\left\{
    D_{c,w}(V)>\kappa z_{1-\alpha}\sqrt c
  \right\},
\]
and, for each $v>0$, the map $w\mapsto D_{c,w}(v)$ is increasing.
\end{theorem}

The constants tie both parts back to the null theory. Since $\kappa_0=\kappa$ and $e^0=1$, setting $w=0$ in part~(i) recovers the null boundary-layer limit underlying Corollary~\ref{cor:phasediagram}, and the calibration cost on a bounded layer is of the same second order in $s$ as for local dense alternatives. In part~(ii) the centring, of order $\log K$, dominates the $O_p(1)$ stable fluctuation, and the threshold $\kappa z_{1-\alpha}\sqrt c$ is the large-$s$ form of the calibrated cutoff, $q_\alpha(s_K)=\kappa z_{1-\alpha}\sqrt c\log K+o(\log K)$ by Proposition~\ref{prop:bl-structure}. On this scale the two tests separate sharply: the raw CCT rejects with probability tending to $1/2$ under the null and the alternative alike, so it ceases to discriminate, whereas the BL limiting power is increasing in the random-effects strength $w$. At $w=0$ it reduces to the null size $\Xi_\alpha(c)$ of Proposition~\ref{prop:bl-c-phase}, since $D_{c,0}(V)=B_c(V)/c=M_c$, and as $c\downarrow0$ it tends to $\alpha$, matching the BL exactness boundary $\rho_K\log K\to0$.

\begin{remark}[Plug-in calibration under alternatives]\label{rem:power-plugin}
The BL power $\Pi^{\rm BL}_K$ is defined with the oracle cutoff $q_\alpha(s_K)$. Under an alternative, the plug-in estimator \eqref{eq:main-rhohat} need not be consistent for $\rho_K$, because the sample variance of the $z_i$ also absorbs the signal. This biases $\widehat\rho$ downwards and moves $q_\alpha(\widehat s_K)$ towards the raw cutoff $t_\alpha$, so the estimation error acts in the direction of more rejections rather than fewer; the null size guarantees are unaffected, since under the null $\widehat\rho$ is consistent (Proposition~\ref{prop:plugin}). A full analysis of plug-in power is not pursued here.
\end{remark}

\section{Numerical experiments}
\label{sec:numerics}

This section describes the computation of BL-CCT and presents several Monte Carlo experiments to validate the theoretical results of Sections~\ref{sec:raw cct} and~\ref{sec:calibration}. 

\subsection{Implementation}

The calibrated $p$-value $p_{\mathrm{BL}}(t;s)$ is a one-dimensional integral with respect to the standard normal density of $V$, which can be evaluated by standard numerical integration. Under the equicorrelated Gaussian copula, a natural plug-in estimator of $\rho_K$ is given by
\begin{equation}\label{eq:main-rhohat}
  \widehat\rho
  :=
  \max \left\{0,
    1-\frac{1}{K-1}\sum_{i=1}^K (z_i-\bar z)^2
  \right\},
\end{equation}
with
\begin{equation}\label{eq:main-shat}
  z_i=\Ph^{-1}(1-p_i),
  \qquad
  \bar z=\frac1K\sum_{i=1}^K z_i,
  \qquad
  \widehat s_K=\sqrt{\widehat\rho} (\log K)^{3/2},
  \qquad
  \widehat c_K=\widehat\rho\log K.
\end{equation}
Proposition~\ref{prop:plugin} in the Appendix shows that $\widehat\rho-\rho_K=O_p(K^{-1/2})$ under the null model, so $\widehat s_K-s_K\to0$ on compact boundary layers and the plug-in BL-CCT has the same asymptotic calibration as the oracle procedure (Theorem~\ref{thm:bluniform}). The same proposition also records a sharper rate, $|\widehat s_K-s_K|=O_p((\log K)^{3/2}/\sqrt{K\rho_K})$, when $K\rho_K/(\log K)^3\to\infty$, showing that the plug-in reference tracks the oracle reference in larger-correlation regimes as well.

\subsection{Monte Carlo experiments}

All simulations are conducted under the equicorrelated Gaussian null with one-sided $p$-values at nominal level $\alpha=0.05$. Each configuration uses $2000$ Monte Carlo repetitions. The scale parameters are held fixed while $K$ varies, so the correlation $\rho_K$ is determined by the identity $s_K=\sqrt{\rho_K}(\log K)^{3/2}$ or $c_K=\rho_K\log K$: specifically, we set $\rho_K=s^2/(\log K)^3$ so that $s_K\equiv s$ in panel~(a), and $\rho_K=c/\log K$ so that $c_K\equiv c$ in panel~(b). Since $\rho_K$ decreases with $K$ for each fixed $s$ or $c$, larger $K$ corresponds to weaker dependence.

\begin{figure}[tbp]
\centering
\includegraphics[width=0.98\linewidth]{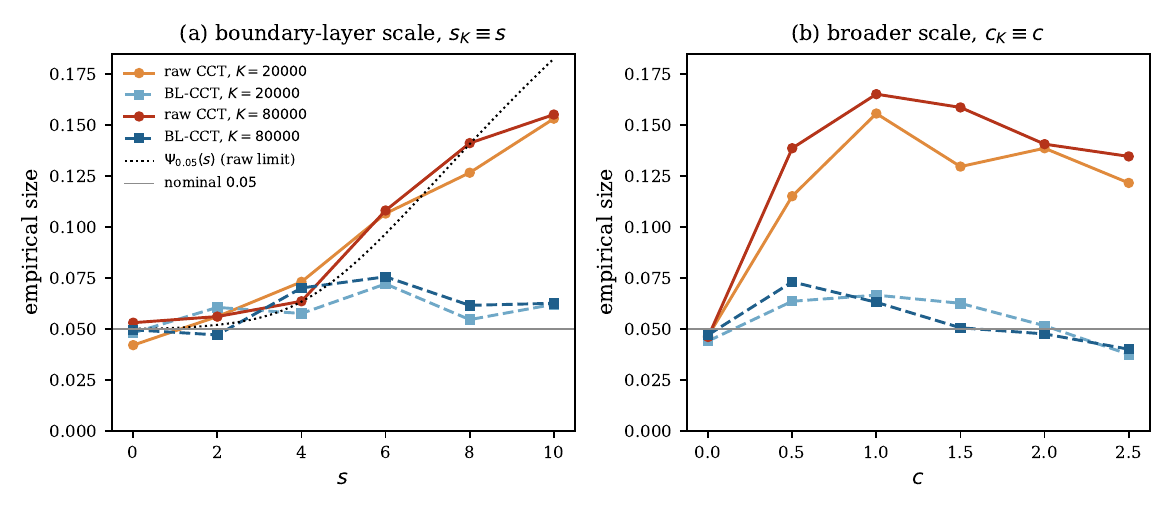}
\caption{Empirical size at level $0.05$ under the equicorrelated Gaussian null: raw CCT (solid) versus plug-in BL-CCT (dashed) on a common axis, for $K\in\{20000,80000\}$. Panel~(a): boundary-layer scale $s_K\equiv s$, with the raw-CCT limit $\Psi_{0.05}(s)$ of Corollary~\ref{cor:phasediagram} shown dotted. Panel~(b): broader scale $c_K\equiv c$.}
\label{fig:sim}
\end{figure}

Figure~\ref{fig:sim} overlays the empirical size of the raw CCT (solid) and of the plug-in BL-CCT (dashed) as functions of the scale parameter, for $K\in\{20000,80000\}$, so that the distortion and its correction can be compared on a single axis. Panel~(a) uses the boundary-layer parametrisation, with $s$ ranging over $\{0,2,4,6,8,10\}$. The raw CCT increases monotonically with $s$ and closely follows the theoretical limit $\Psi_{0.05}(s)$ of Corollary~\ref{cor:phasediagram}: its empirical size is close to the nominal level at $s=0$ and exceeds $0.15$ at $s=10$, illustrating the practical severity of the fixed-level distortion. The BL-CCT instead stays near the nominal level throughout, never exceeding about $0.075$; the small finite-$K$ excess at moderate $s$ is consistent with the remainder term in Theorem~\ref{thm:bluniform}, and by Proposition~\ref{prop:plugin} the plug-in calibration used here is asymptotically equivalent to the oracle calibration.

Panel~(b) uses the broader scale $c_K=\rho_K\log K$, with $c$ ranging over $\{0,0.5,1,1.5,2,2.5\}$. The two procedures separate already at small $c$: the raw CCT is strongly anti-conservative, its empirical size rising to roughly $0.16$, whereas the BL-CCT stays close to the nominal level and becomes conservative as $c$ grows, in agreement with the conservative limit $\Xi_{0.05}(c)$ of Proposition~\ref{prop:bl-c-phase} and the exactness characterisation of Corollary~\ref{cor:bl-iff-common}. The overlay confirms that the boundary-layer calibration removes most of the latent-factor size inflation suffered by the raw CCT under both scalings.

\section{Conclusion}\label{sec:discussion}

This paper has studied the fixed-level calibration of the Cauchy combination test under dependence. In the one-factor Gaussian copula model, the fixed-level size distortion of the raw CCT is driven by a deterministic latent-factor centring term, and exactness holds if and only if $\rho_K(\log K)^3\to0$. Replacing only the reference law by a Gaussian-smoothed Cauchy family yields BL-CCT, which is asymptotically exact under the weaker condition $\rho_K\log K\to0$, necessary and sufficient at the conventional levels covered by Corollary~\ref{cor:bl-iff-common}.

Because BL-CCT only raises the cutoff, the calibration could in principle cost power. Section~\ref{sec:power} shows that it essentially does not in the regimes where the stable-limit analysis applies, namely deterministic local dense shifts, positive off-boundary sparse signals, and annealed dense Gaussian random effects: on the exactness scale there is no first-order power cost, and above the sparse max boundary both tests have power tending to one. Sharper results at the knife-edge $r=r_{\mathrm{max}}(\beta)$, quenched random-effects power, and signed or non-Gaussian alternatives are left for future work.

The equicorrelated model isolates a single latent factor; under multi-factor or heterogeneous-block dependence each factor may contribute its own centring term, and whether the one-parameter BL-CCT family still suffices is an open question. Two further extensions are natural: pairing the statistic-modification strategies of \cite{OuyangEtAl2024,LiuMengPillai2025,ChenXuGao2025TCCT,Bouamara2025StepC} with a boundary-layer calibration of their own reference laws, the two corrections being largely orthogonal; and a unified fixed-level theorem for combination rules based on regularly varying scores with tail index $1$, which should exhibit analogous centring effects under the same copula model and include the harmonic mean $p$-value \cite{Wilson2019HMP} and the classes of \cite{FangEtAl2023} and \cite{GuiJiangWang2025}.

\bibliographystyle{alpha}
\bibliography{refs}

\newpage

\appendix

\begin{center}
  \Large Appendix
\end{center}
\section{Auxiliary lemmas}\label{app:common}

We retain the notation of Sections~\ref{sec:setup}--\ref{sec:calibration}. In particular, \eqref{eq:main-onefactor}, \eqref{eq:main-score}, \eqref{eq:main-stat}, \eqref{eq:main-mu}, and \eqref{eq:main-scales} remain in force throughout the appendix. When conditioning on $V=v$ in the triangular regime, we write
\[
  Z_{i,K}^{(v)}=\sqrt{\rho_K} v+\sqrt{1-\rho_K} \varepsilon_i,
  \qquad
  X_{i,K}^{(v)}=f \left(Z_{i,K}^{(v)}\right),
  \qquad
  T_K^{(v)}=\frac1K\sum_{i=1}^K X_{i,K}^{(v)}.
\]

\begin{lemma}\label{lem:f-basic}
The function $f$ defined in \eqref{eq:main-score} satisfies:
\begin{enumerate}
\item $f$ is odd: $f(-z)=-f(z)$ for all $z\in\R$.
\item $f$ is strictly increasing on $\R$.
\item For all $z\in\R$,
\begin{equation}\label{eq:f-cot}
  f(z) = \cot\left(\pi \barPh(z)\right).
\end{equation}
\end{enumerate}
\end{lemma}

\begin{proof}
Oddness follows from $\Ph(-z)=1-\Ph(z)$, monotonicity from the monotonicity of $z\mapsto \Ph(z)$ and $u\mapsto \tan(\pi u)$ on $(-1/2,1/2)$, and \eqref{eq:f-cot} from
\[
  \tan\left(\pi(\Ph(z)-1/2)\right)
  =
  \tan\left(\pi(1/2-\barPh(z))\right)
  =
  \cot\left(\pi\barPh(z)\right).
\]
\end{proof}

\begin{lemma}\label{lem:inv-asymp}
Let $t(x)=f^{-1}(x)$ for $x>0$.
Then $t(x)\to\infty$ as $x\to\infty$ and
\[
\frac{t(x)}{\sqrt{2\log x}}\to 1,
\qquad x\to\infty.
\]
Moreover,
\[
e^{-t(x)^2/2} = \sqrt{2/\pi}\frac{t(x)}{x} (1+o(1)),
\qquad x\to\infty.
\]
\end{lemma}

\begin{proof}
Monotonicity of $f$ (Lemma~\ref{lem:f-basic}) and $f(z)\sim\sqrt{2/\pi}z e^{z^2/2}\to\infty$ (Lemma~\ref{lem:f-growth}) imply $t(x)\to\infty$ and
\[
  e^{-t(x)^2/2}=\sqrt{2/\pi}\frac{t(x)}{x}(1+o(1)),\qquad x\to\infty.
\]
Taking logarithms,
\[
  t(x)^2=2\log x-2\log t(x)-\log(2/\pi)+o(1).
\]
Since $t(x)\to\infty$, eventually $\log t(x)\ge 0$, hence $t(x)^2\le 2\log x+O(1)$ and $\log t(x)=O(\log\log x)=o(\log x)$. Dividing the logarithmic identity by $2\log x$ gives $t(x)/\sqrt{2\log x}\to 1$.
\end{proof}

\begin{lemma}\label{lem:f-growth}
We have
\begin{equation}\label{eq:f-asymp}
  \lim_{|z|\to\infty}\frac{f(z)}{\sqrt{2/\pi} z e^{z^2/2}} = 1.
\end{equation}
Moreover, there exists a constant $C>0$ such that for all $z\in\R$,
\begin{equation}\label{eq:f-upper}
  |f(z)| \le C(1+|z|) e^{z^2/2}.
\end{equation}
\end{lemma}
\begin{proof}
We first prove \eqref{eq:f-asymp} for $z\to+\infty$.
By Lemma~\ref{lem:f-basic}\eqref{eq:f-cot},
\[
f(z)=\cot(\pi \barPh(z)).
\]
As $z\to\infty$, $\barPh(z)\downarrow 0$, and the elementary expansion $\cot(\pi y)\sim 1/(\pi y)$ as $y\downarrow0$ gives 
\[
\cot(\pi \barPh(z)) \sim \frac{1}{\pi \barPh(z)}\qquad (z\to\infty).
\]
Using the standard Mills ratio asymptotic $\barPh(z)\sim \ph(z)/z$ as $z\to\infty$, we obtain
\[
f(z)\sim \frac{z}{\pi\ph(z)}
= \sqrt{2/\pi} z e^{z^2/2}.
\]
This proves \eqref{eq:f-asymp} for $z\to+\infty$.
For $z\to-\infty$, use oddness $f(-z)=-f(z)$ from Lemma~\ref{lem:f-basic} and the already proved $z\to+\infty$ asymptotic.

For the global upper bound \eqref{eq:f-upper}, we split into $|z|\le 1$ and $|z|>1$.
On the compact set $|z|\le 1$, $f$ is continuous hence bounded: $|f(z)|\le C_0$. Since $e^{z^2/2}\ge 1$ and $1+|z|\ge 1$, this implies \eqref{eq:f-upper} on $|z|\le 1$ with $C\ge C_0$.

For $|z|>1$, the standard Mills ratio bounds imply for $z>1$:
\[
\barPh(z)\ge \frac{\ph(z)}{z+1/z}\ge \frac{\ph(z)}{2z},
\]
hence by $\cot(\pi y)\le 1/(\pi y)$ for $y\in(0,1/2]$ (since $\tan u\ge u$ for $u\in[0,\pi/2)$),
\[
f(z)=\cot(\pi\barPh(z))\le \frac{1}{\pi\barPh(z)}\le \frac{2z}{\pi\ph(z)} = 2\sqrt{2/\pi} z e^{z^2/2}.
\]
By oddness, the same bound holds for $z<-1$ in absolute value. Absorbing constants yields \eqref{eq:f-upper}.
\end{proof}

\section{Proofs in Section~\ref{sec:raw cct}: fixed positive correlation}\label{app:fixedcorr}

\subsection{Proof of Theorem 3.1}
\begin{proposition}\label{prop:cond-L1}
Fix $\rho\in(0,1)$. For every $v\in\R$, if $Z\sim N(\sqrt{\rho} v, 1-\rho)$ and $X=f(Z)$, then we have
\[
\E\left[|X| \mid V=v\right] < \infty.
\]
In particular, the conditional mean
\[
  \mu_\rho(v)=\E\left[f(\sqrt{\rho} v+\sqrt{1-\rho} \varepsilon)\right],\qquad \varepsilon\sim N(0,1),
\]
is well-defined and finite for all $v\in\R$.
\end{proposition}
\begin{proof}
Let $Z\sim N(m,\sigma^2)$ with $m=\sqrt\rho v$ and $\sigma^2=1-\rho\in(0,1)$. By Lemma~\ref{lem:f-growth}\eqref{eq:f-upper}, $|f(Z)|\le C(1+|Z|)e^{Z^2/2}$, so it suffices to bound $\E[(1+|Z|)e^{\lambda Z^2}]$ for some $\lambda>1/2$. Since $1/(2\sigma^2)>1/2$, the Gaussian MGF identity
\begin{equation}\label{eq:gauss-mgf}
  \E\bigl[e^{\lambda Z^2}\bigr]
  =\frac{1}{\sqrt{1-2\lambda\sigma^2}}\exp\left(\frac{\lambda m^2}{1-2\lambda\sigma^2}\right),\qquad 2\lambda\sigma^2<1,
\end{equation}
yields $\E[e^{\lambda Z^2}]<\infty$ for every $\lambda<1/(2\sigma^2)$. Choosing $\lambda$ slightly larger than $1/2$ and absorbing the polynomial factor via $(1+|z|)\le e^{|z|}\le e^{\eta z^2+1/(4\eta)}$ (with $\eta$ small enough) gives $\E|f(Z)|<\infty$.
\end{proof}

\begin{proof}[Proof of Theorem~\ref{thm:randomlimit}]
Conditional on $V=v$, the $X_i=f(Z_i)$ are i.i.d.\ with $\E[|X_1|\mid V=v]<\infty$ and mean $\mu_\rho(v)$ (Proposition~\ref{prop:cond-L1}). The conditional strong law yields $\Pbb(T_K\to\mu_\rho(v)\mid V=v)=1$ for every $v\in\R$; integrating against the law of $V$ gives $\Pbb(T_K\to\mu_\rho(V))=1$. Since $\mu_\rho(V)$ has a continuous law (Lemma~\ref{lem:mu-props}: continuity and strict monotonicity of $\mu_\rho$), Portmanteau yields $\Pbb(T_K>t_\alpha)\to\Pbb(\mu_\rho(V)>t_\alpha)$.
\end{proof}

\subsection{Proof of Proposition 3.2}\label{subsec:rho0}

At $\rho=0$ the scores $X_i=f(Z_i)$ are i.i.d.\ standard Cauchy and the CCT is exact at every fixed $\alpha$, whereas for fixed $\rho>0$ Theorem~\ref{thm:randomlimit} shows $T_K\to\mu_\rho(V)$ almost surely. The following results establish properties of $\mu_\rho$ and show that $\lim_{\rho\downarrow0}\lim_K$ and $\lim_K\lim_{\rho\downarrow0}$ yield different limits.

\begin{lemma}\label{lem:mu-props}
Fix $\rho\in(0,1)$ and recall the definition \eqref{eq:main-mu}.
Then the function $v\mapsto \mu_\rho(v)$ is finite, continuous, odd, and strictly increasing on $\R$.
Consequently, for every $t>0$ there exists a unique value $v_\rho(t)>0$ such that $\mu_\rho \left(v_\rho(t)\right)=t$.
Moreover, for $V\sim N(0,1)$,
\begin{equation}\label{eq:limit-size-vrho}
  \Pbb\left(\mu_\rho(V)>t\right)=\barPh \left(v_\rho(t)\right).
\end{equation}
\end{lemma}

\begin{proof}
Finiteness of $\mu_\rho(v)$ for each $v$ is proved in Proposition~\ref{prop:cond-L1}.

\paragraph{Oddness.}
Let $v\in\R$ and let $\varepsilon\sim N(0,1)$.
Using symmetry $\varepsilon \dto -\varepsilon$ and oddness of $f$,
\begin{align*}
\mu_\rho(-v)
&=\E \left[f \left(-\sqrt{\rho} v+\sqrt{1-\rho} \varepsilon\right)\right]\\
&=\E \left[f \left(-\sqrt{\rho} v-\sqrt{1-\rho} \varepsilon\right)\right]\\
&=\E \left[f \left(-\left(\sqrt{\rho} v+\sqrt{1-\rho} \varepsilon\right)\right)\right]\\
&=\E \left[-f \left(\sqrt{\rho} v+\sqrt{1-\rho} \varepsilon\right)\right]
= -\mu_\rho(v).
\end{align*}

\paragraph{Strict monotonicity.}
Let $v_1<v_2$.
Couple the corresponding Gaussians using the same $\varepsilon\sim N(0,1)$:
\[
Z_1:=\sqrt{\rho} v_1+\sqrt{1-\rho} \varepsilon,\qquad
Z_2:=\sqrt{\rho} v_2+\sqrt{1-\rho} \varepsilon.
\]
Then $Z_1<Z_2$ almost surely since $\sqrt{\rho}(v_2-v_1)>0$ is constant.
Because $f$ is strictly increasing,
$f(Z_1)<f(Z_2)$ almost surely, hence $\mu_\rho(v_1)=\E[f(Z_1)]<\E[f(Z_2)]=\mu_\rho(v_2)$.

\paragraph{Continuity.}
Let $v_n\to v$ and $Z_n=\sqrt\rho v_n+\sigma\varepsilon$ with $\sigma^2=1-\rho<1$, $\varepsilon\sim N(0,1)$, and $|v_n|\le B$. Then $f(Z_n)\to f(Z)$ a.s.\ by continuity of $f$. Setting $M_\rho:=\sqrt\rho B$ and applying Young's inequality $2M_\rho\sigma|\varepsilon|\le M_\rho^2/\delta+\delta\sigma^2\varepsilon^2$ with $\delta:=(1-\sigma^2)/(2\sigma^2)$ gives
\[
  Z_n^2\le (M_\rho+\sigma|\varepsilon|)^2
  \le M_\rho^2(1+1/\delta)+\sigma^2(1+\delta)\varepsilon^2
  =C_{\rho,B}+\tfrac{1+\sigma^2}{2}\varepsilon^2.
\]
Combined with Lemma~\ref{lem:f-growth}\eqref{eq:f-upper}, this provides the $n$-uniform envelope
\[
  |f(Z_n)|\ph(\varepsilon)\le C(1+|\varepsilon|)\exp\left(-\frac{1-\sigma^2}{4}\varepsilon^2\right),
\]
which is integrable. Dominated convergence yields $\mu_\rho(v_n)\to\mu_\rho(v)$.

\paragraph{Existence of the inverse and \eqref{eq:limit-size-vrho}.}
For fixed $\rho\in(0,1)$, $\mu_\rho$ is odd, continuous, and strictly increasing, with $\mu_\rho(0)=0$. To establish $\mu_\rho(v)\to+\infty$ as $v\to+\infty$, fix $M>0$ and $R>0$ with $f(z)\ge 2M$ for $z\ge R$. Writing $Z=\sqrt\rho v+\sigma\varepsilon$ with $m=\sqrt\rho v$,
\[
  \mu_\rho(v)\ge 2M\Pbb(Z\ge R)-\E[|f(Z)|\1\{Z<0\}].
\]
$\Pbb(Z\ge R)\to 1$ since $m\to\infty$. For the second term, the bound $z^2/2-(z-m)^2/(2\sigma^2)\le -(\rho z^2+m^2)/(2\sigma^2)$ for $z\le 0$, $m>0$, together with Lemma~\ref{lem:f-growth}\eqref{eq:f-upper}, gives $\E[|f(Z)|\1\{Z<0\}]=O(e^{-m^2/(2\sigma^2)})\to 0$. Hence $\mu_\rho(v)\to\infty$. By strict monotonicity, $\{\mu_\rho(V)>t\}=\{V>v_\rho(t)\}$ a.s., proving \eqref{eq:limit-size-vrho}.
\end{proof}

The next result gives the precise rate at which $\mu_\rho(v)$ diverges as $\rho\downarrow 0$ for fixed $v\neq 0$: the product $\rho\mu_\rho(v)$ has a finite nonzero limit.

\begin{theorem}[Small-$\rho$ asymptotic of $\mu_\rho$]\label{thm:mu-smallrho}
For each fixed $v\in\R$, we have
\begin{equation}\label{eq:mu-smallrho}
  \lim_{\rho\downarrow 0}\ \rho \mu_\rho(v)=\sqrt{2/\pi} v e^{v^2/2}.
\end{equation}
Moreover, the convergence in \eqref{eq:mu-smallrho} is uniform in $v$ over compact sets.
In particular, for every $v\neq 0$, $\mu_\rho(v)\to \mathrm{sign}(v)\cdot\infty$ as $\rho\downarrow 0$.
\end{theorem}

\begin{proof}
Fix $v\in\R$ and $\rho\in(0,1)$.
Let $\varepsilon\sim N(0,1)$ and set
\[
Z_\rho:=\sqrt{\rho} v+\sqrt{1-\rho} \varepsilon.
\]
Define
\[
g(z):=\sqrt{2/\pi} z e^{z^2/2},\qquad r(z):=f(z)-g(z).
\]
Then we have
\[
\mu_\rho(v)=\E[f(Z_\rho)]=\E[g(Z_\rho)]+\E[r(Z_\rho)].
\]

\paragraph{Leading term.}
Write $m:=\sqrt{\rho} v$ and $\sigma^2:=1-\rho$.
Then $Z_\rho\sim N(m,\sigma^2)$ and $1-\sigma^2=\rho$.
For $Z\sim N(m,\sigma^2)$ and $\lambda<1/(2\sigma^2)$, the standard Gaussian identity gives
\[
  \E[e^{\lambda Z^2}]
  =
  \frac{1}{\sqrt{1-2\lambda\sigma^2}}
  \exp \left(\frac{\lambda m^2}{1-2\lambda\sigma^2}\right).
\]
Applying this with $\lambda=1/2$ yields
\begin{equation}\label{eq:EZexp-half}
\E \left[e^{Z_\rho^2/2}\right]
=\frac{1}{\sqrt{\rho}}\exp \left(\frac{m^2}{2\rho}\right)
=\frac{1}{\sqrt{\rho}}e^{v^2/2}.
\end{equation}
To compute $\E[Z_\rho e^{Z_\rho^2/2}]$, differentiate the same identity with respect to $m$.
Since
\[
  \partial_m \E[e^{Z^2/2}]
  =
  \frac{1}{\sigma^2}\E[(Z-m)e^{Z^2/2}],
\]
we obtain
\[
  \E[Z e^{Z^2/2}]
  =
  m\E[e^{Z^2/2}]
  +
  \sigma^2 \partial_m \E[e^{Z^2/2}]
  =
  \frac{m}{(1-\sigma^2)^{3/2}}
  \exp \left(\frac{m^2}{2(1-\sigma^2)}\right).
\]
Applying this to $Z=Z_\rho$ with $m=\sqrt{\rho} v$ and $1-\sigma^2=\rho$ gives
\[
\E \left[Z_\rho e^{Z_\rho^2/2}\right]
=\frac{\sqrt{\rho} v}{\rho^{3/2}}e^{v^2/2}
=\frac{v}{\rho}e^{v^2/2}.
\]
Hence
\[
\rho \E[g(Z_\rho)]
=\rho\sqrt{2/\pi}\E \left[Z_\rho e^{Z_\rho^2/2}\right]
=\sqrt{2/\pi} v e^{v^2/2}.
\]

\paragraph{Remainder.}
By Lemma~\ref{lem:f-growth}\eqref{eq:f-asymp},
\[
\lim_{|z|\to\infty}\frac{r(z)}{\sqrt{2/\pi} z e^{z^2/2}}=0.
\]
Fix $\epsilon>0$.
Choose $R_\epsilon>0$ such that for all $|z|\ge R_\epsilon$,
\begin{equation}\label{eq:r-eps}
|r(z)|\le \epsilon \sqrt{2/\pi} |z|e^{z^2/2}.
\end{equation}
Then we have
\[
\rho |\E[r(Z_\rho)]|
\le \rho \E \left[|r(Z_\rho)| \1\{|Z_\rho|<R_\epsilon\}\right]
+\rho \E \left[|r(Z_\rho)| \1\{|Z_\rho|\ge R_\epsilon\}\right].
\]
The first term is bounded by
\[
\rho \sup_{|z|<R_\epsilon}|r(z)| \ \xrightarrow[\rho\downarrow 0]{}\ 0
\]
since the supremum is finite and does not depend on $\rho$.
For the second term, apply \eqref{eq:r-eps}:
\[
\rho \E \left[|r(Z_\rho)| \1\{|Z_\rho|\ge R_\epsilon\}\right]
\le \epsilon \sqrt{2/\pi}\rho \E \left[|Z_\rho|e^{Z_\rho^2/2}\right].
\]
To bound $\rho \E[|Z_\rho|e^{Z_\rho^2/2}]$, use Cauchy--Schwarz as in the proof of Lemma~\ref{lem:mu-props}:
\[
\E \left[|Z_\rho|e^{Z_\rho^2/2}\right]
\le \sqrt{\E \left[Z_\rho^2e^{Z_\rho^2/2}\right]} \sqrt{\E \left[e^{Z_\rho^2/2}\right]}.
\]
As above, $Z_\rho\sim N(m,\sigma^2)$ with $m=\sqrt{\rho} v$ and $\sigma^2=1-\rho$.
Differentiating the same Gaussian identity in $\lambda$ and evaluating at $\lambda=1/2$ yields
\[
\E \left[Z_\rho^2e^{Z_\rho^2/2}\right]
=\left(\frac{\sigma^2}{1-\sigma^2}+\frac{m^2}{(1-\sigma^2)^2}\right)\E \left[e^{Z_\rho^2/2}\right]
=\left(\frac{1-\rho}{\rho}+\frac{\rho v^2}{\rho^2}\right)\frac{1}{\sqrt{\rho}}e^{v^2/2}
=\frac{1-\rho+v^2}{\rho^{3/2}}e^{v^2/2}.
\]
Together with \eqref{eq:EZexp-half}, this gives
\[
\E \left[|Z_\rho|e^{Z_\rho^2/2}\right]
\le \frac{\sqrt{1-\rho+v^2}}{\rho}e^{v^2/2}.
\]
Therefore, we have
\[
\limsup_{\rho\downarrow 0}\ \rho |\E[r(Z_\rho)]|
\le \epsilon \sqrt{2/\pi}\sqrt{1+v^2} e^{v^2/2}.
\]
Since $\epsilon>0$ was arbitrary, we conclude $\rho \E[r(Z_\rho)]\to 0$ as $\rho\downarrow 0$.

Combining the leading-term and remainder estimates yields \eqref{eq:mu-smallrho} for each fixed $v$.

\paragraph{Uniformity on compact $v$-sets.}
Let $|v|\le V_0$.
The bound $\rho \E[|Z_\rho|e^{Z_\rho^2/2}]\le \sqrt{1+V_0^2} e^{V_0^2/2}$ is uniform over $|v|\le V_0$ by the preceding calculation.
Hence the same $\epsilon$--$R_\epsilon$ argument gives uniform convergence on compact sets.

Finally, if $v\neq 0$, \eqref{eq:mu-smallrho} implies
\[
  |\mu_\rho(v)|\sim \frac{\sqrt{2/\pi}}{\rho}|v|e^{v^2/2}\to\infty,
\]
and the sign matches $\mathrm{sign}(v)$ because the leading constant is positive.
\end{proof}

\begin{proof}[Proof of Proposition~\ref{prop:noncommuting-main}]
Fix $\alpha\in(0,1/2)$, let $t_\alpha=\cot(\pi\alpha)$, and write
\[
  s_{\alpha,K}(\rho):=\Pbb_\rho(T_K>t_\alpha).
\]

\medskip\noindent\textit{Fixed-$K$ limit.}
Let $Z^{(\rho)}=(Z_1^{(\rho)},\dots,Z_K^{(\rho)})$ denote the equicorrelated Gaussian vector with correlation $\rho$, i.e.\ $Z^{(\rho)}\sim N(0,\Sigma_\rho)$ with $\Sigma_\rho=(1-\rho)I_K+\rho\mathbf{1}\mathbf{1}^\top$.
As $\rho\downarrow 0$, $\Sigma_\rho\to I_K$, hence $Z^{(\rho)}\dto Z^{(0)}$ where $Z^{(0)}\sim N(0,I_K)$ has independent standard normal coordinates.

Since $f$ is continuous on $\R$, the map
\[
  z\mapsto \frac1K\sum_{i=1}^K f(z_i)
\]
is continuous on $\R^K$. By the continuous mapping theorem, we therefore have
\[
  T_K=\frac1K\sum_{i=1}^K f\left(Z_i^{(\rho)}\right)
  \dto
  \frac1K\sum_{i=1}^K f\left(Z_i^{(0)}\right)
  =:T_K^{(0)}
  \qquad\text{as }\rho\downarrow 0.
\]

When $\rho=0$, $U_i:=\Ph(Z_i^{(0)})$ are i.i.d.\ $\mathrm{Unif}(0,1)$ and
\[
  X_i=f(Z_i^{(0)})=\tan(\pi(U_i-1/2))
\]
are i.i.d.\ standard Cauchy. By $1$-stability of the Cauchy law,
\[
  T_K^{(0)}=\frac1K\sum_{i=1}^K X_i
\]
is again standard Cauchy. Therefore
\[
  \Pbb(T_K^{(0)}>t_\alpha)=\alpha,
\]
since $t_\alpha=\cot(\pi\alpha)$ is the upper $\alpha$-quantile of the standard Cauchy distribution.
Finally, because the standard Cauchy law is continuous, $\Pbb(T_K^{(0)}=t_\alpha)=0$, so convergence in distribution implies convergence of tail probabilities:
\[
\lim_{\rho\downarrow 0}\Pbb_\rho(T_K>t_\alpha)=\Pbb(T_K^{(0)}>t_\alpha)=\alpha.
\]

\medskip\noindent\textit{Iterated limit.}
For each fixed $\rho\in(0,1)$, Theorem~\ref{thm:randomlimit} and continuity of the law of $\mu_\rho(V)$ from Lemma~\ref{lem:mu-props} imply
\[
\lim_{K\to\infty}s_{\alpha,K}(\rho)=\Pbb\left(\mu_\rho(V)>t_\alpha\right)=:s_\infty(\rho).
\]
Let $v_\rho>0$ be the unique solution of $\mu_\rho(v_\rho)=t_\alpha$ from Lemma~\ref{lem:mu-props}. Then we have
\[
  s_\infty(\rho)=\barPh(v_\rho).
\]
Fix any $\delta>0$. By Theorem~\ref{thm:mu-smallrho}, $\mu_\rho(\delta)\to\infty$ as $\rho\downarrow0$, so for all sufficiently small $\rho$ we have $\mu_\rho(\delta)>t_\alpha$. Since $\mu_\rho$ is increasing, this forces $v_\rho<\delta$. As $\delta>0$ was arbitrary, $v_\rho\to0$, and therefore we have
\[
  s_\infty(\rho)=\barPh(v_\rho)\to\barPh(0)=\frac12.
\]
Combining these yields
\[
\lim_{\rho\downarrow 0}\ \lim_{K\to\infty}s_{\alpha,K}(\rho)=\frac12.
\]
\end{proof}

\section{Proofs in Section~\ref{sec:raw cct}: weakening dependence}\label{app:rawboundary}

This section proves the results for the triangular regime $\rho_K\downarrow0$: Theorem~\ref{thm:stablelimit} and Corollaries~\ref{cor:phasediagram}--\ref{cor:iffexact}.

\subsection{Deviation of the statistic from its conditional mean}\label{subsec:devTK}

Fix $\rho\in(0,1)$ and recall that, conditional on $V=v$, the $X_i=f(Z_i)$ are i.i.d.\ and
\[
T_K-\mu_\rho(v)=\frac1K\sum_{i=1}^K \left(X_i-\mu_\rho(v)\right)
\qquad\text{under }(V=v).
\]
Write $Y_i:=X_i-\mu_\rho(v)$ for the centered variables under $(V=v)$.

\begin{proposition}\label{prop:finiteK-dev}
Fix $\rho\in(0,1)$ and $v\in\R$.
Fix any $q$ satisfying
\[
1<q<\min\left\{2,\frac{1}{1-\rho}\right\}.
\]
Then $\E[|Y_1|^q \mid V=v]<\infty$ and, for every $\delta>0$,
\begin{align}
\E\left[|T_K-\mu_\rho(v)|^q \mid V=v\right]
&\le 2K^{1-q} \E\left[|Y_1|^q \mid V=v\right],\label{eq:TK-qmoment}\\
\Pbb\left(|T_K-\mu_\rho(v)|>\delta \mid V=v\right)
&\le 2\delta^{-q} K^{1-q} \E\left[|Y_1|^q \mid V=v\right].\label{eq:TK-dev}
\end{align}
\end{proposition}

\begin{proof}
By Proposition~\ref{prop:uniform-qmoment} (with $M>|v|$), $\E[|X_1|^q\mid V=v]<\infty$, and the bound $|Y_1|^q\le 2^{q-1}(|X_1|^q+|\mu_\rho(v)|^q)$ with Jensen yields $\E[|Y_1|^q\mid V=v]<\infty$. Since $1<q<2$, the von Bahr--Esseen inequality \cite{vonBahrEsseen1965} gives $\E[|\sum_i Y_i|^q\mid V=v]\le 2K\E[|Y_1|^q\mid V=v]$, hence \eqref{eq:TK-qmoment}; Markov's inequality applied to $|T_K-\mu_\rho(v)|^q$ then yields \eqref{eq:TK-dev}.
\end{proof}

\subsection{Uniform control on bounded factor events}\label{subsec:uniformM}

To convert the conditional bound in Proposition~\ref{prop:finiteK-dev} into an unconditional statement, we bound the conditional $q$-moment $\E[|Y_1|^q \mid V=v]$ uniformly over $|v|\le M$.

For $q>0$ and $\eta>0$, write
\[
  A_{q,\eta}:=\sup_{z\in\R}(1+|z|)^q e^{-\eta z^2}.
\]
Since the Gaussian factor dominates polynomial growth, we have $A_{q,\eta}<\infty$ and hence
\[
  (1+|z|)^q\le A_{q,\eta} e^{\eta z^2}
  \qquad\text{for all }z\in\R.
\]

\begin{proposition}\label{prop:uniform-qmoment}
Fix $\rho\in(0,1)$, set $\sigma^2=1-\rho\in(0,1)$ and $\theta=1/\sigma^2$.
Let $q\in[1,\theta)$ and $M>0$.
Then we have
\[
\sup_{|v|\le M}\E\left[|X_1|^q \mid V=v\right] < \infty,
\qquad
\sup_{|v|\le M}\E\left[|Y_1|^q \mid V=v\right] < \infty,
\]
where $X_1=f(Z_1)$ and $Y_1=X_1-\mu_\rho(v)$ under $(V=v)$.
More explicitly, letting $C_f$ be the constant in Lemma~\ref{lem:f-growth}\eqref{eq:f-upper} and setting
\[
\eta:=\frac{\theta-q}{4}>0,
\qquad
\lambda:=\frac{q}{2}+\eta=\frac{q+\theta}{4},
\]
we have for all $|v|\le M$,
\begin{equation}\label{eq:qmoment-bound}
\E\left[|X_1|^q \mid V=v\right]
\le C_f^q A_{q,\eta} 
\frac{1}{\sqrt{1-2\lambda\sigma^2}} 
\exp \left(\frac{\lambda\rho v^2}{1-2\lambda\sigma^2}\right),
\end{equation}
and consequently
\begin{equation}\label{eq:qmoment-bound-Y}
\E\left[|Y_1|^q \mid V=v\right]\le 2^q \E\left[|X_1|^q \mid V=v\right]
\le 2^q C_f^q A_{q,\eta} 
\frac{1}{\sqrt{1-2\lambda\sigma^2}} 
\exp \left(\frac{\lambda\rho M^2}{1-2\lambda\sigma^2}\right).
\end{equation}
\end{proposition}

\begin{proof}
Fix $\rho\in(0,1)$, $q\in[1,\theta)$ and $M>0$.
Let $v\in[-M,M]$ and consider the conditional law $(V=v)$.
Then $Z_1\sim N(m,\sigma^2)$ with $m=\sqrt{\rho} v$ and $\sigma^2=1-\rho$.
By Lemma~\ref{lem:f-growth}\eqref{eq:f-upper},
\[
|X_1|^q = |f(Z_1)|^q \le C_f^q(1+|Z_1|)^q e^{qZ_1^2/2}.
\]
By the definition of $A_{q,\eta}$,
\[
  (1+|Z_1|)^q\le A_{q,\eta}e^{\eta Z_1^2}.
\]
Therefore,
\[
|X_1|^q \le C_f^q A_{q,\eta}  e^{(q/2+\eta)Z_1^2}
= C_f^q A_{q,\eta}  e^{\lambda Z_1^2}.
\]
Taking conditional expectation and using the standard Gaussian identity
\[
  \E[e^{\lambda Z_1^2}]
  =
  \frac{1}{\sqrt{1-2\lambda\sigma^2}}
  \exp \left(\frac{\lambda m^2}{1-2\lambda\sigma^2}\right),
  \qquad 2\lambda\sigma^2<1,
\]
we obtain
\[
\E\left[|X_1|^q \mid V=v\right]
\le C_f^q A_{q,\eta} 
\frac{1}{\sqrt{1-2\lambda\sigma^2}}
\exp \left(\frac{\lambda m^2}{1-2\lambda\sigma^2}\right).
\]
Since $m^2=\rho v^2$, this is exactly \eqref{eq:qmoment-bound}.
Moreover, the right-hand side is finite because $2\lambda\sigma^2<1$:
indeed,
\[
2\lambda\sigma^2
=2\cdot\frac{q+\theta}{4}\cdot\frac{1}{\theta}
=\frac{q+\theta}{2\theta}<1
\qquad\text{since }q<\theta.
\]
Taking the supremum over $|v|\le M$ yields the first claimed uniform bound.

For the centered variable $Y_1=X_1-\mu_\rho(v)$, the restriction $q\ge1$ is important: it lets us use the standard inequality
\[
|Y_1|^q = |X_1-\mu_\rho(v)|^q \le 2^{q-1}\left(|X_1|^q+|\mu_\rho(v)|^q\right).
\]
By Jensen, $|\mu_\rho(v)|^q\le \E[|X_1|^q \mid V=v]$.
Hence
\[
\E[|Y_1|^q \mid V=v] \le 2^{q-1}\left(\E[|X_1|^q \mid V=v]+\E[|X_1|^q \mid V=v]\right)
=2^q \E[|X_1|^q \mid V=v],
\]
which gives \eqref{eq:qmoment-bound-Y} after applying \eqref{eq:qmoment-bound} and bounding $v^2\le M^2$.
\end{proof}

\subsection{Uniform calibration of the conditional $p$-value tail}\label{subsec:size-sandwich}

The factor $\lambda_K(v)$ defined below accounts for the multiplicative distortion of the conditional p-value tail $\Pbb(P_{i,K}^{(v)}\le u)$ relative to $u$ at the scale $u\asymp 1/K$.

Fix $v\in\R$ throughout this section.
Conditionally on $V=v$, define
\[
  Z_{i,K}^{(v)} := \sqrt{\rho_K} v + \sqrt{1-\rho_K} \varepsilon_i,\qquad
  P_{i,K}^{(v)} := \barPh \left(Z_{i,K}^{(v)}\right),\qquad i=1,\dots,K,
\]
where $\varepsilon_1,\varepsilon_2,\dots$ are i.i.d.\ $N(0,1)$.
Set
\begin{equation}\label{eq:def-lambdaK-v}
  \lambda_K(v) := \exp \left(-c_K+\sqrt{2c_K} v\right).
\end{equation}
Since $c_K\to c$, we have
\begin{equation}\label{eq:lambdaK-limit}
  \lambda_K(v)\to \lambda(v):=\exp \left(-c+\sqrt{2c} v\right)\in(0,\infty).
\end{equation}

\begin{proposition}\label{prop:tri-calibration}
Assume $\rho_K\downarrow 0$ and $c_K:=\rho_K\log K\to c\in[0,\infty)$.
Fix $v\in\R$ and let $\lambda_K(v)$ be defined by \eqref{eq:def-lambdaK-v}.
For $x>0$, define
\[
  u_{K,x}(v):=\frac{x}{\lambda_K(v) K},
  \qquad
  q_{K,x}(v):=\barPh^{-1} \left(u_{K,x}(v)\right).
\]
Then, for every compact interval $I=[a,b]\subset(0,\infty)$,
\begin{equation}\label{eq:uniform-tail-calibration}
  \sup_{x\in I}
  \left|
  K \Pbb \left(P_{1,K}^{(v)}\le \frac{x}{\lambda_K(v) K} \mid V=v\right)-x
  \right|
  \to 0.
\end{equation}
\end{proposition}

\begin{proof}
Fix a compact interval $I=[a,b]\subset(0,\infty)$.
For each $x\in I$, set
\[
  u_{K,x}:=u_{K,x}(v)=\frac{x}{\lambda_K(v) K},
  \qquad
  q_{K,x}:=q_{K,x}(v)=\barPh^{-1}(u_{K,x}).
\]
Since $\lambda_K(v)\to \lambda(v)\in(0,\infty)$ by \eqref{eq:lambdaK-limit}, there exist constants $0<\lambda_-<\lambda_+<\infty$ and $K_0\in\mathbb N$ such that
\[
  \lambda_- \le \lambda_K(v)\le \lambda_+
\]
for all $K\ge K_0$.
Consequently, we have
\[
  \frac{a}{\lambda_+ K}\le u_{K,x}\le \frac{b}{\lambda_- K}
\]
for all $x\in I$ and $K\ge K_0$.
In particular, $u_{K,x}\downarrow 0$ uniformly in $x\in I$ as $K\to\infty$.
Since $\barPh$ is strictly decreasing on $(0,\infty)$, it follows that
\[
  q_{K,x}\to\infty
\]
uniformly in $x\in I$.

We first show that
\begin{equation}\label{eq:q-asymp-logK}
  \sup_{x\in I}\left|\frac{q_{K,x}}{\sqrt{2\log K}}-1\right|\to 0.
\end{equation}
By the standard Mills ratio inequalities, for every $x\in I$ and every $K\ge K_0$,
\[
  \frac{\ph(q_{K,x})}{q_{K,x}+q_{K,x}^{-1}}
  \le
  u_{K,x}
  \le
  \frac{\ph(q_{K,x})}{q_{K,x}}.
\]
Using $\ph(t)=(2\pi)^{-1/2}e^{-t^2/2}$, this becomes
\begin{equation}\label{eq:q-mills-two-sided}
  \frac{1}{\sqrt{2\pi}}
  \frac{e^{-q_{K,x}^2/2}}{q_{K,x}+q_{K,x}^{-1}}
  \le
  u_{K,x}
  \le
  \frac{1}{\sqrt{2\pi}}
  \frac{e^{-q_{K,x}^2/2}}{q_{K,x}}.
\end{equation}
Taking logarithms of \eqref{eq:q-mills-two-sided} and using $u_{K,x}\in[a/(\lambda_+K),b/(\lambda_-K)]$, we obtain
\[
  \frac{q_{K,x}^2}{2}=\log K-\log q_{K,x}+O(1)
  \qquad\text{uniformly in }x\in I.
\]
Since $q_{K,x}\to\infty$ uniformly, a one-step bootstrap gives $q_{K,x}^2=2\log K+O(\log\log K)$ uniformly, hence \eqref{eq:q-asymp-logK}.

Now define
\[
  m_K:=\sqrt{\rho_K} v,\qquad \sigma_K:=\sqrt{1-\rho_K},
  \qquad
  y_{K,x}:=\frac{q_{K,x}-m_K}{\sigma_K}.
\]
Since $Z_{1,K}^{(v)} \mid V=v\sim N(m_K,\sigma_K^2)$, we have
\begin{equation}\label{eq:cond-prob-y}
  \Pbb \left(P_{1,K}^{(v)}\le \frac{x}{\lambda_K(v) K} \mid V=v\right)
  =
  \Pbb \left(Z_{1,K}^{(v)}\ge q_{K,x} \mid V=v\right)
  =
  \barPh(y_{K,x}).
\end{equation}

We next compare $y_{K,x}$ with $q_{K,x}$.
Since
\[
  \frac{1}{\sigma_K}-1
  = \frac{1-\sigma_K}{\sigma_K}
  = \frac{\rho_K}{\sigma_K(1+\sigma_K)},
\]
we obtain
\[
  y_{K,x}-q_{K,x}
  =
  q_{K,x} \left(\frac{1}{\sigma_K}-1\right) - \frac{m_K}{\sigma_K}
  =
  \frac{\rho_K q_{K,x}}{\sigma_K(1+\sigma_K)} - \frac{\sqrt{\rho_K} v}{\sigma_K}.
\]
Because $\sigma_K\to 1$, $q_{K,x}=O(\sqrt{\log K})$ uniformly in $x$, and $\rho_K\log K=c_K\to c$, it follows that
\[
  \sup_{x\in I}\left|\frac{\rho_K q_{K,x}}{\sigma_K(1+\sigma_K)}\right|
  \le C \rho_K \sqrt{\log K}
  = C \frac{c_K}{\sqrt{\log K}}
  \to 0,
\]
while
\[
  \left|\frac{\sqrt{\rho_K} v}{\sigma_K}\right|
  \le C_v \sqrt{\rho_K}\to 0.
\]
Therefore,
\begin{equation}\label{eq:y-minus-q}
  \sup_{x\in I}|y_{K,x}-q_{K,x}|\to 0.
\end{equation}
Since $q_{K,x}\to\infty$ uniformly, \eqref{eq:y-minus-q} implies
\begin{equation}\label{eq:q-over-y}
  \sup_{x\in I}\left|\frac{q_{K,x}}{y_{K,x}}-1\right|\to 0.
\end{equation}

We now derive the exponential factor.
A direct calculation gives
\begin{align}
  y_{K,x}^2-q_{K,x}^2
  &=
  \frac{(q_{K,x}-m_K)^2}{1-\rho_K}-q_{K,x}^2 \notag\\
  &=
  \frac{q_{K,x}^2-2m_K q_{K,x}+m_K^2-(1-\rho_K)q_{K,x}^2}{1-\rho_K} \notag\\
  &=
  \frac{\rho_K q_{K,x}^2 - 2m_K q_{K,x}+m_K^2}{1-\rho_K}. \label{eq:y2-minus-q2}
\end{align}
We claim that
\begin{equation}\label{eq:exp-match}
  \sup_{x\in I}
  \left|
  -\frac{y_{K,x}^2-q_{K,x}^2}{2}
  -\log \lambda_K(v)
  \right|
  \to 0.
\end{equation}
Indeed, by \eqref{eq:q-asymp-logK},
\[
  \sup_{x\in I} \left|\rho_K q_{K,x}^2 - 2c_K\right|
  =
  \rho_K  \sup_{x\in I} \left|q_{K,x}^2-2\log K\right|
  \le
  \rho_K  O(\log\log K)
  =
  c_K \frac{O(\log\log K)}{\log K}
  \to 0.
\]
Next,
\[
  m_K q_{K,x}
  =
  \sqrt{\rho_K} v  q_{K,x}
  =
  v \sqrt{\rho_K q_{K,x}^2}.
\]
Since $\rho_K q_{K,x}^2\to 2c$ uniformly in $x\in I$, continuity of the square root yields
\[
  \sup_{x\in I}\left|m_K q_{K,x} - \sqrt{2c_K} v\right|\to 0.
\]
Also,
\[
  m_K^2 = \rho_K v^2 \to 0,
\qquad
  \frac{1}{1-\rho_K}\to 1.
\]
Substituting these facts into \eqref{eq:y2-minus-q2}, we obtain
\[
  \sup_{x\in I}
  \left|
  \frac{y_{K,x}^2-q_{K,x}^2}{2}
  -\left(c_K-\sqrt{2c_K} v\right)
  \right|
  \to 0.
\]
Since $\log\lambda_K(v)=-c_K+\sqrt{2c_K} v$, this proves \eqref{eq:exp-match}.

Finally, apply the standard Mills ratio inequalities both at $q_{K,x}$ and at $y_{K,x}$.
Define the Mills-ratio remainder
\[
  R(z):=\frac{\barPh(z) z}{\ph(z)},\qquad z>0.
\]
These inequalities imply
\[
  \frac{z^2}{1+z^2}\le R(z)\le 1,
\]
so in particular
\[
  \sup_{z\ge m}|R(z)-1|\le \frac{1}{1+m^2}
  \qquad (m>0).
\]
Because $q_{K,x}\to\infty$ and $y_{K,x}\to\infty$ uniformly on $I$, the quantity
\[
  \underline m_K:=\inf_{x\in I}\min\{q_{K,x},y_{K,x}\}
\]
satisfies $\underline m_K\to\infty$. Hence
\[
  \sup_{x\in I}\left|\frac{R(y_{K,x})}{R(q_{K,x})}-1\right|\to 0.
\]
Using the identity $\barPh(z)=\ph(z)R(z)/z$, we therefore obtain
\[
  \frac{\barPh(y_{K,x})}{\barPh(q_{K,x})}
  =
  \frac{\ph(y_{K,x})}{\ph(q_{K,x})} \frac{q_{K,x}}{y_{K,x}} \frac{R(y_{K,x})}{R(q_{K,x})}
  =
  \exp \left(-\frac{y_{K,x}^2-q_{K,x}^2}{2}\right)\frac{q_{K,x}}{y_{K,x}} (1+o(1))
\]
uniformly in $x\in I$.
Using \eqref{eq:q-over-y} and \eqref{eq:exp-match}, we conclude that
\[
  \sup_{x\in I}
  \left|
  \frac{\barPh(y_{K,x})}{\barPh(q_{K,x})} - \lambda_K(v)
  \right|
  \to 0.
\]
Since $\barPh(q_{K,x})=u_{K,x}=x/(\lambda_K(v)K)$, we obtain
\[
  \barPh(y_{K,x})
  =
  \frac{x}{K} (1+o(1))
\]
uniformly in $x\in I$.
Combining this with \eqref{eq:cond-prob-y} proves \eqref{eq:uniform-tail-calibration}.
\end{proof}

\subsection{Point-process approximation for Theorem 3.3}

Proposition~\ref{prop:tri-calibration} controls the upper tail of $P_{i,K}^{(v)}$, which governs large positive scores. To describe the full $1$-stable limit, we also need the lower tail, which governs large negative scores.
For this purpose write
\begin{equation*}
  \lambda_K^{+}(v):=\lambda_K(v)=\exp \left(-c_K+\sqrt{2c_K} v\right),
  \qquad
  \lambda_K^{-}(v):=\exp \left(-c_K-\sqrt{2c_K} v\right),
\end{equation*}
and
\begin{equation*}
  \lambda^{+}(v):=\exp \left(-c+\sqrt{2c} v\right),
  \qquad
  \lambda^{-}(v):=\exp \left(-c-\sqrt{2c} v\right).
\end{equation*}
Also write
\[
  Q_{i,K}^{(v)}:=\Ph \left(Z_{i,K}^{(v)}\right)=1-P_{i,K}^{(v)}.
\]
Large positive CCT scores correspond to small values of $P_{i,K}^{(v)}$, whereas large negative CCT scores correspond to small values of $Q_{i,K}^{(v)}$.

The lower-tail estimates follow by symmetry, because
\[
  Q_{1,K}^{(v)}=\barPh\left(-Z_{1,K}^{(v)}\right)
\]
and $-Z_{1,K}^{(v)}$ has the same law as $Z_{1,K}^{(-v)}$ under $(V=v)$.

Recall that
\[
  X_{i,K}^{(v)}:=f \left(Z_{i,K}^{(v)}\right)
  =\tan \left(\pi(\Ph(Z_{i,K}^{(v)})-1/2)\right).
\]
When $X_{i,K}^{(v)}>0$ we have, by \eqref{eq:f-cot},
\[
  X_{i,K}^{(v)}=\cot \left(\pi P_{i,K}^{(v)}\right),
\]
whereas when $X_{i,K}^{(v)}<0$,
\[
  X_{i,K}^{(v)}
  =
  \tan \left(\pi(Q_{i,K}^{(v)}-1/2)\right)
  =
  -\cot \left(\pi Q_{i,K}^{(v)}\right).
\]

\begin{lemma}\label{lem:score-interval-prob}
Assume $\rho_K\downarrow 0$ and $c_K:=\rho_K\log K\to c\in[0,\infty)$.
Fix $v\in\R$.
Let $0<r<s<\infty$.
Then, conditionally on $V=v$,
\begin{equation}\label{eq:score-interval-positive}
  K \Pbb \left(r<\frac{X_{1,K}^{(v)}}{K}\le s \mid V=v\right)
  \to
  \frac{\lambda^{+}(v)}{\pi}\left(\frac1r-\frac1s\right),
\end{equation}
and
\begin{equation}\label{eq:score-interval-negative}
  K \Pbb \left(-s\le\frac{X_{1,K}^{(v)}}{K}< -r \mid V=v\right)
  \to
  \frac{\lambda^{-}(v)}{\pi}\left(\frac1r-\frac1s\right).
\end{equation}
In addition,
\begin{equation}\label{eq:score-tail-positive}
  K \Pbb \left(\frac{X_{1,K}^{(v)}}{K}> M \mid V=v\right)
  \to
  \frac{\lambda^{+}(v)}{\pi M},
\end{equation}
and
\begin{equation}\label{eq:score-tail-negative}
  K \Pbb \left(\frac{X_{1,K}^{(v)}}{K}< -M \mid V=v\right)
  \to
  \frac{\lambda^{-}(v)}{\pi M}
\end{equation}
for every $M>0$.
\end{lemma}

\begin{proof}
We first prove \eqref{eq:score-interval-positive}.
Because $Z_{1,K}^{(v)} \mid V=v$ has a Gaussian density and $\bar\Ph$, $\Ph$ are smooth monotone maps,
the conditional laws of $P_{1,K}^{(v)}$ and $Q_{1,K}^{(v)}$ are continuous. In particular,
all threshold events below have the same conditional probability whether written with $<$ or $\le$.
On the event $\{X_{1,K}^{(v)}>0\}$ we have
\[
  X_{1,K}^{(v)}=\cot \left(\pi P_{1,K}^{(v)}\right),
\]
and the function $u\mapsto \cot(\pi u)$ is strictly decreasing on $(0,1/2)$.
Therefore
\begin{align*}
  \left\{r<\frac{X_{1,K}^{(v)}}{K}\le s\right\}
  &=
  \left\{Kr<\cot \left(\pi P_{1,K}^{(v)}\right)\le Ks\right\} \\
  &=
  \left\{\frac{1}{\pi}\operatorname{arccot}(Ks)\le P_{1,K}^{(v)} < \frac{1}{\pi}\operatorname{arccot}(Kr)\right\}.
\end{align*}
Multiplying by $\lambda_K^{+}(v)K$ gives
\begin{align}
  &\left\{r<\frac{X_{1,K}^{(v)}}{K}\le s\right\} \notag\\
  &\qquad=
  \left\{
    a_{K}^{+}(s)
    \le
    \lambda_K^{+}(v)K P_{1,K}^{(v)}
    <
    a_{K}^{+}(r)
  \right\}, \label{eq:score-positive-interval-Y}
\end{align}
where
\[
  a_K^{+}(x):=\frac{\lambda_K^{+}(v)K}{\pi}\operatorname{arccot}(Kx),\qquad x>0.
\]
We claim that
\begin{equation}\label{eq:akplus-limit}
  a_K^{+}(x)\to \frac{\lambda^{+}(v)}{\pi x}\qquad\text{for every }x>0.
\end{equation}
Indeed, for $y\to\infty$,
\[
  \operatorname{arccot}(y)=\arctan(1/y)=\frac1y+O \left(\frac{1}{y^3}\right),
\]
because $\arctan u = u + O(u^3)$ as $u\to 0$.
Substituting $y=Kx$ yields
\[
  K \operatorname{arccot}(Kx)=\frac1x+O \left(\frac1{K^2}\right),
\]
hence
\[
  a_K^{+}(x)
  =
  \frac{\lambda_K^{+}(v)}{\pi}
  \left(
    \frac1x+O \left(\frac1{K^2}\right)
  \right)
  \to
  \frac{\lambda^{+}(v)}{\pi x},
\]
which proves \eqref{eq:akplus-limit}.

Fix $x\in\{r,s\}$.
Since $a_K^{+}(x)\to \lambda^{+}(v)/(\pi x)\in(0,\infty)$, there exists a compact interval $I_x\subset(0,\infty)$ such that $a_K^{+}(x)\in I_x$ for all sufficiently large $K$.
Applying Proposition~\ref{prop:tri-calibration} at the point $a_K^{+}(x)$ gives
\[
  K \Pbb \left(\lambda_K^{+}(v)K P_{1,K}^{(v)}\le a_K^{+}(x) \mid V=v\right)
  =
  a_K^{+}(x)+o(1).
\]
By \eqref{eq:score-positive-interval-Y},
\begin{align*}
  K \Pbb \left(r<\frac{X_{1,K}^{(v)}}{K}\le s \mid V=v\right)
  &=
  K \Pbb \left(\lambda_K^{+}(v)K P_{1,K}^{(v)}< a_K^{+}(r) \mid V=v\right)\\
  &\quad-
  K \Pbb \left(\lambda_K^{+}(v)K P_{1,K}^{(v)}< a_K^{+}(s) \mid V=v\right).
\end{align*}
Using the previous display at $x=r$ and $x=s$, we obtain
\[
  K \Pbb \left(r<\frac{X_{1,K}^{(v)}}{K}\le s \mid V=v\right)
  =
  a_K^{+}(r)-a_K^{+}(s)+o(1).
\]
Taking the limit and using \eqref{eq:akplus-limit} proves \eqref{eq:score-interval-positive}.

The negative interval statement follows from the same calculation applied to
$Q_{1,K}^{(v)}$ and Proposition~\ref{prop:tri-calibration} with $-v$ in place of $v$.
Equivalently, conditionally on $V=v$, the variable $Q_{1,K}^{(v)}=\barPh(-Z_{1,K}^{(v)})$ has the same distribution as $P_{1,K}^{(-v)}$ under the conditional law with latent value $-v$.
Indeed, on $\{X_{1,K}^{(v)}<0\}$ we have
\[
  X_{1,K}^{(v)}=-\cot \left(\pi Q_{1,K}^{(v)}\right),
\]
so the same arccot transformation yields \eqref{eq:score-interval-negative} with
$\lambda^{-}(v)=\lambda(-v)$.

The one-sided tail limits \eqref{eq:score-tail-positive} and \eqref{eq:score-tail-negative}
follow in exactly the same way by replacing the interval $(r,s]$ with a single threshold event.
\end{proof}

For $\eta>0$, let
\[
  B_\eta:=(-\infty,-\eta]\cup[\eta,\infty).
\]
Define the finite measure on $B_\eta$ by
\begin{equation}\label{eq:def-nu-v}
  \Lambda_v(dx)
  :=
  \frac{\lambda^{+}(v)}{\pi x^2} \1\{x>0\} dx
  +
  \frac{\lambda^{-}(v)}{\pi x^2} \1\{x<0\} dx.
\end{equation}

For $\eta>0$, define the signed extreme-score sum
\begin{equation}\label{eq:def-Seta}
  S_{K,\eta}(v)
  :=
  \sum_{i=1}^K \frac{X_{i,K}^{(v)}}{K} \1 \left\{\left|\frac{X_{i,K}^{(v)}}{K}\right|>\eta\right\}.
\end{equation}
This sum keeps only the score contributions whose size is at least $\eta$ after the natural $1/K$ scaling.

\begin{proposition}\label{prop:truncated-sum-limit}
Assume $\rho_K\downarrow 0$ and $c_K:=\rho_K\log K\to c\in[0,\infty)$.
Fix $v\in\R$ and $\eta>0$.
Then, conditionally on $V=v$,
\[
  S_{K,\eta}(v)\dto S_{\eta,v},
\]
where $S_{\eta,v}$ is an infinitely divisible random variable with characteristic function
\begin{equation}\label{eq:truncated-sum-cf}
  \E \left[e^{it S_{\eta,v}}\right]
  =
  \exp \left\{
    \int_{|x|>\eta} (e^{itx}-1) \Lambda_v(dx)
  \right\},
  \qquad t\in\R.
\end{equation}
Equivalently, $S_{\eta,v}$ is the compound-Poisson sum generated by a Poisson point process on $B_\eta$ with intensity measure $\Lambda_v$.
\end{proposition}

\begin{proof}
Fix $t\in\R$.
For $M>\eta$, define the additionally truncated sum
\[
  S_{K,\eta,M}(v)
  :=
  \sum_{i=1}^K \frac{X_{i,K}^{(v)}}{K} 
  \1 \left\{\eta<\left|\frac{X_{i,K}^{(v)}}{K}\right|\le M\right\}.
\]
We first identify the limit of $S_{K,\eta,M}(v)$ for fixed $M$, and then remove the upper cutoff $M$.

\paragraph{Bounded truncation.}
Let
\[
  C_{\eta,M}:=[-M,-\eta]\cup[\eta,M],
  \qquad
  \widetilde h_{t,\eta,M}(x):=e^{itx}-1,
  \quad x\in C_{\eta,M}.
\]
Because the conditional law of $X_{1,K}^{(v)}/K$ is absolutely continuous and $\Lambda_v$ has a density, both $\mu_{K,v}$ and $\Lambda_v$ assign zero mass to every singleton in $C_{\eta,M}$. In particular, changing endpoint values at $\{\pm\eta,\pm M\}$ does not alter any of the integrals below, so the original cutoff integrand $x\mapsto (e^{itx}-1)\1\{\eta<|x|\le M\}$ may be replaced by the continuous representative $\widetilde h_{t,\eta,M}$ on $C_{\eta,M}$.
The function $\widetilde h_{t,\eta,M}$ is bounded and uniformly continuous on the compact set $C_{\eta,M}$.
Let
\[
  \mu_{K,v}(A):=
  K \Pbb \left(\frac{X_{1,K}^{(v)}}{K}\in A \mid V=v\right),
  \qquad A\subset C_{\eta,M}\ \text{Borel}.
\]
By Lemma~\ref{lem:score-interval-prob}, the interval asymptotics are already known for intervals of the form $(r,s]\subset(0,\infty)$ and $[-s,-r)\subset(-\infty,0)$. Since singleton boundary masses vanish for both $\mu_{K,v}$ and $\Lambda_v$, the same limits hold for any interval $B\subset C_{\eta,M}$ whose endpoints lie in $[-M,-\eta]\cup[\eta,M]$; changing endpoint inclusions affects only finitely many massless boundary points.
We now prove that
\begin{equation}\label{eq:h-compact-conv}
  \int_{C_{\eta,M}} \widetilde h_{t,\eta,M}(x) \mu_{K,v}(dx)
  \to
  \int_{C_{\eta,M}} \widetilde h_{t,\eta,M}(x) \Lambda_v(dx).
\end{equation}

Since $\widetilde h_{t,\eta,M}$ is bounded and continuous on $C_{\eta,M}$,
the convergence \eqref{eq:h-compact-conv} is the special case
$g=\widetilde h_{t,\eta,M}$ of Lemma~\ref{lem:mu-conv-away-zero} below,
whose proof relies only on Lemma~\ref{lem:score-interval-prob} and is
independent of the present proposition.

Next define
\[
  W_{i,K}^{(\eta,M)}(v)
  :=
  \frac{X_{i,K}^{(v)}}{K} 
  \1 \left\{\eta<\left|\frac{X_{i,K}^{(v)}}{K}\right|\le M\right\}.
\]
Conditionally on $V=v$, the variables $W_{1,K}^{(\eta,M)}(v),\dots,W_{K,K}^{(\eta,M)}(v)$ are i.i.d., and
\[
  S_{K,\eta,M}(v)=\sum_{i=1}^K W_{i,K}^{(\eta,M)}(v).
\]
Therefore
\begin{align}
  \E \left[e^{it S_{K,\eta,M}(v)} \mid V=v\right]
  &=
  \left(
    \E \left[e^{it W_{1,K}^{(\eta,M)}(v)} \mid V=v\right]
  \right)^K \notag\\
  &=
  \left(
    1+
    \E \left[
      (e^{itX_{1,K}^{(v)}/K}-1) 
      \1 \left\{\eta<\left|\frac{X_{1,K}^{(v)}}{K}\right|\le M\right\}
 \mid V=v
    \right]
  \right)^K \notag\\
  &=
  \left(
    1+\frac{a_{K,\eta,M}(t,v)}{K}
  \right)^K, \label{eq:cf-bounded-trunc}
\end{align}
where
\[
  a_{K,\eta,M}(t,v)
  :=
  K 
  \E \left[
    (e^{itX_{1,K}^{(v)}/K}-1) 
    \1 \left\{\eta<\left|\frac{X_{1,K}^{(v)}}{K}\right|\le M\right\}
 \mid V=v
  \right].
\]
By \eqref{eq:h-compact-conv} and the boundary-mass observation at the start of the bounded-truncation step,
\[
  a_{K,\eta,M}(t,v)
  \to
  A_{\eta,M}(t,v)
  :=
  \int_{\eta<|x|\le M}(e^{itx}-1) \Lambda_v(dx).
\]
Substituting into \eqref{eq:cf-bounded-trunc} and using
\[
  \left(1+\frac{u_K}{K}\right)^K\to e^u
  \qquad\text{whenever }u_K\to u,
\]
we obtain
\begin{equation}\label{eq:bounded-trunc-limit}
  \E \left[e^{it S_{K,\eta,M}(v)} \mid V=v\right]
  \to
  \exp \left(A_{\eta,M}(t,v)\right).
\end{equation}

\paragraph{Removing the upper cutoff.}
We compare $S_{K,\eta}(v)$ and $S_{K,\eta,M}(v)$.
If these two sums differ, then there exists at least one index $i$ such that
\[
  \left|\frac{X_{i,K}^{(v)}}{K}\right|>M.
\]
Hence, by the union bound,
\begin{align}
  \Pbb \left(S_{K,\eta}(v)\neq S_{K,\eta,M}(v) \mid V=v\right)
  &\le
  K \Pbb \left(\left|\frac{X_{1,K}^{(v)}}{K}\right|>M \mid V=v\right) \notag\\
  &=
  K \Pbb \left(\frac{X_{1,K}^{(v)}}{K}>M \mid V=v\right)
  +
  K \Pbb \left(\frac{X_{1,K}^{(v)}}{K}< -M \mid V=v\right). \label{eq:tail-union}
\end{align}
By \eqref{eq:score-tail-positive} and \eqref{eq:score-tail-negative},
\begin{equation}\label{eq:tail-union-limit}
  \limsup_{K\to\infty}
  \Pbb \left(S_{K,\eta}(v)\neq S_{K,\eta,M}(v) \mid V=v\right)
  \le
  \frac{\lambda^{+}(v)+\lambda^{-}(v)}{\pi M}.
\end{equation}
Therefore, for every fixed $t\in\R$,
\begin{align}
  &\limsup_{K\to\infty}
  \left|
    \E \left[e^{itS_{K,\eta}(v)} \mid V=v\right]
    -
    \E \left[e^{itS_{K,\eta,M}(v)} \mid V=v\right]
  \right| \notag\\
  &\qquad\le
  2 
  \limsup_{K\to\infty}
  \Pbb \left(S_{K,\eta}(v)\neq S_{K,\eta,M}(v) \mid V=v\right) \notag\\
  &\qquad\le
  \frac{2(\lambda^{+}(v)+\lambda^{-}(v))}{\pi M}. \label{eq:cf-tail-control}
\end{align}
On the other hand, since $|e^{itx}-1|\le 2$ and $\Lambda_v(|x|>\eta)<\infty$,
\[
  A_{\eta,M}(t,v)
  =
  \int_{\eta<|x|\le M}(e^{itx}-1) \Lambda_v(dx)
  \to
  A_{\eta}(t,v)
  :=
  \int_{|x|>\eta}(e^{itx}-1) \Lambda_v(dx)
\]
as $M\to\infty$ by dominated convergence.
Therefore
\[
  \exp \left(A_{\eta,M}(t,v)\right)\to \exp \left(A_{\eta}(t,v)\right).
\]

Now let $\varepsilon>0$.
Choose $M$ large enough that both
\[
  \frac{2(\lambda^{+}(v)+\lambda^{-}(v))}{\pi M}<\frac{\varepsilon}{3}
\]
and
\[
  \left|
    \exp \left(A_{\eta,M}(t,v)\right)-\exp \left(A_{\eta}(t,v)\right)
  \right|
  <
  \frac{\varepsilon}{3}
\]
hold.
Then choose $K$ large enough that, by \eqref{eq:bounded-trunc-limit},
\[
  \left|
    \E \left[e^{it S_{K,\eta,M}(v)} \mid V=v\right]
    -
    \exp \left(A_{\eta,M}(t,v)\right)
  \right|
  <
  \frac{\varepsilon}{3}.
\]
Using \eqref{eq:cf-tail-control}, we conclude that
\[
  \left|
    \E \left[e^{it S_{K,\eta}(v)} \mid V=v\right]
    -
    \exp \left(A_{\eta}(t,v)\right)
  \right|
  <\varepsilon
\]
for all sufficiently large $K$.
This proves
\[
  \E \left[e^{it S_{K,\eta}(v)} \mid V=v\right]
  \to
  \exp \left\{
    \int_{|x|>\eta}(e^{itx}-1) \Lambda_v(dx)
  \right\}.
\]
Since the right-hand side is a characteristic function, Lévy's continuity theorem gives the conditional convergence in distribution of $S_{K,\eta}(v)$ to a random variable $S_{\eta,v}$ with characteristic function \eqref{eq:truncated-sum-cf}.
Because $\Lambda_v(B_\eta)<\infty$, this law is compound Poisson.
\end{proof}

\subsection{Conditional stable limit for Theorem 3.3}\label{sec:stable}

Throughout this section assume
\[
  \rho_K\downarrow 0,
  \qquad
  c_K:=\rho_K\log K\to c\in[0,\infty),
\]
fix $v\in\R$, and work conditionally on $V=v$.
Write
\[
  Y_{i,K}^{(v)}:=\frac{X_{i,K}^{(v)}}{K},
  \qquad
  T_{K}^{(v)}:=\sum_{i=1}^K Y_{i,K}^{(v)}
  =
  \frac1K\sum_{i=1}^K X_{i,K}^{(v)}.
\]
For each $K$ and fixed $v$, define the finite measure
\[
  \mu_{K,v}(A)
  :=
  K \Pbb \left(Y_{1,K}^{(v)}\in A \mid V=v\right),
  \qquad
  A\subset\R\ \text{Borel}.
\]

\begin{lemma}\label{lem:mu-conv-away-zero}
Fix $v\in\R$ and numbers $0<\eta<M<\infty$.
Let
\[
  C_{\eta,M}:=[-M,-\eta]\cup[\eta,M].
\]
Then for every bounded continuous function $g:C_{\eta,M}\to\R$,
\[
  \int_{C_{\eta,M}} g(x) \mu_{K,v}(dx)
  \to
  \int_{C_{\eta,M}} g(x) \Lambda_v(dx),
\]
where $\Lambda_v$ is the finite measure on $\R\setminus\{0\}$ defined in \eqref{eq:def-nu-v}.
\end{lemma}

\begin{proof}
Fix $\eta$ and $M$ and abbreviate $C:=C_{\eta,M}$.
Since $C$ is compact and $g$ is continuous on $C$, the function $g$ is uniformly continuous and bounded on $C$.
Because the conditional law of $Y_{1,K}^{(v)}$ is absolutely continuous and $\Lambda_v$ has a density, both $\mu_{K,v}$ and $\Lambda_v$ assign zero mass to every singleton in $C$.

We first prove convergence of $\mu_{K,v}(B)$ to $\Lambda_v(B)$ for the canonical interval orientations furnished by Lemma~\ref{lem:score-interval-prob}.
There are two cases.

\paragraph{Case 1: $B=(r,s]\subset(0,\infty)$ with $\eta\le r<s\le M$.}
By Lemma~\ref{lem:score-interval-prob},
\[
  K \Pbb \left(r<\frac{X_{1,K}^{(v)}}{K}\le s \mid V=v\right)
  \to
  \frac{\lambda^{+}(v)}{\pi}\left(\frac1r-\frac1s\right).
\]
By definition of $\mu_{K,v}$ and $\Lambda_v$,
\[
  \mu_{K,v}(B)
  =
  K \Pbb \left(Y_{1,K}^{(v)}\in B \mid V=v\right),
\]
and
\[
  \Lambda_v(B)
  =
  \int_r^s \frac{\lambda^{+}(v)}{\pi x^2} dx
  =
  \frac{\lambda^{+}(v)}{\pi}\left(\frac1r-\frac1s\right).
\]
Therefore
\[
  \mu_{K,v}(B)\to \Lambda_v(B).
\]

\paragraph{Case 2: $B=[-s,-r)\subset(-\infty,0)$ with $\eta\le r<s\le M$.}
Again by Lemma~\ref{lem:score-interval-prob},
\[
  K \Pbb \left(-s\le\frac{X_{1,K}^{(v)}}{K}< -r \mid V=v\right)
  \to
  \frac{\lambda^{-}(v)}{\pi}\left(\frac1r-\frac1s\right).
\]
Also,
\[
  \Lambda_v(B)
  =
  \int_{-s}^{-r}\frac{\lambda^{-}(v)}{\pi x^2} dx
  =
  \frac{\lambda^{-}(v)}{\pi}\left(\frac1r-\frac1s\right).
\]
Hence
\[
  \mu_{K,v}(B)\to \Lambda_v(B).
\]

Thus the convergence just proved holds for the canonical interval orientations on the positive and negative half-axes. Since singleton boundary masses vanish for both measures, the same limit extends immediately to any interval $B\subset C$ whose endpoints lie in $[-M,-\eta]\cup[\eta,M]$; changing endpoint inclusions alters only finitely many massless boundary points. In particular,
\begin{equation}\label{eq:mu-interval-conv}
  \mu_{K,v}(B)\to \Lambda_v(B)
\end{equation}
for every partition cell $B$ used below.

We now pass from interval indicators to a general bounded continuous function.
Fix $\varepsilon>0$.
By uniform continuity of $g$ on the compact set $C$, there exists a finite partition of $C$ into pairwise disjoint half-open intervals
\[
  C=B_1\cup\cdots\cup B_L
\]
such that
\[
  \sup_{x,y\in B_\ell}|g(x)-g(y)|\le \varepsilon
  \qquad\text{for every }\ell=1,\dots,L.
\]
Choose a point $x_\ell\in B_\ell$ for each $\ell$.

Then we have
\begin{align}
  \left|
    \int_C g(x) \mu_{K,v}(dx)
    -
    \sum_{\ell=1}^L g(x_\ell) \mu_{K,v}(B_\ell)
  \right|
  &\le
  \sum_{\ell=1}^L \int_{B_\ell}|g(x)-g(x_\ell)| \mu_{K,v}(dx) \notag\\
  &\le
  \varepsilon \sum_{\ell=1}^L \mu_{K,v}(B_\ell)
  =
  \varepsilon \mu_{K,v}(C). \label{eq:g-mu-approx}
\end{align}
Likewise,
\begin{align}
  \left|
    \int_C g(x) \Lambda_v(dx)
    -
    \sum_{\ell=1}^L g(x_\ell) \Lambda_v(B_\ell)
  \right|
  &\le
  \sum_{\ell=1}^L \int_{B_\ell}|g(x)-g(x_\ell)| \Lambda_v(dx) \notag\\
  &\le
  \varepsilon \sum_{\ell=1}^L \Lambda_v(B_\ell)
  =
  \varepsilon \Lambda_v(C). \label{eq:g-nu-approx}
\end{align}

It remains to bound $\mu_{K,v}(C)$.
Since
\[
  C=[-M,-\eta]\cup[\eta,M],
\]
Lemma~\ref{lem:score-interval-prob} gives
\[
  \mu_{K,v}(C)
  =
  K \Pbb \left(\eta\le \left|\frac{X_{1,K}^{(v)}}{K}\right|\le M \mid V=v\right)
  \to
  \Lambda_v(C)<\infty.
\]
Therefore there exists a finite constant $C_{\eta,M,v}$ such that
\begin{equation}\label{eq:mu-C-bdd}
  \sup_{K\ge1}\mu_{K,v}(C)\le C_{\eta,M,v}.
\end{equation}

By \eqref{eq:mu-interval-conv}, for each $\ell$,
\[
  \mu_{K,v}(B_\ell)\to \Lambda_v(B_\ell),
\]
so
\[
  \sum_{\ell=1}^L g(x_\ell) \mu_{K,v}(B_\ell)
  \to
  \sum_{\ell=1}^L g(x_\ell) \Lambda_v(B_\ell).
\]
Combining this convergence with \eqref{eq:g-mu-approx}, \eqref{eq:g-nu-approx}, and \eqref{eq:mu-C-bdd}, we obtain
\begin{align*}
  \limsup_{K\to\infty}
  \left|
    \int_C g(x) \mu_{K,v}(dx)
    -
    \int_C g(x) \Lambda_v(dx)
  \right|
  &\le
  \varepsilon C_{\eta,M,v}
  +
  \varepsilon \Lambda_v(C).
\end{align*}
Because $\varepsilon>0$ was arbitrary, the desired convergence follows.
\end{proof}

\begin{lemma}\label{lem:uniform-tail-xinv}
Fix $v\in\R$ and assume $c_K=\rho_K\log K\to c\in[0,\infty)$.
Then there exist constants $C_v<\infty$ and $K_0\in\mathbb N$ such that for all $K\ge K_0$ and all $x\in(0,1]$,
\begin{equation}\label{eq:uniform-tail-pos}
  K \Pbb \left(Y_{1,K}^{(v)}>x \mid V=v\right)\le \frac{C_v}{x},
\end{equation}
and
\begin{equation}\label{eq:uniform-tail-neg}
  K \Pbb \left(Y_{1,K}^{(v)}< -x \mid V=v\right)\le \frac{C_v}{x}.
\end{equation}
\end{lemma}

\begin{proof}
We prove \eqref{eq:uniform-tail-pos}; the proof of \eqref{eq:uniform-tail-neg} is identical after replacing $Z_{1,K}^{(v)}$ by $-Z_{1,K}^{(v)}$.

Fix $v$.
Because $c_K\to c$, there exists $K_1$ such that
\[
  c_K\le c+1
\]
for all $K\ge K_1$.
For each such $K$, define
\[
  m_K:=\sqrt{\rho_K} v,
  \qquad
  \sigma_K:=\sqrt{1-\rho_K}.
\]
Then $\sigma_K\to 1$ and $m_K\to 0$.

Let $x\in(0,1]$.

\paragraph{Range $0<x<1/K$.}
Since probabilities are bounded by $1$,
\[
  K \Pbb \left(Y_{1,K}^{(v)}>x \mid V=v\right)\le K.
\]
Because $x<1/K$, we have $K<1/x$, so
\[
  K \Pbb \left(Y_{1,K}^{(v)}>x \mid V=v\right)\le \frac{1}{x}.
\]
Thus \eqref{eq:uniform-tail-pos} holds in this range.

\paragraph{Range $1/K\le x\le 1$.}
Set
\[
  u:=Kx\in[1,K].
\]
Since $f$ is strictly increasing by Lemma~\ref{lem:f-basic}, the event $\{Y_{1,K}^{(v)}>x\}$ is equivalent to
\[
  \left\{X_{1,K}^{(v)}>u\right\}
  =
  \left\{f \left(Z_{1,K}^{(v)}\right)>u\right\}
  =
  \left\{Z_{1,K}^{(v)}\ge q_u\right\},
\]
where
\[
  q_u:=f^{-1}(u)=\bar\Ph^{-1} \left(\frac{1}{\pi}\operatorname{arccot}(u)\right).
\]
Define
\[
  a_u:=\frac{1}{\pi}\operatorname{arccot}(u)=\barPh(q_u),
  \qquad
  y_u:=\frac{q_u-m_K}{\sigma_K}.
\]
Then, conditionally on $V=v$,
\[
  \Pbb \left(Y_{1,K}^{(v)}>x \mid V=v\right)
  =
  \Pbb \left(Z_{1,K}^{(v)}\ge q_u \mid V=v\right)
  =
  \barPh(y_u).
\]

We first obtain deterministic bounds on $q_u$.
Since $u\in[1,K]$ and $\operatorname{arccot}$ is decreasing on $(0,\infty)$,
\[
  \frac{1}{\pi}\operatorname{arccot}(K)\le a_u\le \frac{1}{\pi}\operatorname{arccot}(1)=\frac14.
\]
Therefore
\[
  q_u\ge \bar\Ph^{-1} \left(\frac14\right)=:q_\star>0.
\]
Moreover, because $a_u\ge \frac{1}{\pi}\operatorname{arccot}(K)$ and $\operatorname{arccot}(K)\sim 1/K$, the standard Mills ratio bounds imply $q_u=O(\sqrt{\log K})$ uniformly in $u\in[1,K]$.
More explicitly, there exists a constant $C_1<\infty$ such that
\begin{equation}\label{eq:qu-upper}
  q_u^2\le C_1 \log K
  \qquad\text{for all }u\in[1,K]\text{ and all sufficiently large }K.
\end{equation}

Next we compare $y_u$ with $q_u$.
Since
\[
  y_u-q_u
  =
  q_u \left(\frac1{\sigma_K}-1\right)-\frac{m_K}{\sigma_K},
\]
and
\[
  \frac1{\sigma_K}-1
  =
  \frac{1-\sigma_K}{\sigma_K}
  =
  \frac{\rho_K}{\sigma_K(1+\sigma_K)},
\]
we obtain
\[
  |y_u-q_u|
  \le
  \frac{\rho_K q_u}{\sigma_K(1+\sigma_K)}
  +
  \frac{|m_K|}{\sigma_K}.
\]
Because $\sigma_K\to 1$, $q_u=O(\sqrt{\log K})$ uniformly by \eqref{eq:qu-upper}, $\rho_K\log K=c_K\le c+1$, and $m_K=\sqrt{\rho_K} v$, it follows that
\[
  \sup_{u\in[1,K]}|y_u-q_u|\to 0
  \qquad\text{as }K\to\infty.
\]
Hence, for all sufficiently large $K$,
\[
  y_u\ge \frac{q_\star}{2}>0
  \qquad\text{for all }u\in[1,K].
\]
Therefore Mills' inequality applies to both $q_u$ and $y_u$, and yields
\begin{align}
  \frac{\barPh(y_u)}{\barPh(q_u)}
  &\le
  \frac{\phi(y_u)}{y_u} 
  \frac{q_u+q_u^{-1}}{\phi(q_u)} \notag\\
  &=
  \frac{q_u+q_u^{-1}}{y_u} 
  \exp \left(-\frac{y_u^2-q_u^2}{2}\right). \label{eq:ratio-prebound}
\end{align}

We now bound the right-hand side of \eqref{eq:ratio-prebound}.
Since $\sup_{u\in[1,K]}|y_u-q_u|\to0$ and $q_u\ge q_\star$, we have $y_u\ge q_u/2$ for all sufficiently large $K$.
Therefore
\[
  \frac{q_u+q_u^{-1}}{y_u}
  \le
  \frac{q_u}{y_u} + \frac{q_u^{-1}}{y_u}
  \le
  2+\frac{2}{q_u^2}
  \le
  2+\frac{2}{q_\star^2}
  =:C_2.
\]

Next,
\[
  y_u^2-q_u^2
  =
  \frac{\rho_K q_u^2 -2m_K q_u + m_K^2}{1-\rho_K},
\]
exactly as in \eqref{eq:y2-minus-q2}.
Therefore
\begin{align*}
  -\frac{y_u^2-q_u^2}{2}
  &=
  -\frac{\rho_K q_u^2}{2(1-\rho_K)}
  +\frac{m_K q_u}{1-\rho_K}
  -\frac{m_K^2}{2(1-\rho_K)} \\
  &\le
  \frac{|m_K|q_u}{1-\rho_K},
\end{align*}
because the first and third terms on the right-hand side are nonpositive.
Using \eqref{eq:qu-upper},
\[
  |m_K|q_u
  =
  |v|\sqrt{\rho_K} q_u
  \le
  |v| \sqrt{\rho_K C_1\log K}
  \le
  |v| \sqrt{C_1(c+1)}.
\]
Also $1-\rho_K\to1$, so for all sufficiently large $K$,
\[
  \exp \left(-\frac{y_u^2-q_u^2}{2}\right)\le C_3
\]
for some finite constant $C_3$ depending only on $v$ and the limit $c$.
Combining this with \eqref{eq:ratio-prebound} yields
\[
  \frac{\barPh(y_u)}{\barPh(q_u)}\le C_2 C_3=:C_4.
\]
Recalling that $\barPh(q_u)=a_u=\frac{1}{\pi}\operatorname{arccot}(u)$, we obtain
\[
  \barPh(y_u)\le C_4 a_u.
\]
Since $\operatorname{arccot}(u)\le 1/u$ for all $u>0$,
\[
  a_u=\frac{1}{\pi}\operatorname{arccot}(u)\le \frac{1}{\pi u}.
\]
Therefore
\[
  \Pbb \left(Y_{1,K}^{(v)}>x \mid V=v\right)
  =
  \barPh(y_u)
  \le
  \frac{C_4}{\pi u}
  =
  \frac{C_4}{\pi Kx}.
\]
Multiplying by $K$ gives
\[
  K \Pbb \left(Y_{1,K}^{(v)}>x \mid V=v\right)\le \frac{C_4/\pi}{x}.
\]

Combining Steps 1 and 2 proves \eqref{eq:uniform-tail-pos} after enlarging the constant if necessary.
The proof of \eqref{eq:uniform-tail-neg} is identical.
\end{proof}

\begin{proposition}\label{prop:small-jumps-vanish}
Fix $v\in\R$ and let $C_v$ be the constant from Lemma~\ref{lem:uniform-tail-xinv}.
For $\eta\in(0,1]$, define
\[
  R_{K,\eta}(v)
  :=
  \sum_{i=1}^K
  \left(
    Y_{i,K}^{(v)}\1\{|Y_{i,K}^{(v)}|\le \eta\}
    -
    \E \left[Y_{1,K}^{(v)}\1\{|Y_{1,K}^{(v)}|\le \eta\} \mid V=v\right]
  \right).
\]
Then there exists $K_0=K_0(v,\eta)$ such that, for all $K\ge K_0$,
\begin{equation}\label{eq:small-jump-L2}
  \E \left[ |R_{K,\eta}(v)|^2 \mid V=v\right]
  \le 4 C_v \eta.
\end{equation}
Consequently, for every $\varepsilon>0$,
\begin{equation}\label{eq:small-jump-prob}
  \lim_{\eta\downarrow 0}\ \limsup_{K\to\infty}
  \Pbb \left(|R_{K,\eta}(v)|>\varepsilon \mid V=v\right)
  =0.
\end{equation}
\end{proposition}

\begin{proof}
Conditionally on $V=v$, the random variables
\[
  Y_{1,K}^{(v)},\dots,Y_{K,K}^{(v)}
\]
are i.i.d.
Hence
\begin{align}
  \E \left[ |R_{K,\eta}(v)|^2 \mid V=v\right]
  &=
  K 
  \Var \left(
    Y_{1,K}^{(v)}\1\{|Y_{1,K}^{(v)}|\le \eta\}
 \mid V=v
  \right) \notag\\
  &\le
  K 
  \E \left[
    |Y_{1,K}^{(v)}|^2\1\{|Y_{1,K}^{(v)}|\le \eta\}
 \mid V=v
  \right]. \label{eq:RJ-var-bound}
\end{align}

Let
\[
  U_{1,K}^{(v)}:=|Y_{1,K}^{(v)}|.
\]
For every deterministic $u\ge 0$ and every $\eta>0$,
\[
  u^2\1\{u\le\eta\}
  \le
  2\int_0^\eta x \1\{u>x\} dx.
\]
Indeed, if $u\le\eta$, then the right-hand side equals
\[
  2\int_0^u x dx = u^2,
\]
whereas if $u>\eta$, the left-hand side is $0$ and the right-hand side is nonnegative.

Applying this pointwise inequality with $u=U_{1,K}^{(v)}$ and then taking conditional expectations, we obtain
\begin{align}
  \E \left[
    |Y_{1,K}^{(v)}|^2\1\{|Y_{1,K}^{(v)}|\le \eta\}
 \mid V=v
  \right]
  &\le
  2\int_0^\eta x 
  \Pbb \left(U_{1,K}^{(v)}>x \mid V=v\right) dx \notag\\
  &=
  2\int_0^\eta x 
  \Pbb \left(|Y_{1,K}^{(v)}|>x \mid V=v\right) dx. \label{eq:U2-tail-int}
\end{align}
Multiplying by $K$ and using
\[
  \Pbb \left(|Y_{1,K}^{(v)}|>x \mid V=v\right)
  \le
  \Pbb \left(Y_{1,K}^{(v)}>x \mid V=v\right)
  +
  \Pbb \left(Y_{1,K}^{(v)}< -x \mid V=v\right),
\]
we obtain from Lemma~\ref{lem:uniform-tail-xinv} that
\begin{align}
  K 
  \E \left[
    |Y_{1,K}^{(v)}|^2\1\{|Y_{1,K}^{(v)}|\le \eta\}
 \mid V=v
  \right]
  &\le
  2\int_0^\eta x 
  K \Pbb \left(|Y_{1,K}^{(v)}|>x \mid V=v\right) dx \notag\\
  &\le
  2\int_0^\eta x\left(\frac{C_v}{x}+\frac{C_v}{x}\right) dx \notag\\
  &=4C_v\eta. \label{eq:U2-final}
\end{align}
Combining \eqref{eq:RJ-var-bound} and \eqref{eq:U2-final} gives
\[
  \E \left[ |R_{K,\eta}(v)|^2 \mid V=v\right]
  \le 4C_v\eta
\]
for all sufficiently large $K$, which proves \eqref{eq:small-jump-L2}.

Finally, Chebyshev's inequality yields for every $\varepsilon>0$ and all sufficiently large $K$,
\[
  \Pbb \left(|R_{K,\eta}(v)|>\varepsilon \mid V=v\right)
  \le
  \frac{
    \E \left[ |R_{K,\eta}(v)|^2 \mid V=v\right]
  }{\varepsilon^2}
  \le
  \frac{4C_v\eta}{\varepsilon^2}.
\]
Taking $\limsup_{K\to\infty}$ and then letting $\eta\downarrow 0$ proves \eqref{eq:small-jump-prob}.
\end{proof}

For $0<\eta<1$, define the current truncated drift
\begin{equation}\label{eq:def-dKeta-current}
  d_{K,\eta}(v)
  :=
  K 
  \E \left[
    Y_{1,K}^{(v)}\1\{\eta<|Y_{1,K}^{(v)}|\le1\}
 \mid V=v
  \right].
\end{equation}

Finally, define the full truncated-mean centering
\begin{equation}\label{eq:def-bKv-early}
  b_K(v)
  :=
  K 
  \E \left[
    Y_{1,K}^{(v)}\1\{|Y_{1,K}^{(v)}|\le 1\}
 \mid V=v
  \right].
\end{equation}

\begin{theorem}[Conditional centered $1$-stable limit]\label{thm:centered-stable-limit}
Assume $\rho_K\downarrow 0$ and $c_K:=\rho_K\log K\to c\in[0,\infty)$.
Fix $v\in\R$ and let $b_K(v)$ be defined by \eqref{eq:def-bKv-early}.
Then, conditionally on $V=v$,
\[
  T_{K}^{(v)}-b_K(v)\dto S_{c,v},
\]
where $S_{c,v}$ is an infinitely divisible random variable with characteristic function
\begin{equation}\label{eq:stable-limit-cf}
  \E \left[e^{itS_{c,v}}\right]
  =
  \exp \left\{
    \int_{\R\setminus\{0\}}
    \left(e^{itx}-1-itx\1\{|x|\le 1\}\right) \Lambda_v(dx)
  \right\},
  \qquad t\in\R.
\end{equation}
Equivalently, $S_{c,v}$ is the $1$-stable law with Lévy measure $\Lambda_v$ and truncation function $x\mapsto x\1\{|x|\le1\}$.
\end{theorem}

\begin{proof}
Fix $v$.
For $\eta\in(0,1)$, recall the current truncated drift $d_{K,\eta}(v)$ from \eqref{eq:def-dKeta-current}, and define
\[
  A_{K,\eta}(v)
  :=
  S_{K,\eta}(v)-d_{K,\eta}(v),
\]
where $S_{K,\eta}(v)$ is the fixed-cutoff extreme-score sum from \eqref{eq:def-Seta}, and
\[
  R_{K,\eta}(v)
  :=
  \sum_{i=1}^K
  \left(
    Y_{i,K}^{(v)}\1\{|Y_{i,K}^{(v)}|\le \eta\}
    -
    \E \left[
      Y_{1,K}^{(v)}\1\{|Y_{1,K}^{(v)}|\le \eta\}
 \mid V=v
    \right]
  \right).
\]
We first record the decomposition
\begin{equation}\label{eq:decomposition-full}
  T_{K}^{(v)}-b_K(v)=A_{K,\eta}(v)+R_{K,\eta}(v),
\end{equation}
which follows by splitting both $T_K^{(v)}=\sum_{i=1}^K Y_{i,K}^{(v)}$ and $b_K(v)$ at the threshold $|Y_{i,K}^{(v)}|=\eta$: the $\{|\cdot|>\eta\}$ parts give $A_{K,\eta}(v)=S_{K,\eta}(v)-d_{K,\eta}(v)$ and the $\{|\cdot|\le\eta\}$ parts give $R_{K,\eta}(v)$, by \eqref{eq:def-dKeta-current}, \eqref{eq:def-bKv-early}, and \eqref{eq:def-Seta}.

\paragraph{Deterministic centering.}
On the compact set
\[
  C_{\eta,1}=[-1,-\eta]\cup[\eta,1],
\]
consider the function
\[
  g_\eta(x):=x.
\]
This function is bounded and continuous on $C_{\eta,1}$.
Because the conditional law of $Y_{1,K}^{(v)}$ is absolutely continuous for every $K$ and $\Lambda_v$ has a density, both $\mu_{K,v}$ and $\Lambda_v$ assign zero mass to the boundary points $\{\pm\eta,\pm1\}$. Hence
\[
  d_{K,\eta}(v)
  =
  \int_{C_{\eta,1}} g_\eta(x) \mu_{K,v}(dx),
\]
and, defining
\begin{equation}\label{eq:def-deta}
  d_\eta(v)
  :=
  \int_{\eta<|x|\le1} x \Lambda_v(dx)
  =
  \int_{C_{\eta,1}} g_\eta(x) \Lambda_v(dx),
\end{equation}
we may apply Lemma~\ref{lem:mu-conv-away-zero} to obtain
\begin{equation}\label{eq:dKeta-conv}
  d_{K,\eta}(v)\to d_\eta(v).
\end{equation}

\paragraph{Large-jump limit at fixed $\eta$.}
By Proposition~\ref{prop:truncated-sum-limit},
\[
  S_{K,\eta}(v)\dto S_{\eta,v},
\]
where
\[
  \E \left[e^{itS_{\eta,v}}\right]
  =
  \exp \left\{
    \int_{|x|>\eta}(e^{itx}-1) \Lambda_v(dx)
  \right\}.
\]
Since $d_{K,\eta}(v)\to d_\eta(v)$ deterministically by \eqref{eq:dKeta-conv}, Slutsky's theorem implies
\[
  A_{K,\eta}(v)\dto L_{\eta,v},
\]
where $L_{\eta,v}:=S_{\eta,v}-d_\eta(v)$ has characteristic function
\begin{align}
  \E \left[e^{itL_{\eta,v}}\right]
  &=
  \exp \left\{
    \int_{|x|>\eta}(e^{itx}-1) \Lambda_v(dx)-it d_\eta(v)
  \right\} \notag\\
  &=
  \exp \left\{
    \int_{|x|>\eta}\left(e^{itx}-1-itx\1\{\eta<|x|\le1\}\right) \Lambda_v(dx)
  \right\}. \label{eq:Leta-cf}
\end{align}
Define
\[
  \Psi_{\eta,v}(t)
  :=
  \int_{|x|>\eta}\left(e^{itx}-1-itx\1\{\eta<|x|\le1\}\right) \Lambda_v(dx).
\]
Then we have
\begin{equation}\label{eq:Aeta-cf-limit}
  \E \left[e^{itA_{K,\eta}(v)} \mid V=v\right]
  \to
  e^{\Psi_{\eta,v}(t)}
  \qquad\text{for every }t\in\R.
\end{equation}

\paragraph{Small-jump remainder.}
By \eqref{eq:decomposition-full},
\[
  T_{K}^{(v)}-b_K(v)-A_{K,\eta}(v)=R_{K,\eta}(v).
\]
Therefore, using the elementary inequality $|e^{iu}-e^{iw}|\le |u-w|$ for real $u,w$,
\begin{align}
  &\left|
    \E \left[e^{it(T_{K}^{(v)}-b_K(v))} \mid V=v\right]
    -
    \E \left[e^{itA_{K,\eta}(v)} \mid V=v\right]
  \right| \notag\\
  &\qquad\le
  |t| 
  \E \left[ |R_{K,\eta}(v)| \mid V=v\right]. \label{eq:cf-diff-R}
\end{align}
By Cauchy--Schwarz,
\[
  \E \left[ |R_{K,\eta}(v)| \mid V=v\right]
  \le
  \left(
    \E \left[ |R_{K,\eta}(v)|^2 \mid V=v\right]
  \right)^{1/2}.
\]
Hence Proposition~\ref{prop:small-jumps-vanish} implies
\begin{equation}\label{eq:cf-diff-R-final}
  \lim_{\eta\downarrow0}\ \limsup_{K\to\infty}
  \left|
    \E \left[e^{it(T_{K}^{(v)}-b_K(v))} \mid V=v\right]
    -
    \E \left[e^{itA_{K,\eta}(v)} \mid V=v\right]
  \right|
  =0.
\end{equation}

\paragraph{Passage $\eta\downarrow0$.}
Define
\[
  \Psi_v(t)
  :=
  \int_{\R\setminus\{0\}}
  \left(e^{itx}-1-itx\1\{|x|\le1\}\right) \Lambda_v(dx).
\]
We first verify that this integral is absolutely convergent.
For $|x|\le1$, the elementary Taylor bound gives
\[
  |e^{itx}-1-itx|\le \frac{t^2x^2}{2}.
\]
Therefore
\[
  \int_{0<|x|\le1}
  \left|e^{itx}-1-itx\right| \Lambda_v(dx)
  \le
  \frac{t^2}{2}\int_{0<|x|\le1}x^2 \Lambda_v(dx)
  <\infty,
\]
because $x^2\Lambda_v(dx)$ is a finite multiple of Lebesgue measure on $(-1,1)\setminus\{0\}$.
For $|x|>1$,
\[
  |e^{itx}-1|\le 2
\]
and
\[
  \Lambda_v(|x|>1)
  =
  \int_{|x|>1}\Lambda_v(dx)
  <\infty,
\]
so the large-jump part is also absolutely integrable.
Thus $\Psi_v(t)$ is well defined.

Next, for every $t\in\R$,
\[
  \Psi_v(t)-\Psi_{\eta,v}(t)
  =
  \int_{0<|x|\le\eta}\left(e^{itx}-1-itx\right) \Lambda_v(dx).
\]
Using again $|e^{itx}-1-itx|\le t^2x^2/2$, we obtain
\begin{align}
  |\Psi_v(t)-\Psi_{\eta,v}(t)|
  &\le
  \frac{t^2}{2}\int_{0<|x|\le\eta}x^2 \Lambda_v(dx) \notag\\
  &=
  \frac{t^2}{2}
  \left(
    \int_0^\eta x^2\frac{\lambda^{+}(v)}{\pi x^2} dx
    +
    \int_{-\eta}^0 x^2\frac{\lambda^{-}(v)}{\pi x^2} dx
  \right) \notag\\
  &=
  \frac{t^2}{2}
  \frac{\lambda^{+}(v)+\lambda^{-}(v)}{\pi} \eta
  \to 0 \qquad (\eta\downarrow0). \label{eq:Psi-eta-to-Psi}
\end{align}
Hence
\begin{equation}\label{eq:expPsi-eta}
  e^{\Psi_{\eta,v}(t)}\to e^{\Psi_v(t)}
  \qquad (\eta\downarrow0)
\end{equation}
for every $t\in\R$.

\paragraph{Conclusion.}
Fix $t\in\R$. Combining \eqref{eq:Aeta-cf-limit}, \eqref{eq:cf-diff-R-final}, and \eqref{eq:expPsi-eta} via the triangle inequality and taking $K\to\infty$ first then $\eta\downarrow 0$ gives
\[
  \E \left[e^{it(T_{K}^{(v)}-b_K(v))} \mid V=v\right]
  \to
  e^{\Psi_v(t)}
  \qquad\text{for every }t\in\R.
\]

It remains only to verify continuity of the limit at $t=0$.
For $|t|\le1$ the same two bounds used in the absolute-convergence check above dominate the integrand of $\Psi_v(t)$ by the fixed $\Lambda_v$-integrable function $\tfrac{x^2}{2}\1\{|x|\le1\}+2\1\{|x|>1\}$, uniformly in such $t$. Since the integrand converges pointwise to $0$ as $t\to0$, dominated convergence gives $\Psi_v(t)\to0$, hence $e^{\Psi_v(t)}\to1=e^{\Psi_v(0)}$ as $t\to0$, so the pointwise limit is continuous at the origin. L\'evy's continuity theorem now yields
\[
  T_{K}^{(v)}-b_K(v)\dto S_{c,v},
\]
where $S_{c,v}$ has characteristic function \eqref{eq:stable-limit-cf}.
This completes the proof.
\end{proof}

\begin{corollary}\label{cor:explicit-stable-exponent}
Assume the setting of Theorem~\ref{thm:centered-stable-limit} with $c=0$. Then, for every fixed $v\in\R$,
\[
  S_{0,v}\sim \mathsf{C}(0,1).
\]
\end{corollary}

\begin{proof}
When $c=0$, $\lambda^{+}(v)=\lambda^{-}(v)=1$, so $\Lambda_v(dx)=dx/(\pi x^2)$, which is the standard L\'evy representation of the Cauchy characteristic function; the exponent in \eqref{eq:stable-limit-cf} equals $-|t|$.
\end{proof}

\subsection{Asymptotics of the centering sequence}\label{sec:centering}

Theorem~\ref{thm:centered-stable-limit} identifies the centred conditional limit $T_K^{(v)}-b_K(v)\dto S_{c,v}$. To complete the proof of Theorem~\ref{thm:stablelimit}, it remains to determine the asymptotics of the centring sequence
\[
  b_K(v)
  =
  \E \left[
    X_{1,K}^{(v)}\1\{|X_{1,K}^{(v)}|\le K\}
 \mid V=v
  \right].
\]
We show that $\rho_K b_K(v)$ converges to the explicit limit $B_c(v)$ whenever $c_K\to c\in[0,\infty)$, and that a further expansion at $c=0$ yields the boundary-layer scale $s_K=\sqrt{\rho_K}(\log K)^{3/2}$.

Throughout this section we condition on $V=v$, write
\[
  m_K:=\sqrt{\rho_K} v,
  \qquad
  \sigma_K^2:=1-\rho_K,
\]
and denote by
\[
  g_K(z)
  :=
  \frac{1}{\sigma_K\sqrt{2\pi}}
  \exp \left(
    -\frac{(z-m_K)^2}{2\sigma_K^2}
  \right),
  \qquad z\in\R,
\]
the conditional density of $Z_{1,K}^{(v)}$.
We also write
\[
  a_K:=t(K)=f^{-1}(K)>0,
\]
so that, by monotonicity and oddness of $f$,
\[
  |X_{1,K}^{(v)}|\le K
  \qquad\Longleftrightarrow\qquad
  |Z_{1,K}^{(v)}|\le a_K.
\]

\begin{lemma}\label{lem:bK-exact}
For every $K$ and every fixed $v\in\R$,
\begin{equation}\label{eq:bK-exact-z}
  b_K(v)
  =
  \int_0^{a_K} f(z) \left(g_K(z)-g_K(-z)\right) dz.
\end{equation}
Moreover, for every $z\in\R$,
\begin{equation}\label{eq:g-difference}
  g_K(z)-g_K(-z)
  =
  \frac{2}{\sigma_K\sqrt{2\pi}}
  \exp \left(
    -\frac{z^2+m_K^2}{2\sigma_K^2}
  \right)
  \sinh \left(
    \frac{m_K z}{\sigma_K^2}
  \right).
\end{equation}
Consequently, if
\[
  u_K:=\sqrt{\rho_K} a_K,
\]
then we have
\begin{equation}\label{eq:rho-bK-integral}
  \rho_K b_K(v)=\int_0^{u_K} H_K(t,v) dt,
\end{equation}
where
\begin{equation}\label{eq:def-HK}
  H_K(t,v)
  :=
  \sqrt{\rho_K} 
  f \left(\frac{t}{\sqrt{\rho_K}}\right)
  \left[
    g_K \left(\frac{t}{\sqrt{\rho_K}}\right)
    -
    g_K \left(-\frac{t}{\sqrt{\rho_K}}\right)
  \right].
\end{equation}
\end{lemma}

\begin{proof}
By definition, $b_K(v)=\int_{-a_K}^{a_K}f(z)g_K(z)dz$. Substituting $z\mapsto-z$ and using $f(-z)=-f(z)$ (Lemma~\ref{lem:f-basic}) yields \eqref{eq:bK-exact-z}. Expanding the squares $(z\mp m_K)^2$ and applying $e^x-e^{-x}=2\sinh x$ gives \eqref{eq:g-difference}. Finally, multiplying \eqref{eq:bK-exact-z} by $\rho_K$ and substituting $t=\sqrt{\rho_K}z$ yields \eqref{eq:rho-bK-integral} with $H_K$ as in \eqref{eq:def-HK}.
\end{proof}

\begin{theorem}[Asymptotic formula for the centering sequence]\label{thm:bK-general}
Assume $\rho_K\downarrow0$ and $c_K:=\rho_K\log K\to c\in[0,\infty)$.
Fix $v\in\R$ and define
\begin{equation}\label{eq:def-Bc}
  B_c(v)
  :=
  \frac{2}{\pi}
  \int_0^{\sqrt{2c}} t e^{-t^2/2}\sinh(vt) dt.
\end{equation}
Then we have
\begin{equation}\label{eq:rho-bK-limit}
  \rho_K b_K(v)\to B_c(v).
\end{equation}
\end{theorem}

\begin{proof}
Let
\[
  u:=\sqrt{2c}.
\]
Squaring Lemma~\ref{lem:inv-asymp} and multiplying by $\rho_K$ gives
\begin{equation}\label{eq:uK-limit}
  u_K^2=\rho_K a_K^2=2c_K\cdot a_K^2/(2\log K)\to 2c,
  \qquad\text{hence }u_K\to u.
\end{equation}
Lemma~\ref{lem:f-growth} gives $f(z)=\sqrt{2/\pi}z e^{z^2/2}(1+r(z))$ with $r(z)\to 0$ as $|z|\to\infty$ and $r$ bounded on $(0,\infty)$ (the ratio is continuous on $(0,\infty)$, has a finite limit as $z\downarrow0$ since both $f(z)$ and $\sqrt{2/\pi}z e^{z^2/2}$ vanish to first order, and tends to $1$ as $z\to\infty$). Substituting this and \eqref{eq:g-difference} into \eqref{eq:def-HK}, and simplifying using $\sigma_K^2=1-\rho_K$ so that $t^2/(2\rho_K)-t^2/(2\rho_K\sigma_K^2)=-t^2/(2\sigma_K^2)$, we obtain
\begin{equation}\label{eq:HK-simplified}
  H_K(t,v)
  =\frac{2t}{\pi\sigma_K}\exp\left(-\frac{t^2+\rho_K v^2}{2\sigma_K^2}\right)
  \sinh\left(\frac{vt}{\sigma_K^2}\right)\bigl[1+r(t/\sqrt{\rho_K})\bigr].
\end{equation}
Since $\sigma_K\to 1$, $\rho_K\to 0$, and $r(t/\sqrt{\rho_K})\to 0$ for each $t>0$,
\begin{equation}\label{eq:HK-pointwise-limit}
  H_K(t,v)\to h(t,v):=\frac{2}{\pi}t e^{-t^2/2}\sinh(vt).
\end{equation}
Fix $M>u+1$. By \eqref{eq:HK-simplified} and $|\sinh(x)|\le |x|\cosh|x|$, for $K$ large enough that $\sigma_K\ge 1/2$ and $u_K\le M$, we have $|H_K(t,v)|\le C_{M,v}(t+t^2)$ on $[0,M]$, which is integrable.

Define
\[
  \widetilde H_K(t,v):=H_K(t,v) \1\{t\le u_K\},
  \qquad
  \widetilde h(t,v):=h(t,v) \1\{t\le u\}.
\]
For almost every $t\in[0,M]$, \eqref{eq:HK-pointwise-limit} and \eqref{eq:uK-limit} imply
\[
  \widetilde H_K(t,v)\to \widetilde h(t,v).
\]
Because $|\widetilde H_K(t,v)|\le C_{M,v}(t+t^2)$, dominated convergence yields
\[
  \int_0^M \widetilde H_K(t,v) dt
  \to
  \int_0^M \widetilde h(t,v) dt.
\]
Since $u_K\le M$ for all large $K$ and $u\le M$, the left-hand side is
\[
  \int_0^{u_K} H_K(t,v) dt = \rho_K b_K(v)
\]
by \eqref{eq:rho-bK-integral}, whereas the right-hand side is
\[
  \int_0^u h(t,v) dt
  =
  \frac{2}{\pi}\int_0^{\sqrt{2c}} t e^{-t^2/2}\sinh(vt) dt
  =
  B_c(v).
\]
This proves \eqref{eq:rho-bK-limit}.
\end{proof}

\begin{theorem}[Sharp small-$c_K$ asymptotic of the centering sequence]\label{thm:bK-smallc}
Assume $\rho_K\downarrow0$ and $c_K=\rho_K\log K\to0$.
Define
\[
  s_K:=\sqrt{\rho_K} (\log K)^{3/2}.
\]
Then, for every fixed $v\in\R$,
\begin{equation}\label{eq:bK-smallc-main}
  b_K(v)
  =
  \frac{4\sqrt{2}}{3\pi} v s_K
  +
  o(s_K).
\end{equation}
\end{theorem}

\begin{proof}
We keep the notation of Lemma~\ref{lem:bK-exact} and write $b_K(v)=I_{1,K}+I_{2,K}$ with
\[
  I_{1,K}:=\int_0^{M_\varepsilon}f(z)(g_K(z)-g_K(-z))dz,\qquad
  I_{2,K}:=\int_{M_\varepsilon}^{a_K}f(z)(g_K(z)-g_K(-z))dz,
\]
where $M_\varepsilon\ge 1$ is chosen so that $|r(z)|\le\varepsilon$ for $z\ge M_\varepsilon$, with $r(z):=f(z)/(\sqrt{2/\pi}z e^{z^2/2})-1\to 0$ (Lemma~\ref{lem:f-growth}). By Lemma~\ref{lem:inv-asymp} and $c_K\to 0$, $u_K^2=2c_K\cdot a_K^2/(2\log K)\to 0$, so $u_K\to 0$.

\paragraph{Bounded region.} On $[0,M_\varepsilon]$, $|m_K z/\sigma_K^2|\to 0$ uniformly, so $|\sinh(m_K z/\sigma_K^2)|\le 2|m_K z/\sigma_K^2|$ eventually; together with $\sup_{[0,M_\varepsilon]}|f|<\infty$, this gives $|I_{1,K}|\le C_{M_\varepsilon,v}\sqrt{\rho_K}$, hence
\begin{equation}\label{eq:I1rho}
  \rho_K|I_{1,K}|\le C_{M_\varepsilon,v}\rho_K^{3/2}.
\end{equation}

\paragraph{Large-$z$ region.} Substituting $t=\sqrt{\rho_K}z$, $\rho_K I_{2,K}=\int_{M_\varepsilon\sqrt{\rho_K}}^{u_K}H_K(t,v)dt$. On this range, \eqref{eq:HK-simplified} holds with the remainder factor $1+\omega_K(t)$, $|\omega_K(t)|\le\varepsilon$. The expansions $1/\sigma_K=1+O(\rho_K)$, $\exp(-t^2/(2\sigma_K^2)-\rho_K v^2/(2\sigma_K^2))=1+O(t^2+\rho_K)$, and $\sinh(vt/\sigma_K^2)=vt+O(\rho_K t+t^3)$, all uniform in $t\in[0,u_K]$, give
\[
  H_K(t,v)=\frac{2v}{\pi}t^2+O(\varepsilon t^2)+O(\rho_K t^2+t^4)
  \qquad\text{uniformly on }[M_\varepsilon\sqrt{\rho_K},u_K].
\]
Integrating and using $\int t^2dt=u_K^3/3+O(\rho_K^{3/2})$, $\int t^4dt\le u_K^5$, and $u_K^3=\rho_K^{3/2}a_K^3$ with $a_K\to\infty$, $u_K\to 0$ (so $\rho_K^{3/2}$, $u_K^5$, $\rho_K u_K^3$ are all $o(u_K^3)$),
\[
  \rho_K I_{2,K}=\frac{2v}{3\pi}u_K^3+O(\varepsilon u_K^3)+o(u_K^3).
\]
Combined with $\rho_K|I_{1,K}|=O(\rho_K^{3/2})=o(u_K^3)$ from \eqref{eq:I1rho}, and letting $\varepsilon\downarrow 0$,
\begin{equation}\label{eq:rho-bK-smallc-local}
  \rho_K b_K(v)=\frac{2v}{3\pi}u_K^3+o(u_K^3).
\end{equation}
Finally, Lemma~\ref{lem:inv-asymp} gives $a_K^3=(2\log K)^{3/2}(1+o(1))=2\sqrt{2}(\log K)^{3/2}(1+o(1))$, hence $u_K^3/\rho_K=\sqrt{\rho_K}a_K^3=2\sqrt{2}s_K(1+o(1))$. Dividing \eqref{eq:rho-bK-smallc-local} by $\rho_K$ yields \eqref{eq:bK-smallc-main}.
\end{proof}

\begin{proof}[Proof of Theorem~\ref{thm:stablelimit}]
Under the conditional law $(V=v)$, the statistic $T_K$ is exactly $T_K^{(v)}$.
Hence the conditional convergence in \eqref{eq:main-conditional-stable}, together with the characteristic function
\eqref{eq:main-stable-cf}, is exactly Theorem~\ref{thm:centered-stable-limit}.
The centering limit \eqref{eq:main-bk-limit} is Theorem~\ref{thm:bK-general}.
If, in addition, $c_K\to0$, then the sharper expansion \eqref{eq:main-bk-smallc} is Theorem~\ref{thm:bK-smallc}.
This proves Theorem~\ref{thm:stablelimit}.
\end{proof}

\subsection{Proof of Corollary~\ref{cor:phasediagram}}

\begin{corollary}\label{cor:raw-smallc}
Assume $\rho_K\downarrow0$, $c_K=\rho_K\log K\to0$, and
\[
  s_K=\sqrt{\rho_K}(\log K)^{3/2}\to s\in[0,\infty).
\]
Fix $v\in\R$.
Then, conditionally on $V=v$,
\begin{equation}\label{eq:raw-smallc-cond-limit}
  T_K^{(v)}
  \dto
  \mathsf{C} \left(
    \frac{4\sqrt{2}}{3\pi} s v,\ 1
  \right).
\end{equation}
\end{corollary}

\begin{proof}
When $c_K\to0$, Corollary~\ref{cor:explicit-stable-exponent} shows that the conditional limit law from Theorem~\ref{thm:centered-stable-limit} is $\mathsf{C}(0,1)$ for every fixed $v\in\R$. Hence, conditionally on $V=v$,
\[
  T_K^{(v)}-b_K(v)\dto \mathsf{C}(0,1).
\]
By Theorem~\ref{thm:bK-smallc},
\[
  b_K(v)
  =
  \frac{4\sqrt{2}}{3\pi} v s_K+o(s_K)
  \to
  \frac{4\sqrt{2}}{3\pi} s v.
\]
Slutsky's theorem therefore implies
\[
  T_K^{(v)}
  =
  \left(T_K^{(v)}-b_K(v)\right)+b_K(v)
  \dto
  \mathsf{C} \left(
    \frac{4\sqrt{2}}{3\pi} s v,\ 1
  \right),
\]
which is \eqref{eq:raw-smallc-cond-limit}.
\end{proof}

\begin{corollary}\label{cor:raw-size-function}
Under the assumptions of Corollary~\ref{cor:raw-smallc},
for every fixed $\alpha\in(0,1/2)$,
\begin{equation}\label{eq:def-Psi-alpha}
  \Pbb \left(T_K>t_\alpha\right)
  \to
  \Psi_\alpha(s)
  :=
  \E \left[
    \frac12-\frac{1}{\pi}
    \arctan \left(
      t_\alpha-\frac{4\sqrt{2}}{3\pi} s V
    \right)
  \right],
\end{equation}
where $V\sim N(0,1)$ and $t_\alpha=\cot(\pi\alpha)$.
In particular,
\[
  \Psi_\alpha(0)=\alpha.
\]
\end{corollary}

\begin{proof}
Fix $v\in\R$.
By Corollary~\ref{cor:raw-smallc},
\[
  \Pbb \left(T_K>t_\alpha \mid V=v\right)
  \to
  \frac12-\frac{1}{\pi}
  \arctan \left(
    t_\alpha-\frac{4\sqrt{2}}{3\pi} s v
  \right),
\]
because the survival function of $\mathsf{C}(\mu,1)$ is
\[
  x\mapsto \frac12-\frac{1}{\pi}\arctan(x-\mu).
\]
The conditional probabilities are bounded between $0$ and $1$, so dominated convergence yields
\begin{align*}
  \Pbb \left(T_K>t_\alpha\right)
  &=
  \E \left[
    \Pbb \left(T_K>t_\alpha \mid V\right)
  \right] \\
  &\to
  \E \left[
    \frac12-\frac{1}{\pi}
    \arctan \left(
      t_\alpha-\frac{4\sqrt{2}}{3\pi} s V
    \right)
  \right].
\end{align*}
This proves \eqref{eq:def-Psi-alpha}.

If $s=0$, then
\[
  \Psi_\alpha(0)
  =
  \frac12-\frac{1}{\pi}\arctan(t_\alpha),
\]
which is the standard Cauchy survival function evaluated at $t_\alpha=\cot(\pi\alpha)$. Hence $\Psi_\alpha(0)=\alpha$.
\end{proof}

\begin{proposition}\label{prop:large-mean}
Assume $\rho_K\downarrow0$ and $c_K=\rho_K\log K\to\infty$.
Fix $v\neq0$ and write
\[
  \mu_K(v):=\mu_{\rho_K}(v).
\]
Then we have
\[
  \Pbb \left(
    \left|T_K^{(v)}-\mu_K(v)\right|>\frac12|\mu_K(v)|
 \mid V=v
  \right)\to 0.
\]
Consequently, for every fixed $\alpha\in(0,1/2)$, we have
\[
  \Pbb \left(T_K>t_\alpha \mid V=v\right)\to \1\{v>0\}.
\]
\end{proposition}

\begin{proof}
Fix $v\neq0$ and assume $\rho_K<1/2$. Set $\theta_K=(1-\rho_K)^{-1}$, $q_K=(1+\theta_K)/2\in(1,\theta_K)$ (so $q_K<2$ for large $K$), and $Y_{1,K}^{(v)}=f(Z_{1,K}^{(v)})-\mu_K(v)$. Proposition~\ref{prop:finiteK-dev} with $\delta=|\mu_K(v)|/2$ gives
\begin{equation}\label{eq:large-mean-dev}
  \Pbb\bigl(|T_K^{(v)}-\mu_K(v)|>|\mu_K(v)|/2\mid V=v\bigr)
  \le 2^{q_K+1}|\mu_K(v)|^{-q_K}K^{1-q_K}\E[|Y_{1,K}^{(v)}|^{q_K}\mid V=v].
\end{equation}
A direct computation with the parameters of Proposition~\ref{prop:uniform-qmoment} ($\eta_K=(\theta_K-1)/8$, $\lambda_K=(1+3\theta_K)/8$) gives $1-2\lambda_K\sigma_K^2=(\theta_K-1)/(4\theta_K)$ and $\lambda_K\rho_K/(1-2\lambda_K\sigma_K^2)=(1+3\theta_K)/2$, and the elementary bound $A_{q,\eta}\le C\eta^{-q/2}$ for $q\in[1,2]$, $\eta\in(0,1]$. Substituting and using $\theta_K-1\ge\rho_K$ yields
\begin{equation}\label{eq:qmoment-large-mean}
  \E[|Y_{1,K}^{(v)}|^{q_K}\mid V=v]\le C_v\rho_K^{-(q_K+1)/2}.
\end{equation}
By Theorem~\ref{thm:mu-smallrho}, $|\mu_K(v)|\ge m_v/\rho_K$ for some $m_v>0$ and all large $K$, so $|\mu_K(v)|^{-q_K}\le C_v\rho_K^{q_K}$. Combining with \eqref{eq:qmoment-large-mean} gives the product factor $\rho_K^{q_K-(q_K+1)/2}=\rho_K^{(q_K-1)/2}$, which is bounded (in fact $\to 1$) because $(q_K-1)\log\rho_K=\rho_K\log\rho_K/(2(1-\rho_K))\to 0$. Substituting into \eqref{eq:large-mean-dev} yields
\[
  \Pbb\bigl(|T_K^{(v)}-\mu_K(v)|>|\mu_K(v)|/2\mid V=v\bigr)\le C_v K^{1-q_K}
  =C_v\exp\left(-\frac{c_K}{2(1-\rho_K)}\right)\to 0,
\]
since $q_K-1=\rho_K/(2(1-\rho_K))$ and $c_K\to\infty$. This is the relative concentration statement.

For the rejection probability, $\mu_K(v)\to\mathrm{sign}(v)\cdot\infty$ by Theorem~\ref{thm:mu-smallrho}; hence eventually $|\mu_K(v)|>2t_\alpha$ with the sign of $v$, and the concentration above implies $\Pbb(T_K>t_\alpha\mid V=v)\to\1\{v>0\}$.
\end{proof}

\begin{corollary}\label{cor:half-limit-large}
Assume $\rho_K\downarrow0$ and
\[
  s_K=\sqrt{\rho_K}(\log K)^{3/2}\to\infty.
\]
Then, for every fixed $\alpha\in(0,1/2)$,
\[
  \Pbb(T_K>t_\alpha)\to \frac12.
\]
More precisely, for every fixed $v\neq0$,
\[
  \Pbb \left(T_K>t_\alpha \mid V=v\right)\to \1\{v>0\}.
\]
\end{corollary}

\begin{proof}
Fix $v\neq0$. We show $b_K(v)\to\mathrm{sign}(v)\cdot\infty$ and the centred sequence $\{T_K^{(v)}-b_K(v)\}_K$ is tight; then $\Pbb(T_K>t_\alpha\mid V=v)\to\1\{v>0\}$ is immediate from Slutsky's argument, and the unconditional claim $\Pbb(T_K>t_\alpha)\to 1/2$ follows by bounded convergence over $V$ (the null set $\{V=0\}$ contributes nothing).

By the subsequence principle, it suffices to show that any subsequence $(K_j)$ has a further sub-subsequence on which the desired conditional limit holds. Along any such sub-subsequence we may assume $c_{K_j}\to c\in[0,\infty]$. The case $c=\infty$ is Proposition~\ref{prop:large-mean}, so assume $c\in[0,\infty)$. Then Theorem~\ref{thm:centered-stable-limit} gives $T_{K_j}^{(v)}-b_{K_j}(v)\dto S_{c,v}$, hence tightness. For the divergence of $b_{K_j}(v)$: if $c>0$, Theorem~\ref{thm:bK-general} gives $\rho_{K_j}b_{K_j}(v)\to B_c(v)$, whose sign matches that of $v$ since the integrand has the sign of $v$ on $(0,\sqrt{2c}]$; combined with $\rho_{K_j}\downarrow0$, $b_{K_j}(v)\to\mathrm{sign}(v)\cdot\infty$. If $c=0$, Theorem~\ref{thm:bK-smallc} gives $b_{K_j}(v)=\kappa v s_{K_j}+o(s_{K_j})$ with $\kappa=4\sqrt2/(3\pi)$, and the hypothesis $s_K\to\infty$ (which descends to the sub-subsequence) yields the same conclusion.
\end{proof}

\begin{proof}[Proof of Corollary~\ref{cor:phasediagram}]
If $s_K\to s\in[0,\infty)$, then
\[
  c_K
  =
  \rho_K\log K
  =
  \frac{s_K^2}{(\log K)^2}
  \to 0.
\]
Hence Corollary~\ref{cor:raw-size-function} applies and yields \eqref{eq:main-Psi}.
If instead $s_K\to\infty$, then Corollary~\ref{cor:half-limit-large} gives \eqref{eq:main-raw-half}.
This proves Corollary~\ref{cor:phasediagram}.
\end{proof}

\subsection{Proof of Corollary~\ref{cor:iffexact}}

\begin{proof}[Proof of Corollary~\ref{cor:iffexact}]
The ``if'' direction follows from Corollary~\ref{cor:raw-size-function} with $s=0$, because
\[
  \rho_K(\log K)^3\to0
  \qquad\Longleftrightarrow\qquad
  s_K=\sqrt{\rho_K}(\log K)^{3/2}\to0,
\]
and then
\[
  c_K=\rho_K\log K=\frac{s_K^2}{(\log K)^2}\to 0.
\]

For the converse, suppose
\[
  \Pbb(T_K>t_\alpha)\to \alpha,
\]
but
\[
  s_K:=\sqrt{\rho_K}(\log K)^{3/2}\not\to 0.
\]
Then there exist $\varepsilon>0$ and a subsequence $(K_m)$ such that
\[
  s_{K_m}\ge \varepsilon
  \qquad\text{for all }m.
\]
Passing to a further subsequence if necessary, either
\[
  s_{K_m}\to \infty
\]
or
\[
  s_{K_m}\to s\in[\varepsilon,\infty).
\]

If $s_{K_m}\to\infty$, then Corollary~\ref{cor:half-limit-large} yields
\[
  \Pbb(T_{K_m}>t_\alpha)\to \frac12,
\]
contradicting $\alpha\in(0,1/2)$.

If $s_{K_m}\to s\in[\varepsilon,\infty)$, then we have
\[
  c_{K_m}
  =
  \rho_{K_m}\log K_m
  =
  \frac{s_{K_m}^2}{(\log K_m)^2}
  \to 0.
\]
Hence Corollary~\ref{cor:raw-size-function} applies along this subsequence and gives
\[
  \Pbb(T_{K_m}>t_\alpha)\to \Psi_\alpha(s).
\]
For every $s>0$, write
\[
  \Psi_\alpha(s)
  =
  \E\left[g(t_\alpha-\kappa sV)\right],
  \qquad
  g(x):=\frac12-\frac{1}{\pi}\arctan(x).
\]
The derivative of $g(t_\alpha-\kappa sV)$ with respect to $s$ is
\[
  \frac{\kappa V}{\pi\{1+(t_\alpha-\kappa sV)^2\}},
\]
whose absolute value is bounded by $\kappa |V|/\pi$. Since $V\in L^1$, dominated convergence justifies differentiating under the expectation. Symmetrising the resulting expression over $V$ and $-V$ gives
\[
  \Psi_\alpha'(s)
  =
  \frac{4\kappa^2 t_\alpha s}{\pi}
  \int_0^\infty
  \frac{v^2\ph(v)}
  {\left(1+(t_\alpha-\kappa sv)^2\right)\left(1+(t_\alpha+\kappa sv)^2\right)}
   dv
  >0,
\]
because $t_\alpha>0$.
Since $\Psi_\alpha(0)=\alpha$ by Corollary~\ref{cor:raw-size-function}, it follows that
\[
  \Psi_\alpha(s)>\alpha.
\]
This again contradicts the assumed exactness.

Therefore $s_K\to0$, which is equivalent to
\[
  \rho_K(\log K)^3\to0.
\]
This proves the converse implication.
\end{proof}

\section{Proofs in Section~\ref{sec:calibration}: compact and vanishing $c$-scales}\label{app:blcompact}

\subsection{Proof of Proposition 4.1}

We work with the boundary-layer survival function $p_{\mathrm{BL}}(t;s)$ introduced in \eqref{eq:main-pbl}.

\begin{proposition}\label{app-prop:bl-structure}
Let $T_s$ denote a random variable with survival function $t\mapsto p_{\mathrm{BL}}(t;s)$.
Then one may write
\[
  T_s\stackrel{d}{=}\mathsf{C}+\kappa sV,
\]
where $\mathsf{C}$ and $V\sim N(0,1)$ are independent.
Equivalently, the characteristic function of $T_s$ is
\[
  \phi_{T_s}(u)
  =
  \exp \left(
    -|u|-\frac{\kappa^2 s^2u^2}{2}
  \right),
  \qquad u\in\R.
\]
If $0\le s_1\le s_2$, then we have
\[
  T_{s_2}
  \stackrel{d}{=}
  T_{s_1}+\kappa\sqrt{s_2^2-s_1^2} Z,
\]
where $Z\sim N(0,1)$ is independent of $T_{s_1}$.
Moreover, for each $s\ge0$, the law of $T_s$ has a continuous symmetric density $f_s$, and $f_s$ is strictly decreasing on $[0,\infty)$.
\end{proposition}

\begin{proof}
Let
\[
  g(x):=\frac12-\frac1\pi\arctan(x).
\]
Note that $g(x)=\Pbb(\mathsf{C}>x)$ for every $x\in\R$.
Hence, using independence of $\mathsf{C}$ and $V$,
\[
  p_{\mathrm{BL}}(t;s)
  =
  \E[g(t-\kappa sV)]
  =
  \E \left[\Pbb(\mathsf{C}>t-\kappa sV \mid V)\right]
  =
  \Pbb(\mathsf{C}+\kappa sV>t).
\]
So $T_s\stackrel{d}{=}\mathsf{C}+\kappa sV$, as claimed.
The characteristic function therefore factorises:
\[
  \phi_{T_s}(u)
  =
  \E[e^{iu\mathsf{C}}] \E[e^{iu\kappa sV}]
  =
  e^{-|u|} e^{-\kappa^2 s^2u^2/2}
  =
  \exp \left(
    -|u|-\frac{\kappa^2 s^2u^2}{2}
  \right).
\]

The semigroup relation follows by multiplying characteristic functions: $\phi_{T_{s_1}}(u)\phi_{\kappa\sqrt{s_2^2-s_1^2}Z}(u)=\phi_{T_{s_2}}(u)$.

Let
\[
  f_0(x):=\frac{1}{\pi(1+x^2)},
  \qquad x\in\R.
\]
Then $f_0$ is the standard Cauchy density. For $s>0$, if
\[
  \ph_{\kappa s}(x):=\frac{1}{\kappa s}\ph\left(\frac{x}{\kappa s}\right),
\]
then $T_s$ has density
\[
  f_s=f_0*\ph_{\kappa s}.
\]
For $s=0$, the density is simply $f_0$.
In every case the density is continuous and symmetric.

If $s=0$, strict decrease on $[0,\infty)$ is immediate from the formula for $f_0$.
Now fix $s>0$ and write $\sigma:=\kappa s$.
Take $0<x_1<x_2$ and set
\[
  m:=\frac{x_1+x_2}{2},
  \qquad
  d:=\frac{x_2-x_1}{2}.
\]
Then $m,d>0$ and, using the convolution representation and the change of variables $y=m+u$,
\begin{align*}
  f_s(x_1)-f_s(x_2)
  &=
  \int_{\R}
    f_0(y)\left\{\ph_\sigma(x_1-y)-\ph_\sigma(x_2-y)\right\}dy \\
  &=
  \int_{\R}
    f_0(m+u)\left\{\ph_\sigma(d+u)-\ph_\sigma(d-u)\right\}du \\
  &=
  \int_0^\infty
    \left\{f_0(m-u)-f_0(m+u)\right\}
    \left\{\ph_\sigma(d-u)-\ph_\sigma(d+u)\right\}du.
\end{align*}
Since $f_0$ and $\ph_\sigma$ are even and strictly decreasing on $[0,\infty)$, we have
\[
  f_0(m-u)-f_0(m+u)>0
  \quad\text{and}\quad
  \ph_\sigma(d-u)-\ph_\sigma(d+u)>0
\]
for every $u>0$, because $|m-u|<m+u$ and $|d-u|<d+u$.
Therefore the last integral is strictly positive, so $f_s(x_1)>f_s(x_2)$.
Hence $f_s$ is strictly decreasing on $[0,\infty)$.
\end{proof}

\subsection{Proof of Theorem 4.2}

\begin{proof}[Proof of Theorem~\ref{thm:bluniform}]
Let
\[
  D_K:=\sup_{t\in\R}\left|\Pbb(T_K\le t)-F_{s_K}(t)\right|,
  \qquad
  F_{s_K}(t):=1-p_{\mathrm{BL}}(t;s_K).
\]

We first prove \eqref{eq:main-bl-cdf}, i.e.\ that $D_K\to0$.
Suppose, for contradiction, that $D_K\not\to0$.
Then there exist $\varepsilon>0$ and a subsequence $(K_j)$ such that $D_{K_j}\ge\varepsilon$ for all $j$.
Because $(s_{K_j})$ is bounded in $[0,S]$, there is a further subsequence, still denoted $(K_j)$, such that
\[
  s_{K_j}\to s\in[0,S].
\]
Since
\[
  c_{K_j}:=\rho_{K_j}\log K_j
  =
  \frac{s_{K_j}^2}{(\log K_j)^2}
  \to 0,
\]
Corollary~\ref{cor:raw-smallc} applies along this subsequence.

Fix $t\in\R$.
For every fixed $v\in\R$, Corollary~\ref{cor:raw-smallc} yields
\[
  \Pbb(T_{K_j}\le t \mid V=v)
  \to
  \frac12+\frac{1}{\pi}\arctan \left(t-\kappa s v\right).
\]
The conditional probabilities are bounded by $1$, so bounded convergence gives
\[
  \Pbb(T_{K_j}\le t)
  \to
  F_s(t).
\]
Thus $T_{K_j}\dto T_s$, where $T_s$ has continuous distribution function $F_s$.
By P\'olya's theorem,
\[
  \sup_{t\in\R}
  \left|\Pbb(T_{K_j}\le t)-F_s(t)\right|
  \to 0.
\]

Next let $\mathsf{C}$ and $V\sim N(0,1)$ be independent, and define the whole family on a common probability space by
\[
  T_s=\mathsf{C}+\kappa sV,
  \qquad s\ge0.
\]
Then $T_{s_{K_j}}\to T_s$ almost surely because $s_{K_j}\to s$, and hence $T_{s_{K_j}}\dto T_s$.
Since $F_s$ is continuous, another application of P\'olya's theorem yields
\[
  \sup_{t\in\R}
  \left|F_{s_{K_j}}(t)-F_s(t)\right|
  \to 0.
\]
Combining the last two displays gives
\[
  D_{K_j}
  \le
  \sup_{t\in\R}\left|\Pbb(T_{K_j}\le t)-F_s(t)\right|
  +
  \sup_{t\in\R}\left|F_{s_{K_j}}(t)-F_s(t)\right|
  \to 0,
\]
contradicting $D_{K_j}\ge\varepsilon$.
Therefore $D_K\to0$.

This proves \eqref{eq:main-bl-cdf}.

We next prove the oracle statement.
For $u\in(0,1)$, let $q_{K,u}$ denote the unique solution to
\[
  p_{\mathrm{BL}}(q_{K,u};s_K)=u.
\]
Such a solution exists and is unique because, for each fixed $s\ge0$, the map $t\mapsto p_{\mathrm{BL}}(t;s)$ is continuous and strictly decreasing from $1$ to $0$.
For each fixed $v$, the conditional law of $T_K$ under $(V=v)$ is the $K^{-1}$-scaled $K$-fold convolution of the continuous law of $X_{1,K}^{(v)}$, and therefore has no atoms. Hence, for every $t\in\R$,
\[
  \Pbb(T_K=t)=\E\left[\Pbb(T_K=t \mid V)\right]=0,
\]
so the unconditional law of $T_K$ is continuous.
Thus
\[
  \Pbb \left(p_{\mathrm{BL}}(T_K;s_K)\le u\right)
  =
  \Pbb(T_K\ge q_{K,u})
  =
  1-F_K(q_{K,u}),
\]
where $F_K(t):=\Pbb(T_K\le t)$.
Also,
\[
  u
  =
  p_{\mathrm{BL}}(q_{K,u};s_K)
  =
  1-F_{s_K}(q_{K,u}).
\]
Therefore, we have
\[
  \left|\Pbb \left(p_{\mathrm{BL}}(T_K;s_K)\le u\right)-u\right|
  =
  \left|F_K(q_{K,u})-F_{s_K}(q_{K,u})\right|
  \le D_K.
\]
The endpoint cases $u=0$ and $u=1$ are trivial, so taking the supremum over $u\in(0,1)$ and using $D_K\to0$ yields
\[
  \sup_{u\in[0,1]}
  \left|
    \Pbb \left(p_{\mathrm{BL}}(T_K;s_K)\le u\right)-u
  \right|
  \to 0.
\]
This proves the oracle claim.

For the plug-in version, since $|g'(x)|\le 1/\pi$ for $g(x)=\frac12-\frac1\pi\arctan(x)$, the mean value theorem gives the global Lipschitz bound
\[
  \left|p_{\mathrm{BL}}(T_K;\widehat s_K)-p_{\mathrm{BL}}(T_K;s_K)\right|
  \le
  \frac{\kappa}{\pi}\E|V| |\widehat s_K-s_K|.
\]
Hence $p_{\mathrm{BL}}(T_K;\widehat s_K)-p_{\mathrm{BL}}(T_K;s_K)=o_p(1)$ whenever $\widehat s_K-s_K\to0$ in probability.
Since the oracle statistic already converges to $\mathrm{Unif}(0,1)$ and the uniform law has continuous cdf, Slutsky's theorem and P\'olya's theorem imply
\[
  \sup_{u\in[0,1]}
  \left|
    \Pbb \left(p_{\mathrm{BL}}(T_K;\widehat s_K)\le u\right)-u
  \right|
  \to 0.
\]
This completes the proof.
\end{proof}

\subsection{Proof of Corollary 4.3}

\begin{proposition}\label{prop:bl-small-s}
Fix $\alpha\in(0,1/2)$ and let $t_\alpha=\cot(\pi\alpha)$.
Then, for every fixed $t\in\R$,
\begin{equation}\label{eq:pbl-small-s-general}
  p_{\mathrm{BL}}(t;s)
  =
  \frac12-\frac1\pi\arctan(t)
  +
  \frac{\kappa^2 t}{\pi(1+t^2)^2}s^2
  +O_t(s^4)
  \qquad (s\downarrow0).
\end{equation}
In particular,
\begin{equation}\label{eq:Psi-small-s}
  \Psi_\alpha(s)
  =
  \alpha
  +
  \frac{\kappa^2 t_\alpha}{\pi(1+t_\alpha^2)^2}s^2
  +O_\alpha(s^4).
\end{equation}
Moreover, for each fixed $s\ge0$, the map $t\mapsto p_{\mathrm{BL}}(t;s)$ is continuous and strictly decreasing from $1$ to $0$.
Hence, for every $\alpha\in(0,1)$, there is a unique $q_\alpha(s)\in\R$ such that
$p_{\mathrm{BL}}(q_\alpha(s);s)=\alpha$.
If $\alpha\in(0,1/2)$, then the map $s\mapsto q_\alpha(s)$ is strictly increasing on $[0,\infty)$.
As $s\downarrow0$,
\begin{equation}\label{eq:qalpha-small-s}
  q_\alpha(s)
  =
  t_\alpha
  +
  \frac{\kappa^2 t_\alpha}{1+t_\alpha^2}s^2
  +O_\alpha(s^4).
\end{equation}
\end{proposition}

\begin{proof}
Write
\[
  g(x):=\frac12-\frac1\pi\arctan(x),
  \qquad
  p_{\mathrm{BL}}(t;s)=\E\left[g(t-\kappa sV)\right].
\]
Extend the definition to all real $s$ by
\[
  \widetilde p_{\mathrm{BL}}(t,s):=\E\left[g(t-\kappa sV)\right],
  \qquad (t,s)\in\R^2.
\]
Then $\widetilde p_{\mathrm{BL}}(t,s)=p_{\mathrm{BL}}(t;s)$ for $s\ge0$.
For each pair of nonnegative integers $(a,b)$ with $a+b\le4$,
\[
  \partial_t^a\partial_s^b g(t-\kappa sV)
  =
  (-\kappa V)^b g^{(a+b)}(t-\kappa sV).
\]
The derivatives $g^{(m)}$ are rational functions whose denominators are powers of $1+x^2$;
in particular, for each $m\le4$ there is a constant $C_m<\infty$ with $\|g^{(m)}\|_\infty\le C_m$.
Hence
\[
  \left|\partial_t^a\partial_s^b g(t-\kappa sV)\right|
  \le
  C_{a+b}\kappa^b |V|^b,
  \qquad a+b\le4,
\]
and the right-hand side is integrable because the Gaussian variable $V$ has finite moments of every order.
Dominated convergence therefore allows differentiation under the expectation for all mixed derivatives of total order at most four, so $(t,s)\mapsto \widetilde p_{\mathrm{BL}}(t,s)$ is $C^4$ on $\R^2$.

Because $V\sim N(0,1)$ is symmetric, $\widetilde p_{\mathrm{BL}}(t,s)=\widetilde p_{\mathrm{BL}}(t,-s)$ for all $(t,s)$.
Hence, for each fixed $t$, the odd $s$-derivatives vanish at $s=0$.
Using Taylor's theorem in $s$ around $0$, we obtain
\[
  p_{\mathrm{BL}}(t;s)
  =
  p_{\mathrm{BL}}(t;0)
  +
  \frac12 \partial_{ss}p_{\mathrm{BL}}(t;0) s^2
  +O_t(s^4).
\]
Now
\[
  p_{\mathrm{BL}}(t;0)=g(t)=\frac12-\frac1\pi\arctan(t),
\]
and
\[
  g''(x)=\frac{2x}{\pi(1+x^2)^2}.
\]
Therefore
\[
  \partial_{ss}p_{\mathrm{BL}}(t;0)
  =
  \kappa^2\E[V^2]g''(t)
  =
  \frac{2\kappa^2 t}{\pi(1+t^2)^2},
\]
which proves \eqref{eq:pbl-small-s-general}. Taking $t=t_\alpha$ and using
\[
  \frac12-\frac1\pi\arctan(t_\alpha)=\alpha
\]
gives \eqref{eq:Psi-small-s}.

For the monotonicity claim in $t$, differentiate:
\[
  \partial_t p_{\mathrm{BL}}(t;s)
  =
  \E\left[g'(t-\kappa sV)\right]
  =
  -\frac1\pi\E\left[\frac{1}{1+(t-\kappa sV)^2}\right]
  <0.
\]
Thus $t\mapsto p_{\mathrm{BL}}(t;s)$ is strictly decreasing and continuous.
Since $0\le g(t-\kappa sV)\le1$ and $g(x)\to1$ as $x\to-\infty$ while $g(x)\to0$ as $x\to+\infty$, bounded convergence yields
\[
  \lim_{t\to-\infty} p_{\mathrm{BL}}(t;s)=1,
  \qquad
  \lim_{t\to+\infty} p_{\mathrm{BL}}(t;s)=0.
\]
Hence there is a unique $q_\alpha(s)$ with $p_{\mathrm{BL}}(q_\alpha(s);s)=\alpha$.

Because $\partial_t p_{\mathrm{BL}}(t_\alpha;0)=-1/[\pi(1+t_\alpha^2)]\neq0$,
the implicit function theorem gives a $C^4$ function $q_\alpha(s)$ in a neighbourhood of $0$ satisfying
$p_{\mathrm{BL}}(q_\alpha(s);s)=\alpha$ and $q_\alpha(0)=t_\alpha$.
By symmetry, $p_{\mathrm{BL}}(t;s)=p_{\mathrm{BL}}(t;-s)$, so uniqueness implies $q_\alpha(s)=q_\alpha(-s)$ for sufficiently small $|s|$.
Thus $q_\alpha'(0)=q_\alpha^{(3)}(0)=0$.
Differentiating the identity $p_{\mathrm{BL}}(q_\alpha(s);s)=\alpha$ twice and evaluating at $s=0$ gives
\[
  0
  =
  \partial_t p_{\mathrm{BL}}(t_\alpha;0) q_\alpha''(0)
  +
  \partial_{ss} p_{\mathrm{BL}}(t_\alpha;0),
\]
because $q_\alpha'(0)=0$ and $\partial_{ts}p_{\mathrm{BL}}(t_\alpha;0)=0$.
Using the formulas above,
\[
  q_\alpha''(0)
  =
  -\frac{\partial_{ss} p_{\mathrm{BL}}(t_\alpha;0)}{\partial_t p_{\mathrm{BL}}(t_\alpha;0)}
  =
  \frac{2\kappa^2 t_\alpha}{1+t_\alpha^2}.
\]
Since $q_\alpha$ is $C^4$ and even near $0$, Taylor's theorem yields
\[
  q_\alpha(s)
  =
  q_\alpha(0)
  +
  \frac12 q_\alpha''(0)s^2
  +O_\alpha(s^4)
  =
  t_\alpha
  +
  \frac{\kappa^2 t_\alpha}{1+t_\alpha^2}s^2
  +O_\alpha(s^4),
\]
which is \eqref{eq:qalpha-small-s}.

To prove strict monotonicity in $s$, fix $0\le s_1<s_2$ and let
\[
  \sigma:=\kappa\sqrt{s_2^2-s_1^2}.
\]
By Proposition~\ref{app-prop:bl-structure}, we may write
\[
  T_{s_2}\stackrel{d}{=}T_{s_1}+\sigma Z,
\]
where $Z\sim N(0,1)$ is independent of $T_{s_1}$.
Fix $t>0$ and define
\[
  H_t(a):=\Pbb\left(T_{s_1}\in[-t-a,t-a]\right)
  =\int_{-t-a}^{t-a} f_{s_1}(x)dx,
  \qquad a\in\R,
\]
where $f_{s_1}$ is the density from Proposition~\ref{app-prop:bl-structure}.
Because $f_{s_1}$ is even, $H_t$ is even.
Moreover, for $a>0$, Leibniz' rule and symmetry give
\[
  H_t'(a)=f_{s_1}(t+a)-f_{s_1}(|t-a|)<0,
\]
since $|t-a|<t+a$ and $f_{s_1}$ is strictly decreasing on $[0,\infty)$.
Thus $H_t(a)<H_t(0)$ for every $a\ne0$.
Because $\sigma Z$ is nondegenerate,
\[
  \Pbb(|T_{s_2}|\le t)
  =
  \E[H_t(\sigma Z)]
  <
  H_t(0)
  =
  \Pbb(|T_{s_1}|\le t).
\]
Each $T_s$ is symmetric and continuous, so for $t>0$,
\[
  p_{\mathrm{BL}}(t;s)
  =
  \Pbb(T_s>t)
  =
  \frac12\left(1-\Pbb(|T_s|\le t)\right).
\]
Therefore
\[
  p_{\mathrm{BL}}(t;s_2)>p_{\mathrm{BL}}(t;s_1)
  \qquad (t>0).
\]
Now let $\alpha\in(0,1/2)$.
Since $p_{\mathrm{BL}}(0;s)=1/2$ and $t\mapsto p_{\mathrm{BL}}(t;s)$ is strictly decreasing, we have $q_\alpha(s)>0$ for every $s\ge0$.
Taking $t=q_\alpha(s_1)$ in the last display gives
\[
  p_{\mathrm{BL}}(q_\alpha(s_1);s_2)
  >
  p_{\mathrm{BL}}(q_\alpha(s_1);s_1)
  =
  \alpha.
\]
Because $t\mapsto p_{\mathrm{BL}}(t;s_2)$ is strictly decreasing, it follows that
\[
  q_\alpha(s_2)>q_\alpha(s_1).
\]
Hence $s\mapsto q_\alpha(s)$ is strictly increasing on $[0,\infty)$.
\end{proof}

\begin{proposition}\label{prop:bl-large-s}
Let $\mathsf{C}$ and $V\sim N(0,1)$ be independent, and define
\[
  T_s:=\mathsf{C}+\kappa sV,
  \qquad s\ge0.
\]
Then $T_s$ has survival function $t\mapsto p_{\mathrm{BL}}(t;s)$.
Moreover, as $s\to\infty$,
\[
  \frac{T_s}{\kappa s}\dto N(0,1).
\]
Equivalently, for every fixed $x\in\R$,
\[
  p_{\mathrm{BL}}(\kappa s x;s)\to \barPh(x).
\]
Consequently, if $q_\alpha(s)$ is the unique solution to $p_{\mathrm{BL}}(q_\alpha(s);s)=\alpha$ for $\alpha\in(0,1/2)$, then we have
\[
  \frac{q_\alpha(s)}{\kappa s}\to z_{1-\alpha}:=\Ph^{-1}(1-\alpha)
\]
and
\[
  q_\alpha(s)=\kappa z_{1-\alpha}s+O_\alpha(1).
\]
\end{proposition}

\begin{proof}
By Proposition~\ref{app-prop:bl-structure}, $T_s\stackrel{d}{=}\mathsf{C}+\kappa sV$, so $R_s:=T_s/(\kappa s)\stackrel{d}{=}V+\varepsilon_s\mathsf{C}$ with $\varepsilon_s:=1/(\kappa s)$. Slutsky's theorem gives $R_s\dto V\sim N(0,1)$ and, since $\Ph$ is continuous, $p_{\mathrm{BL}}(\kappa sx;s)\to\barPh(x)$.

We sharpen this to an $O_\alpha(\varepsilon_s)$ density estimate. Write $f_s^R$ and $F_s^R$ for the density and distribution function of $R_s$. The characteristic function $\phi_{R_s}(u)=\exp(-u^2/2-\varepsilon_s|u|)$ and Fourier inversion give
\[
  \sup_{x\in\R}|f_s^R(x)-\ph(x)|
  \le\frac{1}{2\pi}\int_\R e^{-u^2/2}\bigl|e^{-\varepsilon_s|u|}-1\bigr|du
  \le\frac{\varepsilon_s}{\pi}.
\]
Set $z=\Ph^{-1}(1-\alpha)$ and $x_s=q_\alpha(s)/(\kappa s)$, so $F_s^R(x_s)=1-\alpha=\Ph(z)$. By symmetry $F_s^R(0)=\Ph(0)=1/2$, and integrating the density bound from $0$ gives $|F_s^R(x)-\Ph(x)|\le |x|\varepsilon_s/\pi$ for every $x$. By P\'olya's theorem $\sup_x|F_s^R(x)-\Ph(x)|\to 0$, since $\Ph$ is continuous and strictly increasing, convergence of the $(1-\alpha)$-quantiles yields $x_s\to z$; in particular $x_s\in[z/2,3z/2]$ for $s$ large enough. On this fixed interval $\ph\ge c_\alpha>0$, so $|F_s^R(x_s)-F_s^R(z)|\ge c_\alpha|x_s-z|/2$ once $s$ is large enough that $f_s^R\ge c_\alpha/2$ on $[z/2,3z/2]$. Combining with $|F_s^R(x_s)-F_s^R(z)|\le 3z\varepsilon_s/\pi$ yields $|x_s-z|=O_\alpha(\varepsilon_s)$; multiplying by $\kappa s$ gives $q_\alpha(s)=\kappa z_{1-\alpha}s+O_\alpha(1)$.
\end{proof}

Together, Proposition~\ref{app-prop:bl-structure}, Proposition~\ref{prop:bl-small-s}, and Proposition~\ref{prop:bl-large-s} prove Proposition~\ref{prop:bl-structure}.

\begin{proof}[Proof of Corollary~\ref{cor:bl-c0}]
Set
\[
  D_K:=
  \sup_{u\in[0,1]}
  \left|
    \Pbb \left(p_{\mathrm{BL}}(T_K;s_K)\le u\right)-u
  \right|.
\]
Suppose, for contradiction, that $D_K\not\to0$.
Then there exist $\varepsilon>0$ and a subsequence $(K_j)$ such that
\[
  D_{K_j}\ge \varepsilon
  \qquad\text{for all }j.
\]

Either $(s_{K_j})$ has a bounded subsequence or it has a subsequence tending to $+\infty$.

\paragraph{Case 1: bounded subsequence.}
Passing to a further subsequence if necessary, still denoted $(K_j)$, assume
\[
  \sup_j s_{K_j}<\infty.
\]
Then Theorem~\ref{thm:bluniform} applies along this subsequence, yielding
\[
  D_{K_j}\to 0,
\]
contradicting $D_{K_j}\ge\varepsilon$.

\paragraph{Case 2: divergent subsequence.}
Passing to a further subsequence if necessary, still denoted $(K_j)$, assume
\[
  s_{K_j}\to\infty.
\]
Let
\[
  Y_j:=\frac{T_{K_j}}{\kappa s_{K_j}},
  \qquad
  G_j(x):=\Pbb \left(\frac{T_{s_{K_j}}}{\kappa s_{K_j}}\le x\right),
\]
where $T_s$ denotes the BL family from Proposition~\ref{prop:bl-large-s}.
We first claim that
\[
  Y_j\dto N(0,1).
\]
Fix $v\in\R$.
Because $c_{K_j}\to0$, Corollary~\ref{cor:explicit-stable-exponent} and
Theorem~\ref{thm:centered-stable-limit} yield
\[
  T_{K_j}^{(v)}-b_{K_j}(v)\dto \mathsf{C}(0,1)
\]
under $(V=v)$.
Hence the centered conditional laws are tight, and since $s_{K_j}\to\infty$ we have
\[
  \frac{T_{K_j}^{(v)}-b_{K_j}(v)}{\kappa s_{K_j}}\to 0
  \qquad\text{in probability under }(V=v).
\]
Also, Theorem~\ref{thm:bK-smallc} gives
\[
  \frac{b_{K_j}(v)}{\kappa s_{K_j}}\to v.
\]
Therefore
\[
  \frac{T_{K_j}^{(v)}}{\kappa s_{K_j}}\to v
  \qquad\text{in probability under }(V=v).
\]
Let $h:\R\to\R$ be bounded and continuous.
Then
\[
  \E \left[
    h \left(\frac{T_{K_j}^{(v)}}{\kappa s_{K_j}}\right)
 \mid V=v
  \right]
  \to h(v)
\]
for every fixed $v$, and dominated convergence with respect to the law of $V$ yields
\[
  \E[h(Y_j)]\to \E[h(V)].
\]
Hence $Y_j\dto N(0,1)$.

Also, Proposition~\ref{prop:bl-large-s} gives
\[
  \varepsilon_j:=\sup_{x\in\R}|G_j(x)-\Ph(x)|\to 0.
\]

By the exact survival representation from Proposition~\ref{prop:bl-large-s}, for each deterministic $t$,
\[
  p_{\mathrm{BL}}(t;s_{K_j})
  =
  \Pbb(T_{s_{K_j}}>t)
  =
  1-G_j \left(\frac{t}{\kappa s_{K_j}}\right).
\]
Evaluating this identity at $t=T_{K_j}$ gives the pointwise relation
\[
  p_{\mathrm{BL}}(T_{K_j};s_{K_j})
  =
  1-G_j(Y_j).
\]
With
\[
  U_j:=1-\Ph(Y_j)=\barPh(Y_j),
\]
we have the deterministic bound
\[
  \left|p_{\mathrm{BL}}(T_{K_j};s_{K_j})-U_j\right|
  \le \varepsilon_j.
\]
Since $Y_j\dto N(0,1)$ and $\barPh$ is continuous,
\[
  U_j=\barPh(Y_j)\dto \barPh(V)\sim \mathrm{Unif}(0,1).
\]
By P\'olya's theorem,
\[
  \Delta_j:=
  \sup_{u\in[0,1]}
  \left|\Pbb(U_j\le u)-u\right|
  \to 0.
\]

Fix $u\in[0,1]$.
From $|p_{\mathrm{BL}}(T_{K_j};s_{K_j})-U_j|\le\varepsilon_j$, we get
\[
  \Pbb(U_j\le u-\varepsilon_j)
  \le
  \Pbb \left(p_{\mathrm{BL}}(T_{K_j};s_{K_j})\le u\right)
  \le
  \Pbb(U_j\le u+\varepsilon_j),
\]
where probabilities at arguments outside $[0,1]$ are understood via truncation at $0$ and $1$.
Therefore
\[
  \left|
    \Pbb \left(p_{\mathrm{BL}}(T_{K_j};s_{K_j})\le u\right)-u
  \right|
  \le
  \Delta_j+\varepsilon_j.
\]
Taking the supremum over $u\in[0,1]$ yields
\[
  D_{K_j}\le \Delta_j+\varepsilon_j\to 0,
\]
again contradicting $D_{K_j}\ge\varepsilon$.
This contradiction proves that $D_K\to0$.

For the cutoff statement, note that the proof of Theorem~\ref{thm:bluniform} already showed that the law of $T_K$ is continuous for every fixed $K$. Therefore
\[
  \{p_{\mathrm{BL}}(T_K;s_K)\le \alpha\}
  =
  \{T_K\ge q_\alpha(s_K)\}
\]
almost surely, and hence
\[
  \Pbb \left(T_K>q_\alpha(s_K)\right)
  =
  \Pbb \left(p_{\mathrm{BL}}(T_K;s_K)\le \alpha\right).
\]
Substituting $u=\alpha$ in the already proved uniformity gives
\[
  \Pbb \left(T_K>q_\alpha(s_K)\right)\to \alpha.
\]
\end{proof}

\section{Proofs in Section~\ref{sec:calibration}: broader $c$-scale and plug-in estimators}\label{app:blpositive}

\subsection{Proof of Proposition 4.4}

\begin{proof}[Proof of Proposition~\ref{prop:bl-c-phase}]
\medskip\noindent\textit{Constant-$c$ regime.}
Assume first that $c_K\to c\in(0,\infty)$.
Fix $v\in\R$.
By Theorem~\ref{thm:centered-stable-limit},
\[
  T_K^{(v)}-b_K(v)\dto S_{c,v}
\]
for an explicit $1$-stable law $S_{c,v}$.
In particular, the conditional laws of $T_K^{(v)}-b_K(v)$ are tight, so
\[
  \frac{T_K^{(v)}-b_K(v)}{\log K}\to 0
  \qquad\text{in probability under }(V=v).
\]
Also, Theorem~\ref{thm:bK-general} yields
\[
  \rho_K b_K(v)\to B_c(v).
\]
Since $c_K=\rho_K\log K\to c$,
\[
  \frac{b_K(v)}{\log K}
  =
  \frac{\rho_K b_K(v)}{c_K}
  \to
  \frac{B_c(v)}{c}.
\]
Therefore, conditionally on $V=v$,
\[
  \frac{T_K^{(v)}}{\log K}
  =
  \frac{T_K^{(v)}-b_K(v)}{\log K}
  +
  \frac{b_K(v)}{\log K}
  \to
  \frac{B_c(v)}{c}
\]
in probability.

Now let $h:\R\to\R$ be bounded and continuous.
For each fixed $v$, we have
\[
  \E \left[
    h \left(\frac{T_K^{(v)}}{\log K}\right)
  \right]
  \to
  h \left(\frac{B_c(v)}{c}\right).
\]
Since $|h|\le\|h\|_\infty$, dominated convergence with respect to the law of $V$ gives
\[
  \E \left[
    h \left(\frac{T_K}{\log K}\right)
  \right]
  \to
  \E \left[
    h \left(\frac{B_c(V)}{c}\right)
  \right].
\]
Hence
\[
  \frac{T_K}{\log K}\dto M_c:=\frac{B_c(V)}{c}.
\]

Next, for every $c>0$,
\[
  B_c'(v)
  =
  \frac{2}{\pi}\int_0^{\sqrt{2c}} t^2 e^{-t^2/2}\cosh(vt) dt
  >0,
  \qquad v\in\R.
\]
Thus $B_c$ is continuous and strictly increasing, so the random variable $M_c=B_c(V)/c$ has a continuous distribution.

By Proposition~\ref{prop:bl-large-s},
\[
  \frac{q_\alpha(s_K)}{\log K}
  =
  \frac{q_\alpha(s_K)}{s_K}\cdot \frac{s_K}{\log K}
  \to
  \kappa z_{1-\alpha}\sqrt c,
\]
because $s_K/\log K=\sqrt{c_K}\to\sqrt c$.
Slutsky's theorem therefore yields
\[
  \frac{T_K-q_\alpha(s_K)}{\log K}
  \dto
  M_c-\kappa z_{1-\alpha}\sqrt c.
\]
Since $M_c$ has a continuous law,
\[
  \Pbb \left(T_K>q_\alpha(s_K)\right)
  \to
  \Pbb \left(M_c>\kappa z_{1-\alpha}\sqrt c\right)
  =:\Xi_\alpha(c).
\]

\medskip\noindent\textit{Regime $c_K\to\infty$.}
Assume now that $c_K\to\infty$.
Fix $v\neq0$ and write
\[
  \mu_K(v):=\mu_{\rho_K}(v).
\]
By Proposition~\ref{prop:large-mean},
\[
  \Pbb \left(
    \left|T_K^{(v)}-\mu_K(v)\right|>\frac12|\mu_K(v)|
 \mid V=v
  \right)\to0.
\]
Also, Theorem~\ref{thm:mu-smallrho} gives
\[
  \rho_K \mu_K(v)\to \sqrt{2/\pi} v e^{v^2/2}.
\]
Because Proposition~\ref{prop:bl-large-s} yields
\[
  q_\alpha(s_K)\sim \kappa z_{1-\alpha}s_K
  =\kappa z_{1-\alpha}\sqrt{\rho_K}(\log K)^{3/2},
\]
while Theorem~\ref{thm:mu-smallrho} gives
\[
  |\mu_K(v)|
  \sim
  \frac{\sqrt{2/\pi}}{\rho_K}|v|e^{v^2/2}
  \qquad (v\neq0),
\]
we obtain
\[
  \frac{q_\alpha(s_K)}{|\mu_K(v)|}
  \asymp
  \rho_K^{3/2}(\log K)^{3/2}
  =
  c_K^{3/2}
  \to
  \infty.
\]
Hence, for every fixed $v\neq0$, we have
\[
  q_\alpha(s_K)>2|\mu_K(v)|
\]
for all sufficiently large $K$.
On the event
\[
  \left\{\left|T_K^{(v)}-\mu_K(v)\right|\le\frac12|\mu_K(v)|\right\},
\]
it follows that
\[
  T_K^{(v)}
  \le
  |\mu_K(v)|+\frac12|\mu_K(v)|
  <
  q_\alpha(s_K).
\]
Therefore, for every fixed $v\neq0$,
\[
  \Pbb \left(T_K>q_\alpha(s_K) \mid V=v\right)
  \le
  \Pbb \left(
    \left|T_K^{(v)}-\mu_K(v)\right|>\frac12|\mu_K(v)|
 \mid V=v
  \right)
  \to 0.
\]
The conditional probability is bounded by $1$, and the set $\{V=0\}$ has probability zero under the standard normal law.
Dominated convergence with respect to $V$ gives
\[
  \Pbb \left(T_K>q_\alpha(s_K)\right)
  =
  \E \left[
    \Pbb \left(T_K>q_\alpha(s_K) \mid V\right)
  \right]
  \to 0.
\]
This completes the proof.
\end{proof}

\subsection{Auxiliary continuity lemma for Proposition 4.5}

\begin{lemma}\label{lem:Mc-continuity}
Fix $0<c_-<c_+<\infty$.
For each $c\in[c_-,c_+]$, define
\[
  M_c:=\frac{B_c(V)}{c},
  \qquad
  H_c(x):=\Pbb(M_c\le x).
\]
Then the following hold.
\begin{enumerate}
\item For every $c\in[c_-,c_+]$, the law of $M_c$ has a continuous density $h_c$.
\item There exists a finite constant $C_{c_-,c_+}$ such that
\[
  \sup_{c\in[c_-,c_+]}\|h_c\|_\infty\le C_{c_-,c_+}.
\]
\item If $c_n\to c\in[c_-,c_+]$, then we have
\[
  M_{c_n}\dto M_c
  \qquad\text{and}\qquad
  \sup_{x\in\R}|H_{c_n}(x)-H_c(x)|\to0.
\]
\end{enumerate}
Consequently, for each fixed $\alpha\in(0,1/2)$, the map
\[
  \Xi_\alpha(c):=\Pbb \left(M_c>\kappa z_{1-\alpha}\sqrt c\right)
\]
is continuous on $(0,\infty)$.
\end{lemma}

\begin{proof}
Fix $c\in[c_-,c_+]$ and write
\[
  \psi_c(v):=\frac{B_c(v)}{c}.
\]
For each fixed $v$ and $c>0$, differentiation under the integral sign in \eqref{eq:def-Bc} is justified because the derivative of the integrand is continuous and integrable on $[0,\sqrt{2c}]$. We therefore obtain
\[
  \psi_c'(v)
  =
  \frac{2}{\pi c}
  \int_0^{\sqrt{2c}} t^2 e^{-t^2/2}\cosh(vt) dt.
\]
Hence $\psi_c$ is $C^1$, strictly increasing, and, using $c\le c_+$, $\sqrt{2c}\ge\sqrt{2c_-}$, and $\cosh(vt)\ge1$, satisfies the explicit uniform lower bound
\[
  \psi_c'(v)
  \ge
  m_{c_-,c_+}
  :=
  \frac{2}{\pi c_+}\int_0^{\sqrt{2c_-}} t^2 e^{-t^2/2} dt
  >0
  \qquad\text{for every }c\in[c_-,c_+].
\]
Moreover,
\[
  \lim_{v\to\infty}\psi_c(v)=\infty,
  \qquad
  \lim_{v\to-\infty}\psi_c(v)=-\infty,
\]
because \eqref{eq:def-Bc} integrates $\sinh(vt)$ over a nontrivial interval of positive $t$.
Therefore $\psi_c$ is a $C^1$ bijection from $\R$ onto $\R$, and the change-of-variables formula gives the density
\[
  h_c(x)=\frac{\ph(\psi_c^{-1}(x))}{\psi_c'(\psi_c^{-1}(x))}.
\]
Since $\ph\le(2\pi)^{-1/2}$ and $\psi_c'\ge m_{c_-,c_+}$, we obtain
\[
  \|h_c\|_\infty
  \le
  \frac{1}{\sqrt{2\pi} m_{c_-,c_+}}
\]
uniformly in $c\in[c_-,c_+]$.
This proves (i) and (ii).

Now let $c_n\to c\in[c_-,c_+]$.
For each fixed $v\in\R$, the change of variables $t=\sqrt u r$ in \eqref{eq:def-Bc} gives
\[
  \frac{B_u(v)}{u}
  =
  \frac{2}{\pi}
  \int_0^{\sqrt2}
  r e^{-ur^2/2}\sinh(v\sqrt u r) dr,
  \qquad u>0.
\]
For $u\in[c_-,c_+]$ and $r\in[0,\sqrt2]$, the integrand is continuous in $u$ and bounded in absolute value by
\[
  \frac{2}{\pi} r \sinh \left(|v|\sqrt{c_+} r\right),
\]
which is integrable on $[0,\sqrt2]$.
Dominated convergence therefore gives continuity of $u\mapsto B_u(v)/u$ on $[c_-,c_+]$.
Hence
\[
  \psi_{c_n}(V)\to \psi_c(V)
  \qquad\text{almost surely.}
\]
If $g:\R\to\R$ is bounded and continuous, then we have
\[
  g(\psi_{c_n}(V))\to g(\psi_c(V))
  \qquad\text{almost surely,}
\]
and bounded convergence yields
\[
  \E[g(M_{c_n})]\to \E[g(M_c)].
\]
Thus $M_{c_n}\dto M_c$.

Because $M_c$ has a continuous distribution by (i), Polya's theorem implies
\[
  \sup_{x\in\R}|H_{c_n}(x)-H_c(x)|\to0.
\]
Finally, fix $\alpha\in(0,1/2)$ and write $z=z_{1-\alpha}$.
Then we have
\begin{align*}
  \left|\Xi_\alpha(c_n)-\Xi_\alpha(c)\right|
  &=
  \left|
    H_{c_n} \left(\kappa z\sqrt{c_n}\right)
    -
    H_c \left(\kappa z\sqrt c\right)
  \right| \\
  &\le
  \sup_{x\in\R}|H_{c_n}(x)-H_c(x)|
  +
  C_{c_-,c_+}\kappa |z| \left|\sqrt{c_n}-\sqrt c\right|,
\end{align*}
which tends to zero by (ii).
This proves continuity of $\Xi_\alpha$ on $[c_-,c_+]$.
Since every compact subset of $(0,\infty)$ can be embedded into such an interval, $\Xi_\alpha$ is continuous on $(0,\infty)$.
\end{proof}

\subsection{Proof of Proposition 4.5}

\begin{proof}[Proof of Proposition~\ref{prop:bl-c-current}]
Set
\[
  F_K(x):=\Pbb \left(\frac{T_K}{\log K}\le x\right).
\]
Suppose, for contradiction, that
\[
  \sup_{x\in\R}|F_K(x)-H_{c_K}(x)|
  \not\to 0.
\]
Then there exist $\varepsilon>0$ and a subsequence $K_j$ such that
\[
  \sup_{x\in\R}|F_{K_j}(x)-H_{c_{K_j}}(x)|\ge\varepsilon
\]
for every $j$.
Since $c_{K_j}\in[c_-,c_+]$, compactness gives a further subsequence, not relabelled, with
\[
  c_{K_j}\to c\in[c_-,c_+].
\]
For each fixed $v\in\R$, Theorem~\ref{thm:centered-stable-limit} and Theorem~\ref{thm:bK-general} imply
\[
  \frac{T_{K_j}^{(v)}-b_{K_j}(v)}{\log K_j}\to 0
  \qquad\text{in probability under }(V=v),
\]
and
\[
  \frac{b_{K_j}(v)}{\log K_j}
  =
  \frac{\rho_{K_j} b_{K_j}(v)}{c_{K_j}}
  \to
  \frac{B_c(v)}{c}.
\]
Hence
\[
  \frac{T_{K_j}^{(v)}}{\log K_j}\to \frac{B_c(v)}{c}
  \qquad\text{in probability under }(V=v).
\]
If $h:\R\to\R$ is bounded and continuous, dominated convergence with respect to the law of $V$ yields
\[
  \E\left[h\left(\frac{T_{K_j}}{\log K_j}\right)\right]
  \to
  \E[h(M_c)].
\]
Thus
\[
  \frac{T_{K_j}}{\log K_j}\dto M_c.
\]
Lemma~\ref{lem:Mc-continuity} yields
\[
  \sup_{x\in\R}|H_{c_{K_j}}(x)-H_c(x)|\to0,
\]
while Polya's theorem gives
\[
  \sup_{x\in\R}|F_{K_j}(x)-H_c(x)|\to0
\]
because $M_c$ has a continuous distribution.
Therefore
\[
  \sup_{x\in\R}|F_{K_j}(x)-H_{c_{K_j}}(x)|
  \le
  \sup_{x\in\R}|F_{K_j}(x)-H_c(x)|
  +
  \sup_{x\in\R}|H_{c_{K_j}}(x)-H_c(x)|
  \to0,
\]
a contradiction.
This proves the first claim.

Now fix $\alpha\in(0,1/2)$ and write $z:=z_{1-\alpha}$ and
\[
  x_K:=\frac{q_\alpha(s_K)}{\log K}.
\]
Because $c_K\in[c_-,c_+]$, we have $s_K=\sqrt{c_K} \log K\to\infty$.
Proposition~\ref{prop:bl-large-s} therefore gives
\[
  \frac{q_\alpha(s_K)}{s_K}\to \kappa z,
\]
and hence
\[
  x_K-\kappa z\sqrt{c_K}
  =
  \left(
    \frac{q_\alpha(s_K)}{s_K}-\kappa z
  \right)\sqrt{c_K}
  \to 0.
\]
Using the first part together with the uniform density bound from Lemma~\ref{lem:Mc-continuity}, we obtain
\begin{align*}
  \left|
    \Pbb \left(T_K>q_\alpha(s_K)\right)-\Xi_\alpha(c_K)
  \right|
  &=
  \left|
    \Pbb \left(\frac{T_K}{\log K}>x_K\right)
    -
    \Pbb \left(M_{c_K}>\kappa z\sqrt{c_K}\right)
  \right| \\
  &\le
  \sup_{x\in\R}|F_K(x)-H_{c_K}(x)|
  +
  \left|
    H_{c_K}(x_K)-H_{c_K}(\kappa z\sqrt{c_K})
  \right| \\
  &\le
  \sup_{x\in\R}|F_K(x)-H_{c_K}(x)|
  +
  C_{c_-,c_+} \left|x_K-\kappa z\sqrt{c_K}\right|
  \to 0.
\end{align*}
This proves the size formula.
\end{proof}

\subsection{Proof of Corollary 4.6}

\begin{proof}
Fix $\alpha\in[\barPh(\sqrt{3}),1/2)$ and write
\[
  z:=z_{1-\alpha}=\Ph^{-1}(1-\alpha),
\]
so $0<z\le\sqrt{3}$.
For every $x>0$,
\[
  \frac{\sinh x}{x}
  =
  \sum_{m=0}^\infty \frac{x^{2m}}{(2m+1)!}
  <
  \sum_{m=0}^\infty \frac{(x^2/6)^m}{m!}
  =
  e^{x^2/6},
\]
where the strict inequality uses $(2m+1)!\ge 6^m m!$ for all $m\ge0$ and $(2m+1)!>6^m m!$ for every $m\ge2$.
Hence, for every $t>0$,
\[
  e^{-t^2/2}\sinh(zt)
  <
  zt\exp \left(-\frac{3-z^2}{6}t^2\right)
  \le
  zt.
\]
Let $a:=\sqrt{2c}$.
By the definition of $B_c$,
\begin{align*}
  B_c(z)
  &=
  \frac{2}{\pi}\int_0^a t e^{-t^2/2}\sinh(zt) dt \\
  &<
  \frac{2z}{\pi}\int_0^a t^2 dt
  =
  \frac{2za^3}{3\pi}
  =
  \kappa z c^{3/2}.
\end{align*}

Now define $r_\alpha(c)>0$ as the unique solution of
\[
  B_c \left(r_\alpha(c)\right)=\kappa z c^{3/2}.
\]
To justify existence and uniqueness, first note that $v\mapsto B_c(v)$ is continuous and
\[
  \partial_v B_c(v)
  =
  \frac{2}{\pi}\int_0^{\sqrt{2c}} t^2 e^{-t^2/2}\cosh(vt) dt
  >0,
\]
so it is strictly increasing on $\R$. Moreover,
\[
  B_c(v)
  \ge
  \frac{2}{\pi}\int_{\sqrt{2c}/2}^{\sqrt{2c}} t e^{-t^2/2}\sinh(vt) dt
  \to \infty
  \qquad (v\to\infty),
\]
because the integrand is positive and grows exponentially in $v$ on the fixed interval
$[\sqrt{2c}/2,\sqrt{2c}]$. Since $B_c(z)<\kappa z c^{3/2}$ by the previous display, the intermediate value theorem yields a solution $r_\alpha(c)>z$, and strict monotonicity gives uniqueness.
Consequently, we have
\begin{align*}
  \Xi_\alpha(c)
  &= \Pbb \left(M_c>\kappa z\sqrt{c}\right) \\
  &= \Pbb \left(B_c(V)>\kappa z c^{3/2}\right) \\
  &= \Pbb \left(V>r_\alpha(c)\right) \\
  &= \barPh \left(r_\alpha(c)\right) \\
  &< \barPh(z)
   = \alpha.
\end{align*}
This proves the conservative inequality $\Xi_\alpha(c)<\alpha$.

If $c_K\to0$, then Corollary~\ref{cor:bl-c0} gives
\[
  \Pbb \left(T_K>q_\alpha(s_K)\right)\to \alpha.
\]
Conversely, suppose that $c_K$ does not converge to $0$.
Then there exists a subsequence $K_j$ such that either $c_{K_j}\to c\in(0,\infty)$ or $c_{K_j}\to\infty$.
In the first case, Proposition~\ref{prop:bl-c-phase} yields
\[
  \Pbb \left(T_{K_j}>q_\alpha(s_{K_j})\right)
  \to
  \Xi_\alpha(c)
  <
  \alpha.
\]
In the second case, Proposition~\ref{prop:bl-c-phase} gives
\[
  \Pbb \left(T_{K_j}>q_\alpha(s_{K_j})\right)\to 0.
\]
Thus the full sequence cannot converge to $\alpha$ unless $c_K\to0$.
This proves the exactness criterion at conventional levels.
\end{proof}

\subsection{Plug-in consistency under the null model}

\begin{proposition}\label{prop:plugin}
Under the equicorrelated Gaussian null with possibly $K$-dependent correlation parameter $\rho_K\in[0,1)$, let
\[
  \widehat\rho_K
  :=
  \max \left\{0, 1-\frac{1}{K-1}\sum_{i=1}^K (z_i-\bar z)^2\right\},
  \qquad
  \widehat s_K:=\sqrt{\widehat\rho_K} (\log K)^{3/2},
\]
where $z_i=\Ph^{-1}(1-p_i)=Z_i$ and $\bar z=K^{-1}\sum_{i=1}^K z_i$.
Then we have
\[
  \widehat\rho_K-\rho_K=O_p(K^{-1/2}).
\]
Moreover:
\begin{enumerate}[label=(\roman*)]
\item If $\sup_K s_K<\infty$, then $\widehat s_K-s_K\to0$ in probability and therefore
\[
  \sup_{u\in[0,1]}
  \left|
    \Pbb \left(p_{\mathrm{BL}}(T_K;\widehat s_K)\le u\right)-u
  \right|
  \to 0.
\]
\item If
\[
  \frac{K\rho_K}{(\log K)^3}\to\infty,
\]
then
\[
  |\widehat s_K-s_K|
  =
  O_p\left(\frac{(\log K)^{3/2}}{\sqrt{K\rho_K}}\right)
  =
  o_p(1).
\]
Consequently,
\[
  \sup_{t\in\R}
  \left|
    p_{\mathrm{BL}}(t;\widehat s_K)-p_{\mathrm{BL}}(t;s_K)
  \right|
  =
  o_p(1).
\]
In particular, this condition is satisfied whenever $\rho_K\ge c_0/\log K$ eventually for some $c_0>0$.
\end{enumerate}
\end{proposition}

\begin{proof}
Under the null model we have $z_i=Z_i$, so by the one-factor representation
\[
  z_i=\sqrt{\rho_K} V+\sqrt{1-\rho_K} \varepsilon_i.
\]
Subtracting the sample mean removes the common factor:
\[
  z_i-\bar z=\sqrt{1-\rho_K} (\varepsilon_i-\bar\varepsilon),
  \qquad
  \bar\varepsilon:=\frac1K\sum_{i=1}^K \varepsilon_i.
\]
Hence
\[
  \frac{1}{K-1}\sum_{i=1}^K (z_i-\bar z)^2
  =
  (1-\rho_K) \frac{1}{K-1}\sum_{i=1}^K (\varepsilon_i-\bar\varepsilon)^2.
\]
The Gaussian sample variance satisfies
\[
  \frac{1}{K-1}\sum_{i=1}^K (\varepsilon_i-\bar\varepsilon)^2
  \sim \frac{\chi^2_{K-1}}{K-1}
  =
  1+O_p(K^{-1/2}).
\]
Therefore
\[
  1-\frac{1}{K-1}\sum_{i=1}^K (z_i-\bar z)^2
  =
  \rho_K+O_p(K^{-1/2}).
\]
Because $x\mapsto\max\{0,x\}$ is $1$-Lipschitz and $\rho_K\ge0$, this implies
\[
  \widehat\rho_K-\rho_K=O_p(K^{-1/2}).
\]

Now assume $\sup_K s_K<\infty$. Then we have
\[
  |\widehat s_K-s_K|
  =
  (\log K)^{3/2} |\sqrt{\widehat\rho_K}-\sqrt{\rho_K}|
  \le
  (\log K)^{3/2}\sqrt{|\widehat\rho_K-\rho_K|},
\]
so
\[
  |\widehat s_K-s_K|
  =
  O_p \left((\log K)^{3/2}K^{-1/4}\right)
  =
  o_p(1).
\]
The final uniformity statement in part (i) now follows directly from Theorem~\ref{thm:bluniform}.

For part (ii), assume $K\rho_K/(\log K)^3\to\infty$. Then $\rho_K>0$ eventually, and
\[
  |\sqrt{\widehat\rho_K}-\sqrt{\rho_K}|
  =
  \frac{|\widehat\rho_K-\rho_K|}
       {\sqrt{\widehat\rho_K}+\sqrt{\rho_K}}
  \le
  \frac{|\widehat\rho_K-\rho_K|}{\sqrt{\rho_K}}.
\]
Therefore
\[
  |\widehat s_K-s_K|
  \le
  (\log K)^{3/2}
  \frac{|\widehat\rho_K-\rho_K|}{\sqrt{\rho_K}}
  =
  O_p\left(\frac{(\log K)^{3/2}}{\sqrt{K\rho_K}}\right)
  =
  o_p(1).
\]
Finally, for every $t\in\R$ the same Lipschitz calculation used in the proof of Theorem~\ref{thm:bluniform} gives
\[
  \left|p_{\mathrm{BL}}(t;\widehat s_K)-p_{\mathrm{BL}}(t;s_K)\right|
  \le
  \frac{\kappa}{\pi}\E|V||\widehat s_K-s_K|,
\]
uniformly in $t$, proving the displayed oracle--plug-in approximation.
\end{proof}

\section{Proofs in Section~\ref{sec:power}: power analysis}\label{app:power}

We retain the notation of Section~\ref{sec:power}.

\subsection{Perturbation lemmas}

\begin{lemma}[Local Gaussian perturbations]\label{lem:power-local-hetero}
Let $a_{1,K},\ldots,a_{K,K}$ be deterministic real numbers and let
\[
  Z_{i,K}=a_{i,K}+\sigma_K\varepsilon_i,\qquad i=1,\ldots,K,
\]
where the $\varepsilon_i$'s are independent $N(0,1)$ variables and
$\sigma_K^2=1+\delta_K>0$.  Assume
\[
  \delta_K \log K\to0,\qquad
  \max_{1\le i\le K}|a_{i,K}|\sqrt{\log K}\to0,
\]
and
\begin{equation}\label{eq:local-l1-assumption}
  \frac{(\log K)^{3/2}}{K}\sum_{i=1}^K |a_{i,K}|=O(1),
  \qquad
  \frac{\kappa (\log K)^{3/2}}{K}\sum_{i=1}^K a_{i,K}\to h .
\end{equation}
Then
\[
  \frac{1}{K}\sum_{i=1}^K f(Z_{i,K})\dto \mathsf{C}(h,1).
\]
More precisely, with
\[
  B_K:=\frac{1}{K}\sum_{i=1}^K
  \E\left[
    f(Z_{i,K})\mathbf 1\{|f(Z_{i,K})|\le K\}
  \right],
\]
one has $B_K\to h$ and
\[
  \frac{1}{K}\sum_{i=1}^K f(Z_{i,K})-B_K\dto \mathsf{C}(0,1).
\]
\end{lemma}

\begin{proof}
Put $Y_{i,K}=f(Z_{i,K})/K$ and condition on the deterministic means $a_{i,K}$.

\emph{Tail asymptotics.}
For $u$ in a compact subset of $(0,\infty)$, set $q_{K,u}:=\barPh^{-1}(u/K)$; Mills'
ratio gives $q_{K,u}=\sqrt{2\log K}\{1+o(1)\}$ uniformly. Since
$\max_i|a_{i,K}|\sqrt{\log K}\to0$ and $\delta_K\log K\to0$,
\[
  \tfrac12\Big[q_{K,u}^2-\big((q_{K,u}-a_{i,K})/\sigma_K\big)^2\Big]
  =a_{i,K}q_{K,u}+\tfrac{\delta_Kq_{K,u}^2}{2}+o(1)=o(1),
\]
so $\barPh((q_{K,u}-a_{i,K})/\sigma_K)/\barPh(q_{K,u})\to1$ uniformly in $i$ and
locally uniformly in $u$; the same calculation applied to $-Z_{i,K}$ gives the lower
tail. Transforming $p=\barPh(z)$ through $f(z)=\cot(\pi p)$ exactly as in
Lemma~\ref{lem:score-interval-prob} yields, for $0<r<s<\infty$,
\[
  \sum_{i=1}^K\Pbb\{r<Y_{i,K}\le s\}\to\tfrac1\pi\Big(\tfrac1r-\tfrac1s\Big),
  \qquad
  \sum_{i=1}^K\Pbb\{-s\le Y_{i,K}<-r\}\to\tfrac1\pi\Big(\tfrac1r-\tfrac1s\Big),
\]
together with the uniform bound $\sum_{i=1}^K\Pbb\{|Y_{i,K}|>x\}\le C/x$ for
$K^{-1}\le x\le\eta_0$.

\emph{Stable convergence.}
From this uniform tail bound, the integration-by-parts argument of
Proposition~\ref{prop:small-jumps-vanish} gives the small-jump condition
$\lim_{\eta\downarrow0}\limsup_K\sum_i\E[Y_{i,K}^2\mathbf 1\{|Y_{i,K}|\le\eta\}]=0$. The
triangular-array criterion for infinitely divisible limits then yields
\[
  \sum_{i=1}^K Y_{i,K}-B_K\dto\mathsf{C}(0,1),
  \qquad
  B_K:=\sum_{i=1}^K\E\left[Y_{i,K}\mathbf 1\{|Y_{i,K}|\le1\}\right].
\]

\emph{Centering.}
With $a_K=f^{-1}(K)=\sqrt{2\log K}\{1+o(1)\}$ (Lemma~\ref{lem:inv-asymp}) and $g_{i,K}$
the $N(a_{i,K},\sigma_K^2)$ density, the integral representation of
Lemma~\ref{lem:bK-exact} gives
\[
  \E\left[f(Z_{i,K})\mathbf 1\{|f(Z_{i,K})|\le K\}\right]
  =\int_0^{a_K} f(z)\{g_{i,K}(z)-g_{i,K}(-z)\}dz .
\]
Running the small-$c$ computation of Theorem~\ref{thm:bK-smallc} with mean $a_{i,K}$ in
place of $\sqrt{\rho_K}v$ (that is, $g_{i,K}(z)-g_{i,K}(-z)=2a_{i,K}z\ph(z)\{1+o(1)\}$
and $f(z)\ph(z)=z/\pi\{1+o_M(1)\}$ on $[M,a_K]$), gives, uniformly in $i$,
\[
  \E\left[f(Z_{i,K})\mathbf 1\{|f(Z_{i,K})|\le K\}\right]
  =\frac{2a_{i,K}}{\pi}\int_0^{a_K}z^2dz+o\big(|a_{i,K}|(\log K)^{3/2}\big)+O(|a_{i,K}|).
\]
Averaging over $i$ under the $L^1$-locality assumption
\eqref{eq:local-l1-assumption},
\[
  B_K=\frac{2a_K^3}{3\pi K}\sum_{i=1}^K a_{i,K}+o(1)
  =\frac{\kappa (\log K)^{3/2}}{K}\sum_{i=1}^K a_{i,K}+o(1)\to h .
\]
\end{proof}

\begin{lemma}[Deleting a sparse null set]\label{lem:power-delete-null}
Assume the null one-factor model and
\[
  \rho_K\downarrow0,\qquad \rho_K\log K=O(1).
\]
Condition on $V=v$.  Let $I_K\subset\{1,\ldots,K\}$ be deterministic, or random but
independent of the null noises, and suppose
\[
  \frac{|I_K|\log K}{K}\to0
\]
in probability.  Then
\[
  \frac{1}{K}\sum_{i\in I_K}
  f\left(\sqrt{\rho_K}v+\sqrt{1-\rho_K}\varepsilon_i\right)
  \xrightarrow[]{p}0
\]
conditionally on $V=v$.  Consequently, deleting $I_K$ changes the null CCT statistic
by $o_p(1)$.
\end{lemma}

\begin{proof}
It is enough to prove the claim for deterministic $I_K$.  Indeed, the estimates below
are uniform over all deterministic sets with $n_K=|I_K|$, and the random case follows
by conditioning on $I_K$ on events where $n_K\log K/K$ is small.  Put
\[
  n_K=|I_K|,
  \qquad
  Y_K=\frac{1}{K}f\left(\sqrt{\rho_K}v+\sqrt{1-\rho_K}\varepsilon\right).
\]
We shall prove that the sum of $n_K$ independent copies of $Y_K$ converges to zero in
probability.

The uniform score-tail bound needed here is the bound proved in
Lemma~\ref{lem:uniform-tail-xinv}.  That lemma is stated for convergent $c_K$, but the
proof depends only on an upper bound for $c_K$ and on the fixed value of $v$.  Hence,
under the present assumption $c_K=\rho_K\log K=O(1)$, there is a finite constant $C=C(v)$
such that, for all large $K$,
\[
  K\Pbb\{|Y_K|>x\}\le \frac{C}{x},
  \qquad K^{-1}\le x\le 1 .
\]
In particular, for every fixed $\eta\in(0,1)$,
\[
  n_K\Pbb\{|Y_K|>\eta\}
  \le \frac{C}{\eta}\frac{n_K}{K}
  \to0 .
\]
Thus the contribution of scores with $|Y_K|>\eta$ is absent with probability tending to
one.  We next control the truncated mean at the same cutoff $\eta$.  Since
\[
  K\E\left[Y_K\mathbf 1\{|Y_K|\le1\}\right]=b_K(v)
\]
and $b_K(v)=O(\log K)$ whenever $\rho_K\log K=O(1)$, Theorems
\ref{thm:bK-general}--\ref{thm:bK-smallc} and a subsequence argument give
\[
  n_K\E\left[Y_K\mathbf 1\{|Y_K|\le1\}\right]
  =\frac{n_K}{K}b_K(v)=o(1).
\]
Moreover, the same tail bound gives, for fixed $\eta\in(0,1)$,
\[
\begin{aligned}
  n_K\E\left[
    |Y_K|\mathbf 1\{\eta<|Y_K|\le1\}
  \right]
  &\le
  n_K\eta\Pbb\{|Y_K|>\eta\}
  +
  n_K\int_\eta^1 \Pbb\{|Y_K|>x\}dx       \\
  &\le
  C\frac{n_K}{K}\{1+\log(1/\eta)\}
  =o(1).
\end{aligned}
\]
Consequently
\[
  n_K\E\left[Y_K\mathbf 1\{|Y_K|\le\eta\}\right]=o(1)
\]
for every fixed $\eta\in(0,1)$.

Finally, by the integration-by-parts bound of
Proposition~\ref{prop:small-jumps-vanish} together with the displayed uniform tail bound,
\[
  K\E\left[Y_K^2\mathbf 1\{|Y_K|\le\eta\}\right]\le \frac1K+C\eta,
\]
so that
\[
  n_K\E\left[Y_K^2\mathbf 1\{|Y_K|\le\eta\}\right]
  =\frac{n_K}{K}K\E\left[Y_K^2\mathbf 1\{|Y_K|\le\eta\}\right]
  =o(1).
\]
The variance of the sum of the $n_K$ $\eta$-truncated variables is thus $o(1)$, and
its mean is $o(1)$.  Combining this with the extreme-score bound proves
\[
  \sum_{i\in I_K}
  \frac{1}{K}f\left(\sqrt{\rho_K}v+\sqrt{1-\rho_K}\varepsilon_i\right)
  \xrightarrow[]{p}0
\]
conditionally on $V=v$.  This is the desired statement.
\end{proof}

\begin{lemma}[Null anti-concentration on bounded common-correlation scales]\label{lem:power-null-anticonc}
Assume the null one-factor model and
\[
  \rho_K\downarrow0,\qquad c_K:=\rho_K\log K=O(1).
\]
Let $T_K^{(0)}$ be the null CCT statistic.  Then, for every deterministic real
sequence $d_K$,
\[
  \lim_{\varepsilon\downarrow0}\limsup_{K\to\infty}
  \Pbb_0\{|T_K^{(0)}-d_K|\le\varepsilon\}=0 .
\]
\end{lemma}

\begin{proof}
By the subsequence principle it suffices to bound
$\limsup_K\Pbb_0\{|T_K^{(0)}-d_K|\le\varepsilon\}$ along an arbitrary subsequence; pass to
a further subsequence with $c_K\to c\in[0,\infty)$ and, when $c=0$, with
$s_K\to s\in[0,\infty]$. On each scale the null statistic, normalised by a sequence
$\beta_K$, converges to a non-degenerate limit $L$ with continuous (hence non-atomic)
law:
\begin{itemize}
\item if $c\in(0,\infty)$, then $\beta_K=\log K$ and $T_K^{(0)}/\log K\dto M_c=B_c(V)/c$
  (Proposition~\ref{prop:bl-c-phase}), with $M_c$ non-atomic by
  Lemma~\ref{lem:Mc-continuity};
\item if $c=0$ and $s_K\to s<\infty$, then $\beta_K=1$ and
  $T_K^{(0)}\dto\mathsf{C}+\kappa sV$ (Corollary~\ref{cor:raw-smallc}, integrated over
  $V$), whose density $x\mapsto\E[\pi^{-1}\{1+(x-\kappa sV)^2\}^{-1}]$ is continuous;
\item if $c=0$ and $s_K\to\infty$, then $\beta_K=s_K$ and $T_K^{(0)}/s_K\dto\kappa V$
  by Theorems~\ref{thm:centered-stable-limit} and~\ref{thm:bK-smallc}, since
  $T_K^{(0)}-b_K(V)=O_p(1)$ and $b_K(V)/s_K\dto\kappa V$.
\end{itemize}
Pass to a further subsequence with $d_K/\beta_K\to d\in[-\infty,\infty]$. If $d$ is
finite, then for all large $K$ one has
$\{|T_K^{(0)}-d_K|\le\varepsilon\}\subseteq\{|T_K^{(0)}/\beta_K-d|\le\eta\}$ when
$\beta_K\to\infty$, and
$\{|T_K^{(0)}-d_K|\le\varepsilon\}\subseteq\{|T_K^{(0)}-d|\le\varepsilon+\eta\}$ when
$\beta_K\equiv1$; the portmanteau theorem and non-atomicity of $L$ give
\[
  \limsup_K\Pbb_0\{|T_K^{(0)}-d_K|\le\varepsilon\}\le\Pbb\{|L-d|\le\varepsilon+\eta\},
\]
which vanishes as $\eta\downarrow0$ and then $\varepsilon\downarrow0$. If
$d=\pm\infty$, tightness of $T_K^{(0)}/\beta_K$ gives the bound directly. Since every
subsequence admits such a further subsequence, the conclusion holds along the original
sequence.
\end{proof}

\subsection{Proof of Theorem~\ref{thm:power-local}}

\begin{proof}[Proof of Theorem~\ref{thm:power-local}]
Condition on $V=v$.  The conditional means are
\[
  a_{i,K}(v)=\mu_{i,K}+\sqrt{\rho_K}v,\qquad i=1,\ldots,K,
\]
and the conditional variance is $1-\rho_K$.  Since $s_K\to s<\infty$,
$\rho_K\log K=s_K^2/(\log K)^2\to0$.  Also
\[
  \max_i |a_{i,K}(v)|\sqrt{\log K}\to0
\]
for each fixed $v$, and
\[
  \frac{(\log K)^{3/2}}{K}\sum_{i=1}^K |a_{i,K}(v)|=O(1).
\]
Finally,
\[
  \frac{\kappa (\log K)^{3/2}}{K}\sum_{i=1}^K a_{i,K}(v)
  =
  \frac{\kappa (\log K)^{3/2}}{K}\sum_{i=1}^K\mu_{i,K}
  +
  \kappa s_Kv
  \to h+\kappa sv .
\]
Lemma~\ref{lem:power-local-hetero} gives the conditional Cauchy limit.  Conditional
survival probabilities at $t_\alpha$ and $q_\alpha(s_K)\to q_\alpha(s)$ therefore
converge to the displayed Cauchy survival probabilities.  They are bounded by one, so
dominated convergence over $V$ gives the unconditional limits.

The gap formula is the difference between two Cauchy survival functions.  Since the
Cauchy density is bounded by $1/\pi$,
\[
  0\le
  \Pi^{\rm raw}_\alpha(h,s)-\Pi^{\rm BL}_\alpha(h,s)
  \le \frac{q_\alpha(s)-t_\alpha}{\pi}.
\]
The expansion
\[
  q_\alpha(s)-t_\alpha
  =
  \frac{\kappa^2t_\alpha}{1+t_\alpha^2}s^2+O_\alpha(s^4)
\]
from Proposition~\ref{prop:bl-small-s} proves the $O_\alpha(s^2)$ bound.  If
$\rho_K(\log K)^3\to0$, then $s_K\to0$, $q_\alpha(s_K)\to t_\alpha$, and the two limits
coincide.
\end{proof}

\subsection{Sparse signals}

For $x\in\R$, write $x_-:=\max\{-x,0\}$.

\begin{lemma}[Sparse signal score estimates]\label{lem:power-sparse-score}
Assume $\rho_K\log K=O(1)$, fix $v\in\R$, and set
\[
  Z_K^{\rm sig}=\sqrt{2r\log K}+\sqrt{\rho_K}v+\sqrt{1-\rho_K}\varepsilon .
\]
For every fixed $\gamma\ge0$,
\[
  \Pbb\{f(Z_K^{\rm sig})>K^{1+\gamma}\mid V=v\}
  =
  K^{-(\sqrt{1+\gamma}-\sqrt r)^2+o(1)} ,
\]
and
\[
  \Pbb\{f(Z_K^{\rm sig})<-K^{1+\gamma}\mid V=v\}
  =
  K^{-(\sqrt{1+\gamma}+\sqrt r)^2+o(1)} .
\]
The same two exponent estimates remain valid if $K^{1+\gamma}$ is replaced by
$xK^{1+\gamma}$, for any fixed $x\in(0,\infty)$.
For every fixed $\eta>0$,
\[
  \E\left[
    f(Z_K^{\rm sig})_+
    \mathbf 1\{f(Z_K^{\rm sig})\le K\eta\}
    \middle|V=v
  \right]
  \le K^{2\sqrt r-r+o(1)} ,
\]
and for every fixed $\gamma\ge0$,
\[
  \E\left[
    f(Z_K^{\rm sig})_-
    \mathbf 1\{f(Z_K^{\rm sig})_-\le K^{1+\gamma}\}
    \middle|V=v
  \right]
  \le K^{o(1)} .
\]
The $o(1)$ terms are locally uniform for $v$ in compact sets.
\end{lemma}

\begin{proof}
Let $a_{K,\gamma}:=f^{-1}(K^{1+\gamma})$.  Lemma~\ref{lem:inv-asymp} gives
\[
  a_{K,\gamma}=\sqrt{2(1+\gamma)\log K}\{1+o(1)\}.
\]
More generally, for each fixed $x\in(0,\infty)$,
\[
  f^{-1}(xK^{1+\gamma})
  =
  \sqrt{2(1+\gamma)\log K}\{1+o(1)\},
\]
so fixed multiplicative constants in the score threshold do not change the polynomial
exponents below.  Since $\rho_K\log K=O(1)$, Mills' ratio yields
\[
  \Pbb\{f(Z_K^{\rm sig})>K^{1+\gamma}\mid V=v\}
  =
  \barPh\left(
    \frac{a_{K,\gamma}-\sqrt{2r\log K}-\sqrt{\rho_K}v}{\sqrt{1-\rho_K}}
  \right)
  =
  K^{-(\sqrt{1+\gamma}-\sqrt r)^2+o(1)}.
\]
The negative-tail display follows in the same way from the threshold
$-a_{K,\gamma}$:
\[
  \Pbb\{Z_K^{\rm sig}<-a_{K,\gamma}\mid V=v\}
  =
  K^{-(\sqrt{1+\gamma}+\sqrt r)^2+o(1)}.
\]

For the positive truncated moment, let $a_{K,\eta}:=f^{-1}(K\eta)$.  By
Lemma~\ref{lem:f-growth} and the Gaussian density bound, uniformly for fixed $v$,
\[
\begin{aligned}
  \E\left[
    f(Z_K^{\rm sig})_+
    \mathbf 1\{f(Z_K^{\rm sig})\le K\eta\}
    \middle|V=v
  \right]
  &\le
  K^{o(1)}
  \int_0^{a_{K,\eta}}
      (1+z)
      \exp\{\sqrt{2r\log K}z-r\log K\}dz      \\
  &\le K^{2\sqrt r-r+o(1)}.
\end{aligned}
\]
For the negative truncated moment, substitute $u=-z\ge0$.  On
$0\le u\le a_{K,\gamma}$,
\[
  f(-u)_- = f(u)
  \le C(1+u)e^{u^2/2},
\]
and the density of $Z_K^{\rm sig}$ at $-u$ contributes the factor
\[
  \exp\left\{
    -\frac{(u+\sqrt{2r\log K}+\sqrt{\rho_K}v)^2}{2(1-\rho_K)}
  \right\}.
\]
Hence
\[
\begin{aligned}
  \E\left[
    f(Z_K^{\rm sig})_-
    \mathbf 1\{f(Z_K^{\rm sig})_-\le K^{1+\gamma}\}
    \middle|V=v
  \right]
  &\le
  K^{o(1)}
  \int_0^{a_{K,\gamma}}
      (1+u)\exp\{-\sqrt{2r\log K}u\}du       \\
  &\le K^{o(1)}.
\end{aligned}
\]
This proves all four estimates.
\end{proof}

\begin{proof}[Proof of Theorem~\ref{thm:power-sparse}]
For $\delta>0$, define the high-probability event
\[
  A_{K,\delta}:=
  \left\{
    K^{1-\beta-\delta}\le |S_K|\le K^{1-\beta+\delta}
  \right\}.
\]
The assumption $|S_K|=K^{1-\beta+o_p(1)}$ means that
$\Pbb(A_{K,\delta})\to1$ for every fixed $\delta>0$.  Throughout the proof we first
condition on $V=v$ and on a realised signal set $S_K$ satisfying the relevant event
$A_{K,\delta}$.  The conditional estimates are then integrated over $V$ and
$S_K$ at the end of each case.

First suppose $r>r_{\rm max}(\beta)$.  Choose $\gamma>0$ so small that
\[
  1-\beta-(\sqrt{1+\gamma}-\sqrt r)^2>0 .
\]
Then choose $\delta>0$ small enough that
\[
  a_*:=
  1-\beta-\delta-(\sqrt{1+\gamma}-\sqrt r)^2>0,
  \qquad
  \delta<\beta .
\]
Let $N_{K,\gamma}$ be the number of signal coordinates satisfying
$f(Z_i)>K^{1+\gamma}$.  By Lemma~\ref{lem:power-sparse-score}, conditionally on
$V=v$ and on $S_K\in A_{K,\delta}$,
\[
  \E[N_{K,\gamma}\mid V=v,S_K]
  =
  |S_K|K^{-(\sqrt{1+\gamma}-\sqrt r)^2+o(1)}
  \ge K^{a_*+o(1)}\to\infty .
\]
Since $N_{K,\gamma}$ is binomial under this conditioning,
\[
  \Pbb(N_{K,\gamma}=0\mid V=v,S_K)
  \le
  \{\E[N_{K,\gamma}\mid V=v,S_K]\}^{-1}
  \to0 .
\]
Hence, with probability tending to one, the positive signal contribution to $T_K$ is
at least $K^\gamma$.

We now rule out cancellation.  The full null statistic is $O_p(\log K)$ under
$\rho_K\log K=O(1)$, by Theorem~\ref{thm:centered-stable-limit} together with
Theorems~\ref{thm:bK-general}--\ref{thm:bK-smallc}; if $\rho_K\log K$ does not converge,
apply these results along arbitrary convergent subsequences of the bounded sequence
$\rho_K\log K$.  Since $|S_K|\log K/K\le K^{-\beta+\delta}\log K\to0$,
Lemma~\ref{lem:power-delete-null} also shows that replacing the full null statistic by
the non-signal null part changes it by $o_p(1)$.  Thus the non-signal part of the
alternative statistic is $O_p(\log K)$.

For the signal coordinates, Lemma~\ref{lem:power-sparse-score} gives
\[
  \Pbb\{\exists i\in S_K:\ f(Z_i)<-K^{1+\gamma}\mid V=v,S_K\}
  \le
  K^{1-\beta+\delta-(\sqrt{1+\gamma}+\sqrt r)^2+o(1)}
  \to0
\]
after decreasing $\delta$, if necessary.  On the complement of this event, the total
negative signal contribution is bounded by its truncated part.  Its conditional
expectation is at most
\[
  \frac{|S_K|}{K} K^{o(1)}
  \le K^{-\beta+\delta+o(1)}
  \to0
\]
by the final estimate of Lemma~\ref{lem:power-sparse-score}.  Markov's inequality shows
that the negative signal contribution is $o_p(1)$.  Consequently,
\[
  T_K\ge K^\gamma-O_p(\log K)-o_p(1)\to+\infty
\]
in conditional probability.  Since $\max\{\log d_K,0\}=o(\log K)$, we have $d_K=o(K^\gamma)$,
and hence
\[
  \Pbb(T_K>d_K\mid V=v,S_K)\to1
\]
for every fixed $v$ and every realised $S_K\in A_{K,\delta}$.  The conditional
probabilities are bounded by one and $\Pbb(A_{K,\delta})\to1$.  Integrating over
$V$ and $S_K$ gives
\[
  \Pbb_{\rm alt}(T_K>d_K)\to1 .
\]

Now suppose $r<r_{\rm max}(\beta)$.  Choose $\delta>0$ so small that
\[
  1-\beta+\delta-(1-\sqrt r)^2<0,
  \qquad
  -\beta+\delta+2\sqrt r-r<0,
  \qquad
  \delta<\beta .
\]
Again condition on $V=v$ and on $S_K\in A_{K,\delta}$.  Fix $\eta>0$.  By the
fixed-multiplicative-constant version of the tail estimates in
Lemma~\ref{lem:power-sparse-score},
\[
  \Pbb\{\exists i\in S_K:\ f(Z_i)>K\eta\mid V=v,S_K\}
  \le
  K^{1-\beta+\delta-(1-\sqrt r)^2+o(1)}
  \to0,
\]
and similarly
\[
  \Pbb\{\exists i\in S_K:\ f(Z_i)<-K\eta\mid V=v,S_K\}
  \le
  K^{1-\beta+\delta-(1+\sqrt r)^2+o(1)}
  \to0 .
\]
On the complement of these extreme events, the positive and negative signal parts are
truncated at $K\eta$.  Lemma~\ref{lem:power-sparse-score} and Markov's inequality give
\[
\begin{aligned}
  \E\left[
     \frac{1}{K}\sum_{i\in S_K}
       f(Z_i)_+\mathbf 1\{f(Z_i)\le K\eta\}
     \middle| V=v,S_K
  \right]
  &\le
  K^{-\beta+\delta+2\sqrt r-r+o(1)}
  \to0 ,
\end{aligned}
\]
and
\[
  \E\left[
     \frac{1}{K}\sum_{i\in S_K}
       f(Z_i)_-\mathbf 1\{f(Z_i)_-\le K\eta\}
     \middle| V=v,S_K
  \right]
  \le
  K^{-\beta+\delta+o(1)}
  \to0 .
\]
Thus the whole signal contribution is $o_p(1)$, conditionally on $V=v$ and
$S_K\in A_{K,\delta}$.

Let $T_K^{(0)}$ be the null statistic obtained from the same noises after setting all
signals to zero.  The difference between $T_K$ and $T_K^{(0)}$ is the actual signal
contribution, just shown to be $o_p(1)$, minus the null contribution over the signal
locations.  Since $|S_K|\log K/K\le K^{-\beta+\delta}\log K\to0$,
Lemma~\ref{lem:power-delete-null} shows that this deleted null contribution is also
$o_p(1)$.  Hence
\[
  T_K=T_K^{(0)}+o_p(1)
\]
conditionally on $V=v$ and $S_K\in A_{K,\delta}$, and therefore unconditionally
because $\Pbb(A_{K,\delta})\to1$.

It remains only to translate this $o_p(1)$ equivalence into rejection probabilities.
For every $\varepsilon>0$,
\[
  \bigl|
    \Pbb_{\rm alt}(T_K>d_K)-\Pbb_0(T_K^{(0)}>d_K)
  \bigr|
  \le
  \Pbb_0\{|T_K^{(0)}-d_K|\le\varepsilon\}
  +
  \Pbb_{\rm alt}\{|T_K-T_K^{(0)}|>\varepsilon\}.
\]
Taking $\limsup_K$ and then $\varepsilon\downarrow0$, Lemma~\ref{lem:power-null-anticonc}
and $T_K-T_K^{(0)}=o_p(1)$ give
\[
  \Pbb_{\rm alt}(T_K>d_K)-\Pbb_0(T_K^{(0)}>d_K)\to0 .
\]

The raw and BL consequences follow by taking $d_K=t_\alpha$ and
$d_K=q_\alpha(s_K)$ and applying the null results in
Corollary~\ref{cor:raw-size-function}, Corollary~\ref{cor:phasediagram}, and
Corollary~\ref{cor:bl-c0}.
\end{proof}

\subsection{Dense Gaussian random effects}

\begin{lemma}[Gaussian random effects]\label{lem:power-random-effects}
Assume
\[
  \tau_K^2\log K\to w\in[0,\infty),\qquad
  \rho_K\log K\to c\in[0,\infty).
\]
Conditionally on $V=v$, after integrating out the random effects, let
\[
  m_K=\sqrt{\rho_K}v,\qquad
  \sigma_{K,w}^2=1-\rho_K+\tau_K^2,
\]
and let $Z_K\sim N(m_K,\sigma_{K,w}^2)$.  Define
\[
  b_{K,w}(v):=
  \E\left[f(Z_K)\mathbf 1\{|f(Z_K)|\le K\}\right].
\]
Let $\Lambda_v^{(w)}:=e^w\Lambda_v$, where $\Lambda_v$ is the null Lévy measure in
\eqref{eq:def-nu-v}, and let $S_{c,v}^{(w)}$ be the infinitely divisible law with
characteristic function
\[
  \E e^{itS_{c,v}^{(w)}}
  =
  \exp\left\{
    \int_{\R\setminus\{0\}}
      \bigl(e^{itx}-1-itx\mathbf 1\{|x|\le1\}\bigr)
      \Lambda_v^{(w)}(dx)
  \right\}.
\]
Then, conditionally on $V=v$,
\[
  T_K-b_{K,w}(v)\dto S_{c,v}^{(w)} .
\]
If $c\in(0,\infty)$, then
\[
  \frac{b_{K,w}(v)}{\log K}\to
  D_{c,w}(v):=
  \frac{2}{\pi}\int_0^{\sqrt2}
       y\exp\left\{\frac{(w-c)y^2}{2}\right\}
       \sinh(\sqrt cvy)dy .
\]
If instead $\rho_K\log K\to0$ and
\[
  s_K=\sqrt{\rho_K}(\log K)^{3/2}\to s<\infty,
\]
then
\[
  b_{K,w}(v)\to \kappa_wsv,\qquad
  \kappa_w:=\frac{2}{\pi}\int_0^{\sqrt2} y^2e^{wy^2/2}dy .
\]
In particular, $\kappa_0=\kappa$, and for $c=0$,
\[
  S_{0,v}^{(w)}\sim\mathsf{C}(0,e^w).
\]
\end{lemma}

\begin{proof}
We first identify the tail measure.  The conditional law $Z_K\mid V=v\sim N(m_K,\sigma_{K,w}^2)$ differs from the null marginal only through the extra variance $\tau_K^2$.  With $q_{K,u}:=\barPh^{-1}(u/K)=\sqrt{2\log K}\{1+o(1)\}$, locally uniformly for $u$ in compact subsets of $(0,\infty)$,
\[
  \frac{1}{2}\left[
    q_{K,u}^2-\left(\frac{q_{K,u}-m_K}{\sigma_{K,w}}\right)^2
  \right]
  =
  (\tau_K^2-\rho_K)\log K+\sqrt{2\rho_K\log K}v+o(1)
  \to
  w-c+\sqrt{2c}v ,
\]
so that, by Mills' ratio, both score-tail intensities are multiplied by the constant factor $e^w$ relative to their null values $\lambda^\pm(v)=e^{-c\pm\sqrt{2c}v}$.  The remaining steps are those of the null analysis with $1-\rho_K$ replaced throughout by $\sigma_{K,w}^2$: the interval limits of Lemma~\ref{lem:score-interval-prob} now produce the measure $\Lambda_v^{(w)}=e^w\Lambda_v$, the uniform tail bound of Lemma~\ref{lem:uniform-tail-xinv} and the small-jump estimate of Proposition~\ref{prop:small-jumps-vanish} hold verbatim along the present sequence, and the triangular-array argument in the proof of Theorem~\ref{thm:centered-stable-limit}, centred by $b_{K,w}(v)=K\E[Y_K\mathbf 1\{|Y_K|\le1\}\mid V=v]$ with $Y_K=f(Z_K)/K$, gives
\[
  T_K-b_{K,w}(v)\dto S_{c,v}^{(w)} .
\]

It remains to compute $b_{K,w}(v)$.  Let $a_K=f^{-1}(K)$, so
$a_K/\sqrt{\log K}\to\sqrt2$.  If $g_{K,w}$ is the density of
$N(m_K,\sigma_{K,w}^2)$, then
\[
  b_{K,w}(v)
  =
  \int_0^{a_K} f(z)\{g_{K,w}(z)-g_{K,w}(-z)\}dz,
\]
and
\[
  g_{K,w}(z)-g_{K,w}(-z)
  =
  \frac{2}{\sigma_{K,w}\sqrt{2\pi}}
  \exp\left\{-\frac{z^2+m_K^2}{2\sigma_{K,w}^2}\right\}
  \sinh\left(\frac{m_Kz}{\sigma_{K,w}^2}\right).
\]

Fix $M\ge1$.  The bounded region $0\le z\le M$ contributes $O(|m_K|)=o(\log K)$ when
$c>0$, and $o(1)$ in the boundary-layer case $s_K=O(1)$.  On
$M\le z\le a_K$, Lemma~\ref{lem:f-growth} gives
\[
  f(z)=\sqrt{2/\pi}z e^{z^2/2}\{1+r_M(z)\},
  \qquad \sup_{z\ge M}|r_M(z)|\to0\quad(M\to\infty).
\]

First assume $c\in(0,\infty)$.  Put $z=\sqrt{\log K}y$.  After dividing by $\log K$, the
integrand on $M\le z\le a_K$ is
\[
  \frac{2}{\pi}
  y
  \exp\left\{
    \frac{\log Ky^2}{2}-\frac{\log Ky^2}{2\sigma_{K,w}^2}
  \right\}
  \sinh\left(\frac{m_K\sqrt{\log K}y}{\sigma_{K,w}^2}\right)
  \{1+o_M(1)\}.
\]
Here
\[
  \frac{\log Ky^2}{2}-\frac{\log Ky^2}{2\sigma_{K,w}^2}
  \to \frac{(w-c)y^2}{2},
  \qquad
  \frac{m_K\sqrt{\log K}y}{\sigma_{K,w}^2}\to \sqrt cvy .
\]
The integrands are dominated on $0\le y\le\sqrt2+o(1)$ by an integrable function of the
form $C_v y e^{C y^2+C_v y}$, while the omitted interval $0\le z\le M$ is negligible
after division by $\log K$.  Letting first $K\to\infty$ and then $M\to\infty$, dominated
convergence gives
\[
  \frac{b_{K,w}(v)}{\log K}
  \to
  \frac{2}{\pi}\int_0^{\sqrt2}
       y\exp\left\{\frac{(w-c)y^2}{2}\right\}
       \sinh(\sqrt cvy)dy .
\]

Now assume $\rho_K\log K\to0$ and $s_K\to s<\infty$.  Again put
$z=\sqrt{\log K}y$, but do not divide by $\log K$.  Since
\[
  \frac{m_K\sqrt{\log K}y}{\sigma_{K,w}^2}
  =
  \frac{s_Kvy}{\log K}\{1+o(1)\},
\]
we have, uniformly for $0\le y\le\sqrt2+o(1)$,
\[
  \log K\sinh\left(\frac{m_K\sqrt{\log K}y}{\sigma_{K,w}^2}\right)
  \to s v y.
\]
Also
\[
  \frac{\log Ky^2}{2}-\frac{\log Ky^2}{2\sigma_{K,w}^2}
  \to \frac{wy^2}{2}.
\]
The same domination and bounded-region argument gives
\[
  b_{K,w}(v)
  \to
  \frac{2sv}{\pi}\int_0^{\sqrt2} y^2e^{wy^2/2}dy
  =
  \kappa_wsv .
\]
When $c=0$, $\lambda^+(v)=\lambda^-(v)=1$, so the limiting Lévy measure is symmetric
with density $e^w/(\pi x^2)$ on both half-lines.  The corresponding characteristic
exponent is $-e^w|t|$, i.e. $S_{0,v}^{(w)}\sim\mathsf{C}(0,e^w)$.
\end{proof}

\begin{proof}[Proof of Theorem~\ref{thm:power-random-effects}]
Part (i) follows from Lemma~\ref{lem:power-random-effects}.  On the boundary-layer scale,
\[
  b_{K,w}(v)\to\kappa_wsv,\qquad
  S_{0,v}^{(w)}\sim\mathsf{C}(0,e^w).
\]
Thus the conditional limit is $\mathsf{C}(\kappa_wsv,e^w)$.  Taking conditional
survival probabilities at $t_\alpha$ and at $q_\alpha(s_K)\to q_\alpha(s)$, and then
integrating over $V$, gives the two power limits.  The gap formula is the difference of
two Cauchy survival functions.  Since the density of a Cauchy law with scale $e^w$ is
bounded by $1/(\pi e^w)$,
\[
  0\le
  \Pi^{\rm raw,G}_\alpha(w,s)-\Pi^{\rm BL,G}_\alpha(w,s)
  \le \frac{q_\alpha(s)-t_\alpha}{\pi e^w}
  =O_\alpha(s^2)
\]
by Proposition~\ref{prop:bl-small-s}.  If $\rho_K(\log K)^3\to0$, then $s_K\to0$, and both
limits reduce to $\bar F_{e^w}(t_\alpha)$.

For part (ii), Lemma~\ref{lem:power-random-effects} gives, conditionally on $V=v$,
\[
  \frac{T_K}{\log K}=D_{c,w}(v)+o_p(1),
\]
because the centred stable fluctuation is $O_p(1)$.  The function $D_{c,w}$ is
continuous, odd, and strictly increasing in $v$, with $D_{c,w}(v)>0$ if and only if
$v>0$.  Therefore
\[
  \Pbb(T_K>t_\alpha\mid V=v)\to\mathbf 1\{v>0\},
\]
and integration over $V\sim N(0,1)$ gives the raw limit $1/2$.  Proposition
\ref{prop:bl-large-s} gives
\[
  q_\alpha(s_K)=\kappa z_{1-\alpha}s_K+O_\alpha(1)
  =
  \kappa z_{1-\alpha}\sqrt c\log K+o(\log K).
\]
Hence
\[
  \Pbb(T_K>q_\alpha(s_K)\mid V=v)
  \to
  \mathbf 1\{D_{c,w}(v)>\kappa z_{1-\alpha}\sqrt c\}.
\]
The boundary set has probability zero, because $D_{c,w}$ is continuous and strictly
increasing.  Dominated convergence gives the BL limit.  Monotonicity in $w$ for
$v>0$ follows by differentiating the positive integrand with respect to $w$.
\end{proof}

\end{document}